\documentclass{amsbook}[12pt]

\usepackage{amsmath, endnotes, amsthm, enumerate, amsfonts, mathrsfs,xypic,pxfonts,rotating,color,comment, url}
\xyoption{all}

\newtheorem{definition}{Definition}
\newtheorem{theorem}{Theorem}
\newtheorem{proposition}{Proposition}
\newtheorem{corollary}{Corollary}
\newtheorem{lemma}{Lemma}

\title{Notes on symmetric spaces}
\author{Jonathan Holland}
\author{Bogdan Ion}

\newcommand{\ad}{\operatorname{ad}}
\newcommand{\Ad}{\operatorname{Ad}}

\newcommand{\Int}{\operatorname{Int}}
\newcommand{\tr}{\operatorname{tr}}
\newcommand{\im}{\operatorname{im}}

\newcommand{\inj}{\operatorname{inj}}
\newcommand{\End}{\operatorname{End}}
\newcommand{\Lie}{\operatorname{Lie}}
\newcommand{\Exp}{\operatorname{Exp}}

\newcommand{\Stab}{\operatorname{Stab}}

\renewcommand{\div}{\operatorname{div}}

\newcommand{\onebar}{\rule[2.5pt]{12pt}{.5pt}}
\newcommand{\twobar}{\rule[3.5pt]{12pt}{.5pt}\hspace{-12pt}\rule[1.5pt]{12pt}{.5pt}}
\newcommand{\threebar}{\rule[0.5pt]{12pt}{.5pt}\hspace{-12pt}\rule[2.5pt]{12pt}{.5pt}\hspace{-12pt}\rule[4.5pt]{12pt}{.5pt}}
\newcommand{\leftbars}{\twobar\hspace{-8pt}\langle\hspace{3pt}}
\newcommand{\rightbars}{\twobar\hspace{-7pt}\rangle\hspace{3pt}}

\newcommand{\rightthreebars}{\threebar\hspace{-7pt}\rangle\hspace{3pt}}

\newcommand{\commentx}[1]{\textcolor{red}{#1}}
\renewcommand{\commentx}[1]{}

\newcommand{\osarray}[2]{\overset{\displaystyle{\overset{\displaystyle{#1}}{#2}}}{\tfrac{}{}}}

\newcommand{\dynkinexample}{
$$
\def\objectstyle{\displaystyle}
\xymatrix{
&&&&&\bullet\\
\osarray{a}{\bullet}\ar@{-}[r]&\bullet\ar@{-}[r]&\bullet\ar@{..}[r]&\bullet\ar@{-}[r]&\bullet\ar@{-}[ur]\ar@{-}[dr]&\\
&&&&&\bullet
}
$$
}

\makeindex 

\begin{document}
\maketitle 
\tableofcontents
\nocite{*}

\chapter*{Preface}

These are notes prepared from a graduate lecture course by the second author on symmetric spaces at the University of Pittsburgh during the Fall of 2010.  We make no claim to originality, either in the nature of the results or in exposition.  In both capacities, we owe a large debt to the notes of Borel \cite{Borel}, as well as to Kobayashi and Nomizu \cite{KN1}, \cite{KN2} for the earlier chapters, and Helgason \cite{Helgason} for the later chapters.

Several changes were made during the preparation of these notes from the material as it was presented in the course.  First, many proofs have been corrected or streamlined.  Whereever possible, results have been more carefully stated.  Some exercises have been added to the text.  In addition, we have tried to provide references to further reading.  We have also devoted an entire chapter to a naive treatment of examples in an effort to illustrate some of the general features of symmetric spaces.

Most of the material that we have drawn upon is quite standard, but some of it is difficult to find.  The earlier chapters draw primarily from Borel \cite{Borel} and Kobayashi and Nomizu \cite{KN1}.  Chapter \ref{CompactSemisimple} draws from Borel \cite{Borel} and Humphreys \cite{HumphreysReflectionGroups}, although there are many more detailed accounts of compact Lie groups; see Knapp \cite{Knapp} for a thorough treatment.  The use of Vogan diagrams in Chapter \ref{Classification} instead of the traditional Satake diagrams seems to be due to Knapp \cite{Knapp}.

\vspace{24pt}

At present, these notes are a first draft---still very much a work in progress.  As of this draft, there are a number of notable omissions:
\begin{itemize}
\item Certain portions of the text have not been optimally organized.  For instance, we have interspersed discussion of non-compact orthogonal symmetric Lie algebras in the chapter on compact orthogonal symmetric Lie algebras.  At the Lie algebra level, the two notions are dual, but doing it this way avoided certain awkwardness in discussing the ``maximally compact'' (as opposed to ``maximally split'') Cartan subalgebra.
\item Furthermore, the lectures themselves scrupulously avoided the familiar theorem--proof paradigm.  We have attempted in these notes to organize things more around results, though.  However, in some sections it was not clear exactly how to do this, and in these sections the text often becomes more obscure.
\item We have not yet prepared suitable figures, although figures were presented in the lecture, and would shed a great deal of light on certain examples.
\item We have omitted the discussion of Satake diagrams in favor of Vogan diagrams, although Satake diagrams were presented in one of the lectures.  We do not have plans to add this content, since it can be easily found in Helgason, and does not add any substantial results to those we have already obtained.
\item The final chapter on analysis on symmetric spaces is not yet completed.  The lectures presented a very skeletal outline of the theory of spherical functions that we should like to include here.  In the mean time, the expository articles \cite{ZonalSphericalFunction}, \cite{Plancherel} cover approximately the same content, although from a slightly different viewpoint.
\end{itemize}

\begin{flushright}
Jonathan Holland\\
Bogdan Ion\\
University of Pittsburgh\\
December, 2010
\end{flushright}

\chapter{Affine connections and transformations}
\section{Preliminaries}
This section sets the notational conventions and basic facts of differential geometry that will be essential to subsequent investigations.  We shall omit proofs of, and even precise citations for, these facts.  Two standard references for this material are Michor \cite{Michor} and Kol\'{a}\v{r}, Michor, Slov\'{a}k \cite{KMS}.

\subsection{Tensors}\index{Tensor}
Let $M$ be a smooth manifold of dimension $n$.  We denote by $\mathscr{C}^\infty$ the sheaf of smooth functions on $M$.  A tangent vector at a point $p\in M$ is an $\mathbb{R}$-linear derivation $v_p:\mathscr{C}^\infty_p\to \mathbb{R}$.  The vector space of tangent vectors at $p$ is denoted $TM_p$, and the disjoint union over all $p\in M$, denoted $TM$, is the {\em tangent bundle} and carries a natural vector bundle structure.  The dual bundle $TM^*$ is the {\em cotangent vector}.  Tensors on $M$ are formed by taking tensor products of copies of $TM$ with copies of $TM^*$.  Denote by $T^p_qM$ the bundle of tensors
$$T^p_qM = \underbrace{TM\otimes\cdots\otimes TM}_p\otimes\underbrace{TM^*\otimes\cdots\otimes TM^*}_q.$$
Tensor fields are sections of the tensor bundle $T^p_qM$, and here are denoted by $\mathscr{T}^p_qM$.  The space of vector fields is thus $\mathscr{T}^1_0M$ and the space of differential forms is $\mathscr{T}_1^0M$.  It is worth remarking that these are all sheaves, and in fact $\mathscr{T}^1_0$ is the sheaf of derivations of $\mathscr{C}^\infty$ to itself.  The resulting sheaf of $\mathscr{C}^\infty$-modules is locally free of rank $n$, and so this shows that $TM$ is actually a vector bundle.

\subsection{Lie derivatives}\index{Lie derivative}
If $\phi:M\to N$ is a differentiable map, then for each $p\in M$, the pushforward $\phi_{*,p}:TM_p\to TN_{\phi(p)}$ is defined by $(\phi_{*,p}X_p)(f) = X_p(f\circ\phi)$ for $f\in\mathscr{C}^\infty_{\phi(p)}$.  The pushforward is also denoted by $d\phi_p$.  The pullback $\phi^* : TN_{\phi(p)}^* \to TM_p^*$ is the transpose of $\phi_{*,p}$.  If $\phi$ is a diffeomorphism, then we defined $\phi_*$ on one-forms by $\phi_* = (\phi^{-1})^*$ and $\phi^*$ on vectors by $\phi^*=(\phi^{-1})_*$.  When $\phi$ is a diffeomorphism, these operators extend in a unique manner to isomorphisms of the full tensor algebra.

Let $X\in\mathscr{T}^1_0M$.  An integral curve of $X$ through a point $p$ is a differentiable curve $\gamma:(-\epsilon,\epsilon)\to M$ such that, for all $f\in\mathscr{C}^\infty(M)$, $(Xf)\circ\gamma=\frac{d}{d t}(f\circ\gamma)$.  Taking $f$ in turn to be each of $n$ functionally independent coordinates, this defines a first-order differential equation in the coordinates of $\gamma$, so that through each point $p\in M$, there exists an integral curve defined on a small enough interval.  For each $p\in M$, there is an open neighborhood $U\subset M$ containing $p$ and $\epsilon>0$ such that each point $q\in U$ has an integral curve $\gamma_q$ through it defined for all $t\in (-\epsilon,\epsilon)$.  The map $(q,t)\mapsto \gamma_q(t) : U\times (-\epsilon,\epsilon)\to M$ is smooth, by smooth dependence of solutions on initial conditions.  This mapping is the {\em local flow} of the vector field $X$, and we shall denote the local flow by $(q,t)\mapsto \exp(tX)(q)$.\index{Exponential map!Local flow}\index{Flow of a vector field}  Shrinking $U$ and choosing a smaller $\epsilon$ if necessary, $\exp$ is a diffeomorphism of $U$ onto its image for all $t\in (-\epsilon,\epsilon)$.  Indeed, $\exp^{-1}(tX) = \exp(-tX)$ is smooth on $U$.\footnote{We use $\exp$ to denote the local flow and $\Exp$ to denote the exponential map associated to an affine connection later.}

The Lie derivative of a tensor field $T$ along a vector field $X\in \mathscr{T}^1_0M$ is defined by
$$\mathscr{L}_X T = \lim_{t\to 0} \frac{\exp(tX)^*T - T}{t}.$$
For instance, it follows from the definition of the exponential map that $\mathscr{L}_Xf=Xf$ for a function $f\in\mathscr{C}^\infty(M)$.  The Lie derivative along $X$ is an $\mathbb{R}$-linear derivation of $\mathscr{C}^\infty$-modules.  In particular,
$$\mathscr{L}_X (S\otimes T) = \mathscr{L}_X S\otimes T + S\otimes \mathscr{L}_X T.$$
Note also that the Lie derivative commutes with tensor contraction, since by definition the pullback commutes with tensor contraction.  

If $X,Y\in\mathscr{T}^1_0M$, then the Lie bracket of $X$ and $Y$ is the commutator of derivations:
$$[X,Y]f = XYf-YXf.$$
With this operation, $\mathscr{T}^1_0M$ becomes a Lie algebra.  The Lie bracket is related to the Lie derivative by $\mathscr{L}_XY=[X,Y]$.  A proof of this is sketched below using differential forms.

\subsection{Differential forms}\index{Differential form}
Let $\wedge TM^*$ be the exterior algebra of the cotangent bundle, and let $\Omega$ denote the sheaf of sections of $\wedge TM^*$.  This is a graded algebra whose graded components $\Omega^k$ are the spaces of $k$-forms.  Let $Der(\Omega)$ be the space of graded derivations on $\Omega$, equipped with the supercommutator
$$[D,D'] = DD' - (-1)^{\deg D\deg D'}D'D.$$
Then $Der(\Omega)$ is a Lie superalgebra graded by degree.

A Lie superalgebra is a Lie algebra in the category of super vector spaces.  A super vector space is a vector space with a $\mathbb{Z}^2$ grading $V=V_0\oplus V_1$.  The category of super vector spaces is just like the category of vector spaces with a $\mathbb{Z}_2$ grading except that the natural braiding isomorphisms $\tau: V\otimes W\cong W\otimes V$ switch sign if both tensor factors have odd degree.  Thus a Lie superalgebra is a vector space $L=L_0\oplus L_1$ together with a bilinear bracket $[-,-]$ such that
\begin{itemize}
\item $[x_i,y_j]\in L_{i+j(2)}$ for $x_i\in L_i, y_j\in L_j$, $i,j=0,1$.
\item $[x,y] = -(-1)^{\deg x\deg y}[y,x]$
\item $[[x,y],z] + (-1)^{\deg x(\deg y+\deg z)}[[y,z],x] + (-1)^{\deg z(\deg x+\deg y)}[[z,x],y] = 0$. (Note this implies that $\ad(x)$ is itself a derivation of degree $\deg x$).
\end{itemize}
Ordinary Lie algebras are Lie superalgebras concentrated in even degree.

There is a unique derivation of degree 1, $d:\Omega\to\Omega$, characterized by the properties
\begin{itemize}
\item If $f\in\Omega^0=\mathscr{C}^\infty$, then $df$ is the ordinary differential of $f$: $df(X) = Xf$
\item $d(\alpha\wedge\beta)=d\alpha\wedge\beta + (-1)^{\deg\alpha}\alpha\wedge d\beta$
\item $d^2=\tfrac12[d,d]=0$
\end{itemize}
Since the action of $d$ on functions commutes with pullbacks, and $d$ is then characterized by algebraic properties that are also invariant under pullback, it follows easily that $d$ commutes with the pullback along smooth mappings.  Hence also $[d,\mathscr{L}_X]=0$.

For $v\in TM$, there is a derivation of degree $-1$, $i_V:\wedge TM^*\to \wedge TM^*$, characterized by
\begin{itemize}
\item $i_vf = 0$ for $f\in\wedge^0TM^*$
\item $i_v\alpha = \alpha(v)$ for $\alpha\in\wedge^1TM^*=TM^*$
\item $i_v(\alpha\wedge\beta) = i_v\alpha\wedge\beta + (-1)^{\deg\alpha}\alpha\wedge i_v\beta$.
\end{itemize}
Then, acting on forms, $\mathscr{L}_X = [i_X,d]$.  Indeed, this is true on $\Omega^0$ since $[i_X,d]f=i_xdf=Xf=\mathscr{L}_Xf$.  It is also true on $d\Omega^0$, since
$$[i_X,d]df=di_xdf =  d(Xf) = d\mathscr{L}_Xf=\mathscr{L}_Xdf.$$
Now, $\Omega$ is generated as an algebra by $\Omega^0$ and $d\Omega^0$, and thus any derivation is determined by its action on $\Omega^0$ and $d\Omega^0$.  Since $\mathscr{L}_X$ agree on these generators, $\mathscr{L}_X=[i_X,d]$.  The same method, together with the fact that Lie differentiation commutes with tensor contraction, shows that $i_{\mathscr{L}_XY}=[\mathscr{L}_X,i_Y]$.

Let us now prove that $[X,Y]=\mathscr{L}_XY$:
\begin{align*}
[X,Y]f &=XYf-YXf = i_Xdi_Ydf - i_Ydi_Xdf\\
&=[i_x,d]i_Ydf - i_y[d,i_X]df\\
&=[\mathscr{L}_X,i_Y]df = i_{\mathscr{L}_XY}df = (\mathscr{L}_XY)f.
\end{align*}

The operators $i_X,d,\mathscr{L}_X, X\in\mathscr{T}^1_0M$ generate a subalgebra of the Lie superalgebra $Der(\Omega)$ which is determined by {\em Cartan's identities}:\index{Cartan's identities}
\begin{itemize}
\item $[\mathscr{L}_X,\mathscr{L}_Y] = \mathscr{L}_{[X,Y]}$
\item $[\mathscr{L}_X,i_Y] = i_{[X,Y]}$
\item $[i_X,i_Y]=0$
\item $[d,i_X]=\mathscr{L}_X$
\item $[d,\mathscr{L}_X]=0$
\item $[d,d]=0$
\end{itemize}

This is a special case of a Weil algebra.\index{Weil algebra}  Abstractly, if $\mathfrak{g}$ is a Lie algebra, then we form the space
$$\mathfrak{W} = \underbrace{\mathfrak{g}}_{\text{degree $0$}}\  \oplus \quad\underbrace{\left( \mathfrak{g} \oplus \mathbb{R}d\right)}_{\text{degree $1$}}$$
together with a bracket defined as follows.  Let $\phi_0,\phi_1:\mathfrak{g}\to \mathfrak{W}$ be the inclusions of $\mathfrak{g}$ into the even and odd parts.  Then, for $X,Y\in\mathfrak{g}$, define the bracket on $\mathfrak{W}$ by
\begin{itemize}
\item $[\phi_0X,\phi_0Y_0]=\phi_0[X,Y]$
\item $[\phi_0X,\phi_1Y]=\phi_1[X,Y]$
\item $[\phi_1X,\phi_1Y]=0$
\item $[d,\phi_1X]=\phi_0X$
\item $[d,\phi_0]=0$
\item $[d,d]=0$
\end{itemize}
In the special case when $\mathfrak{g}=\mathscr{T}^1_0M$ is the Lie algebra of vector fields on $M$ and $\phi_0X=\mathscr{L}_X$, $\phi_1X=i_X$, and $d$ is the exterior derivative, the Cartan identities are precisely the relations for a Weil algebra.

The full algebra of derivations is much more difficult to describe.  The operation $X\mapsto i_X$ from $\mathscr{T}^1_0M\to Der_{-1}(\Omega)$ can be extended to an operator $\mathscr{T}^1_0M\otimes\Omega^k\to Der_{k-1}\Omega$ denoted by $K\mapsto i_K$.  One can then define the Lie derivative along $K$ by $\mathscr{L}_K=[i_K,d]$.  Then any graded derivation in $Der_k(\Omega)$ can be uniquely decomposed as $i_L + \mathscr{L}_K$ for some $L\in \mathscr{T}^1_0M\otimes\Omega^{k+1}$ and $K\in \mathscr{T}^1_0M\otimes\Omega^k$.\footnote{See  Kol\'{a}\v{r}, Michor, Slov\'{a}k, Theorem 8.3.}

\section{Affine connections}\index{Connection!affine}
An affine connection $\nabla$ on a smooth manifold $M$ is an $\mathbb{R}$-bilinear mapping
$$\nabla : T^1_0M\times\mathscr{T}^1_0M \to \mathscr{T}^1_0M,$$
typically written $\nabla_XY$ for $X\in T^1_0M$ and  $Y\in\mathscr{T}^1_0M$, such that for a fixed $X$, $\nabla_X$ is a derivation of $\mathscr{C}^\infty$-modules via the rule:
$$\nabla_X(fY) = (Xf)Y + f\nabla_XY.$$
The operator $\nabla_X$ extends uniquely to a derivation on the full tensor algebra, also denoted by $\nabla_X$.  It is degree preserving (i.e., degree zero).  The resulting derivation agrees with $X$ on functions and commutes with tensor contractions.\footnote{In fact, every derivation of the tensor algebra into the tensor algebra at $x$ that preserves degree and commutes with tensor contractions is of the form $\nabla_X + S$ for some $X\in T_xM$ and $S\in\operatorname{End}TM_x$.  See Kostant, B. {\em Holonomy and the Lie algebra in infinitesimal motions of a Riemannian maniold}, Trans. Amer. Math. Soc. 80 (1955), 528--542.}

Explicit formulas for the extended connection are therefore available using duality.  Thus if $\alpha\in\mathscr{T}^0_1M$ is a one-form, then $\nabla_X\alpha$ is also a one-form given by
$$(\nabla_X\alpha)(Y) = X(\alpha(Y)) - \alpha(\nabla_XY).$$
Similar expressions hold for higher degree tensors. This defines a derivation with respect to the product $\otimes$, but also with respect to the wedge product $\wedge$ under the natural inclusion of the exterior algebra of differential forms into the tensor algebra.  With this identification, the following analog of Cartan's main identity holds:
$$i_{\nabla_XY} = [\nabla_X,i_Y].$$

If $T\in\mathscr{T}^p_qM$, denote by $\nabla T$ the {\em covariant derivative}\index{Connection!covariant derivative} of $T$: the section of $\mathscr{T}^p_{q+1}M$ defined by double duality
$$\nabla T(X\otimes S) = (\nabla_XT)(S)$$
for all $S\in\mathscr{T}_p^qM$.

\subsection{Connection coefficients}\index{Frame}\index{Connection!coefficients}
Let $X_1,\dots,X_n$ be a {\em local frame} on $M$: a basis of sections of $\mathscr{T}^1_0U$ for some open subset $U\subset M$.  It is sometimes convenient to introduce the {\em connection coefficients} $\Gamma_{ij}^k$ relative to this local basis, via\footnote{The Einstein summation convention is implied here and elsewhere indices are used.}
$$\nabla_{X_i}X_j = \Gamma_{ij}^k X_k.$$
If $V = v^iX_i, W=w^iX_j\in \mathscr{T}^1_0U$, then the covariant derivative of $V$ along $W$ can be expressed in terms of the connection coefficients via
$$\nabla_W V = w^i\nabla_{X_i}(v^jX_j) = w^i\left( X_i(v^k) + \Gamma_{ij}^kv^j\right)X_k.$$

A related formalism is that of the {\em connection one-form}\index{Connection!one-form} relative to the frame.  If $X\in\mathscr{T}^1_0U$, then
$$\nabla_XX_i = \theta_i^k(X)X_k$$
where $\theta_i^j$ is a one-form on $U$.  Equivalently, $\nabla X_i = \theta_i^k\otimes X_k.$ In terms of the connection coefficients, $\theta_i^k = \Gamma_{ij}^k\omega^j$ where $\omega^j\in\mathscr{T}^0_1U$ is the dual coframe to $X_i$, defined by $\omega^j(X_i) = \delta_i^j.$

\subsection{Parallel transport}
Let $\gamma:[0,1]\to M$ be a continuously differentiable immersed curve in $M$ with $\gamma(0)=p, \gamma(1)=q\in M$.  A vector field along $\gamma$ is a (continuously differentiable) association of a vector $X_t\in TM_{\gamma(t)}$ to each $t\in [0,1]$ (that is, $X$ is a section of the pullback bundle $\gamma^{-1}TM$).  Given a vector field $X$ along $\gamma$, for each $t_0\in [0,1]$, there is an open interval $I\ni t_0$ such that $X|_I$ can be extended to a vector field $\widetilde{X}$ on an open subset of $M$ that includes $\gamma(I)$.  We then define $\nabla_{\dot{\gamma}(t)}X_t = (\nabla_{\dot{\gamma}(t)}\widetilde{X})_{\gamma(t)}$.  Note that if $f$ is a function vanishing on $\gamma(I)$ and $Y$ is any vector field in a neighborhood of $\gamma(I)$, then $(\nabla_{\dot{\gamma}(t)}(fY))_{\gamma(t)}=\frac{d}{dt}f(\gamma(t))\,Y_{\gamma(t)} + f(\gamma(t))\,(\nabla_{\dot{\gamma}(t)}Y)_{\gamma(t)}$ which vanishes on $I$.  Since any vector field that vanishes along $I$ can be expressed as a linear combination of vectors of this form, $\nabla_{\cdot{\gamma}(t)}X_t$ does not depend on how $X$ was extended off of $\gamma(I)$.  In a local frame, if $V=v^iX_i$, then
$$\nabla_{\dot{\gamma}(t)}V_t = \left(\frac{d}{dt}v^k(\gamma(t)) + \Gamma_{ij}^j(\gamma(t))\dot{\gamma}^i(t)v^k(\gamma(t))\right)X_{k,\gamma(t)}.$$
Since the right-hand side does not involve extending $V$, this verifies in an alternative way that the operation on the left-hand side is well-defined.

If $X_p\in TM_p$ is given, then the parallel transport of $X_p$ along $\gamma$ is the solution of the first-order ode\index{Parallel transport}
\begin{align*}
\nabla_{\dot{\gamma}(t)} X_t &= 0 \\
X_0 &= X_p\qquad\text{(the given initial value)}
\end{align*}
We then define
$$\pi_\gamma X_p = X_1.$$
This transports the vector $X_p\in TM_p$ to a vector $\pi_\gamma X_p \in TM_q$.

The parallel transport obeys the properties:
\begin{itemize}
\item The parallel transport is invariant under reparametrizations of the curve: if $s:[0,1]\to [0,1]$ is a continuously differentiable function with $s'>0$, then $\pi_\gamma = \pi_{\gamma\circ s}$.
\item If $\gamma$ and $\mu$ are two curves and $\gamma(1)=\mu(0)$, then $\pi_{\gamma\cup\mu} = \pi_\gamma\circ\pi_\mu$.
\item $\pi_\gamma$ is invertible, and $\pi_\gamma^{-1} = \pi_{-\gamma}$ where $-\gamma(t):=\gamma(1-t)$ is the curve obtained by following $\gamma$ in reverse.
\end{itemize}
Although we have only defined parallel transport along $C^1$ curves, we can define it along piecewise $C^1$ curves by decomposing such a curve $\gamma=\gamma_1\cup\gamma_2\cup\cdots\cup\gamma_k$ as the compositum of $C^1$ curves, and then defining $\pi_\gamma = \pi_{\gamma_1}\circ\pi_{\gamma_2}\circ\cdots\circ\pi_{\gamma_k}$.  The second property above ensures that this is compatible with the notion of parallel transport already defined.

From parallel transport, the affine connection can be recovered.  Let $X\in \mathscr{T}^1_0M$ and suppose that $Y=\dot{\gamma}(0)\in TM_p$.  Let $\gamma_s$ be the curve $\gamma_s(t) = \gamma(t/s)$ for $s\in (0,1)$. Then
$$(\nabla_Y X)_p = \lim_{s\to 0^+} \frac{\pi_{\gamma_s}^{-1}(X_{\gamma_s(1)}) - X_p}{s}.$$

{\em Holonomy.}\index{Holonomy}  If $\gamma$ is a closed curve, then the parallel transport $\pi_\gamma\in \operatorname{GL}(TM_p)$.  This map is called the {\em holonomy} around $\gamma$.  As $\gamma$ varies over all closed curves, this defines a subgroup of $\operatorname{GL}(TM_p)$, the {\em holonomy group} $\Psi_p$.  It is an analytic subgroup of $\operatorname{GL}(TM_p)$.\footnote{See Theorem 4.2 in Kobayashi and Nomizu, {\em Foundations of differential geometry, Vol. I}, Wiley--Interscience, 1963.}  Moreover, the connected component of the identity, $\Psi_{p,e}$, is the subgroup generated by the holonomy around closed curves that are homotopic to the constant curve: $\gamma\sim p$.

\subsection{Geodesics and the exponential map}\index{Geodesic}\index{Exponential map!Riemannian}
A geodesic is a $C^1$ curve whose tangent vector is parallel along itself.  Equivalently, $\gamma$ is a geodesic if and only if
$$\nabla_{\dot{\gamma}}\dot{\gamma} = 0.$$
More precisely, this characterizes the geodesics with a preferred class of parametrizations known as affine parametrizations, which are determined up to an affine transformation of time: $t\mapsto at+b$.  By the existence and uniqueness theorem for solutions of ordinary differential equations, there exists a unique geodesic through $p$ with prescribed initial velocity $X_p$.  This geodesic can be continued for sufficiently small time $t$.

For $p\in M$ and $X\in TM_p$, let $\gamma(t) = \Exp_p(tX)$ be the geodesic through $p$ with initial velocity $\gamma'(0)=X$.  This is well-defined in a neighborhood of $X\in TM_p$ for $|t|<\epsilon$.  Since $TM_p$ is locally compact, there is an open $U\subset TM_p$ containing the origin such that $\Exp_p(X)$ exists for all $X\in U$.

We now calculate $d\Exp_{p,0}$.  Let $X\in TM_p$ be a vector.  Since $TM_p$ is a vector space, we can think of $X$ as an element of the tangent space at the origin $T(TM_p)_0$.  If $F\in\mathscr{C}^\infty(TM_p)$, then $XF(0)=\left.\frac{d}{dt}F(tX)\right|_{t=0}$.  Hence, for $f\in\mathscr{C}^\infty M_p$,
$$d\Exp_{p,0}(X)f = X(f\circ \Exp_p)(0) = \left.\frac{d}{dt}f(\Exp_p(tX))\right|_{t=0} = Xf(p)$$
since $\gamma(t) = \Exp_p(tX)$ satisfies $\dot{\gamma}(0)=X$.  Thus $d\Exp_{p,0} = \operatorname{Id}$.  By the inverse function theorem, $\Exp_p$ is a diffeomorphism of a possibly smaller open neighborhood $U$ of the origin in $TM_p$ onto an open neighborhood of $p\in M$.

\subsection{Torsion and curvature}
The two tensorial invariants of an affine connection are its curvature tensor $R\in \Omega^2\otimes \mathscr{T}^1_1M$ and torsion $T\in \Omega^2\otimes\mathscr{T}^1M$ defined by duality as follows.  Let $X_p,Y_p,Z_p\in TM_p$ be vectors at $p$ and extend them to vector fields $X,Y,Z$ in a neighborhood of $p$.  Put\index{Torsion}\index{Curvature}
\begin{align*}
R_p(X_p,Y_p)Z_p &= ([\nabla_X,\nabla_Y]Z - \nabla_{[X,Y]}Z)_p\\
T_p(X_p,Y_p) &= (\nabla_XY- \nabla_YX - [X,Y])_p.
\end{align*}
In each of these, the right-hand side is multilinear in each of the vector fields, and it is readily verified that if one of the vector fields vanishes at $p$, then the whole quantity is zero at $p$.  Therefore these do not actually depend on how $X,Y,Z$ are extended away from $p$.

The curvature measures the failure of a vector $Z_p$ to return to its original position when it is parallel transported over an infinitely small loop.  In fact, if $X$ and $Y$ commute, then $\pi_{\exp(-tY/v)}\pi_{\exp(-tX/u)}\pi_{\exp(tY/v)}\pi_{\exp(tX/u)}Z_p$ is the result of parallel-transporting the vector $Z_p$ over the parallelogram whose sides are integral curves for the vector fields $X$ and $Y$, followed for parameter time $u$ and $v$, respectively.  Then the curvature is recovered from the parallel transport via
$$R_p(X_p,Y_p)Z_p = \lim_{u,v\to 0} \frac{Z_p-\pi_{\exp(-tY/v)}\pi_{\exp(-tX/u)}\pi_{\exp(tY/v)}\pi_{\exp(tX/u)}Z_p}{uv}.$$

Since curvature is generated by the parallel transport over infinitely small loops, naturally there is a correspondence between the holonomy and the curvature.  By taking a ``lasso''  from $x$ to each point $y\in M$, with an infinitely small loop at the end, the following Ambrose--Singer holonomy theorem is reasonable:
\begin{theorem}\index{Ambrose--Singer theorem}
$\operatorname{Lie}(\Psi_p)$ is generated as a Lie algebra by the set of all $\pi_\gamma^{-1} R(X_{\gamma(1)},Y_{\gamma(1)})\in\mathfrak{gl}(TM_p)$ such that $\gamma:[0,1]\to M$ is a piecewise $C^1$ curves with $\gamma(0)=p$, and $X_{\gamma(1)},Y_{\gamma(1)}\in TM_{\gamma(1)}$.
\end{theorem}

\section{Affine diffeomorphisms}
Let $(M,\nabla)$, $(M',\nabla')$ be two manifolds equipped with affine connections.  A diffeomorphism $\phi:M\to M'$ is called an {\em affine diffeomorphism} if\index{Affine diffeomorphism}
$$\nabla'_{\phi_*X}\phi_*Y = \phi_*(\nabla_XY)$$
for all $X,Y\in\mathscr{T}^1_0M$.  A local affine diffeomorphism is an affine diffeomorphism from an open subset of $M$ to an open subset of $M'$.

If $\phi$ is a local affine diffeomorphism, then $\phi$ maps geodesics to geodesics (together with the affine parameter).  Hence $(\phi\circ\Exp_p)(X) = \Exp_{\phi(p)}(\phi_{*,p}X)$.  In particular,
\begin{proposition}\label{faithful}
Let $M$ be connected and $f,g:M\to M'$ two local affine diffeomorphisms (near all points of $M$).  If $f_{*,p}=g_{*,p}$ at some $p\in M$, then $f=g$.
\end{proposition}

Let $Z\in \mathscr{T}^1_0M$.  Then we call $Z$ an {\em infinitesimal affine diffeomorphism}\index{Affine diffeomorphism!infinitesimal} if $\exp(tZ)$ is a local affine diffeomorphism of $M$ for all sufficiently small $t$.  For such a $Z$, by definition
$$\exp(tZ)_*\nabla_X Y = \nabla_{\exp(tZ)_*X} (\exp(tZ)_*Y)$$
for all $Y$.  Differentiating at $t=0$ gives
$$\mathscr{L}_Z\nabla_XY = \nabla_{\mathscr{L}_ZX} Y + \nabla_X(\mathscr{L}_ZY)$$
or,
\begin{equation}\label{InfinitesimalAffine}
[\mathscr{L}_Z,\nabla_X] = \nabla_{[Z,X]}
\end{equation}
for every vector field $X\in\mathscr{T}^1_0M$.  Conversely, if $Z$ is a vector field such that \eqref{InfinitesimalAffine} holds for all vector fields $X,Y$, then the local flow of $Z$ is a local affine diffeomorphism, and so $Z$ itself is an infinitesimal affine diffeomorphism.

The space of infinitesimal affine transformations is a Lie algebra:
\begin{proposition}
If $X,Y$ are infinitesimal affine transformations, then $[X,Y]$ is an infinitesimal affine transformation.
\end{proposition}

\begin{proof}
Let $K\in\mathscr{T}^1_0M$.  Then
\begin{align*}
[\mathscr{L}_{[X,Y]},\nabla_K] &= [[\mathscr{L}_X,\mathscr{L}_Y],\nabla_K]\\
&= [\mathscr{L}_X,[\mathscr{L}_Y,\nabla_K]] - [\mathscr{L}_Y,[\mathscr{L}_X,\nabla_K]] \qquad\text{by \eqref{InfinitesimalAffine}}\\
&= [\mathscr{L}_X,[\nabla_{[Y,K]}]] - [\mathscr{L}_Y,[\nabla_{[X,K]}]] \\
&= \nabla_{[X,[Y,K]] - [Y,[X,K]]} \qquad\text{by \eqref{InfinitesimalAffine} again}\\
&= \nabla_{[[X,Y],K]}.
\end{align*}
as required.
\end{proof}

Let $o\in M$ and $\mathfrak{g}_o$ be the space of germs of infinitesimal affine transformations at $o$.  Let
$$\mathfrak{k}_o = \{ Z\in \mathfrak{g}_o \mid Z_o=0\}$$
be the {\em isotropy algebra} of $o$.\index{Isotropy algebra $\mathfrak{k}_o$}  Now the kernel of $\mathscr{L}_{Z,o} : \mathscr{T}^1_0M_o\to TM_o$ contains all vector fields that vanish at $o$, for if $f(o)=0$, and $X\in\mathscr{T}^1_0M$, then, then $[Z,fX]_o = (Z_of)X_o + f(o)[Z,X]_o=0$ since $f(o)=0$ and $Z_o=0$.  Thus $\mathscr{L}_Z$ factors through the quotient and defines a map $\mathscr{L}_{Z,o} : TM_o\to TM_o$.  That is,$\mathscr{L}_{Z,o} \in \mathfrak{gl}(TM_o).$  This defines a representation $\mathfrak{k}_o\to \mathfrak{gl}(TM_o)$.  The representation is faithful by Proposition \ref{faithful}.  In particular $\mathfrak{k}_o$ is finite-dimensional.

By Proposition \ref{faithful}, a local affine diffeomorphism is determined by its differential at a point.  The following theorem answers the question of {\em which} linear isomorphisms $\psi:TM_x \to T\overline{M}_y$ are the differentials of affine transformations.  Its proof will occupy the next subsection.

\begin{theorem}\label{characterizationofaffine}
Let $M,\overline{M}$ be two affine manifolds, $x\in M$, $y\in\overline{M}$.  Let $\psi:TM_x\to T\overline{M}_y$ be a given linear isomorphism, and define $\Psi=\Exp_y\circ\psi\circ\Exp_x^{-1} : U_x\xrightarrow{\sim} U_y$, where $U_x$ is a star-shaped neigbhorhood of the origin in $TM_x$ small enough that $\Psi$ is a diffeomorphism.  Then $\Psi$ is an affine diffeomorphism if and only if $\Psi^*\overline{R} = R$ and $\Psi^*\overline{T}=T$.
\end{theorem}


\subsection{Proof of Theorem \ref{characterizationofaffine}}
Our proof uses Cartan's structure equations.  Let $V_x=\Exp_x(U_x)\subset M$ and $V_y=\Exp_y(U_y)\subset\overline{M}$.  Let $X_{i,x}$ be a basis of $TM_x$, and extend these by parallel translation along geodesics through $x$ to a local frame $X_i$ on $V_x$.  (Geodesics through the center of a normal neighborhood are called {\em radial}.) Note for later use that $\Psi_*X_i$ is also a local frame on $V_y$ that is parallel along radial geodesics as well.

Let $\omega^i$ be the dual forms to $X_i$ and $\theta_i^j$ the connection forms of $\nabla$ in the frame:
$$\nabla X_i = \theta_i^j\otimes X_j.$$
Let $T^i_{jk}=\omega^i(T(X_i,X_j))$ and $R^\ell_{kij}=\omega^\ell(R(X_i,X_j)X_k)$ be the components of the torsion and curvature in the frame and define
$$T^i = \frac{1}{2}T^i_{jk}\omega^j\wedge\omega^k,\quad\text{and}\quad R^\ell_k = \frac{1}{2}R^\ell_{kij}\omega^i\wedge\omega^j.$$
A computation in the frame gives the Cartan structure equations:
\begin{align*}
d\omega^i &= \omega^k\wedge\theta_k^i + T^i \\
d\theta^i_j &= \theta_j^k\wedge\theta_k^i + R_j^i.
\end{align*}

Define a map $\Phi:\mathbb{R}\times\mathbb{R}^n\to M$ by
$$\Phi(t;a^1,\dots,a^n) = \Exp_x(ta^iX_{i,x}).$$
Then $\Phi$ is called a polar normal coordinate system.  For each fixed $t$, it defines a coordinate system in a neighborhood of $U_x$, and $\Phi(t;a^1,\dots,a^n) =\Phi(1;ta^1,\dots,ta^n)$.  Our strategy is to compute $\Phi^*\omega^i$ and $\Phi^*\theta^i_j$.  Cartan's equations will then give a differential equation that governs the evolution in $t$ of these forms whose coefficients involve the curvature and torsion.  The analogous equations on $\overline{M}$ will then be the diffeomorphic image of the equation on $\overline{M}$ under the mapping $\Phi$.  So by uniqueness of solutions of differential equations, the connection coefficients on $\overline{M}$ are the diffeomorphic image of the connection coefficients on $M$.

Write
$$\Phi^*\omega^i = f^i(t;a^1,\dots,a^n)\,dt + \overline{\omega}^i,\qquad \Phi^*\theta_j^i = g^i_j(t;a_1,\dots,a_n)\,dt + \overline{\theta}_j^i$$
where $\overline{\omega}^i$ and $\overline{\theta}^i_j$ do not involve $dt$.  To compute $f^i$, fix $a^1,\dots,a^n\in\mathbb{R}$, and compute the pullback of $\Phi^*\omega^i$ under the map $t\mapsto (t;a^1,\dots,a^n)$.  That is, let $\gamma(t)=\Phi(ta^iX_{i,x})$ be the geodesic in $M$ through $x$ with initial velocity $\gamma'(0)=a^iX_{i,x}$.  Then
$$\gamma^*\omega^i = f_i(t;a_1,\dots,a_n)dt.$$
But on the other hand
$$\gamma^*\omega^i(d/dt) = \omega^i(\gamma_*d/dt) = \omega^i(a^kX_{k,x}) = a^i.$$

Similarly, to compute $g^i_j$, we compute $\gamma^*\theta^i_j = g^i_j(t;a^1,\dots,a^n)\,dt$ via
$$\gamma^*\theta^i_j(d/dt) = \theta^i_j(a^kX_{k,x})\,dt = \Gamma^i_{jk}(x)a^k\,dt = 0$$
since the frame $X_i$ is parallel translated along geodesics through $x$, and so $0=(\nabla_{X_j}X_k)_x = \Gamma_{jk}^i(x)$.  Thus
\begin{align*}
\Phi^*\omega^i &= a^i\,dt +  \overline{\omega}^i\\
\Phi^*\theta_j^i &= \overline{\theta}_j^i.
\end{align*}
Note also that $\overline{\omega}^i|_{t=0}=0$ for $\Phi_{t=0}$ is constant, and therefore the pullback of any one-form along $\Phi_{t=0}$ must vanish.  For the same reason $\overline{\theta}^i_j|_{t=0}=0$.

Now on the one hand,
\begin{align*}
\Phi^*d\omega^i &= d\Phi^*\omega^i = da^i\,dt + dt\wedge\frac{\partial\overline{\omega}^i}{\partial t} + (\text{terms not involving $dt$})\\
\Phi^*d\theta^i_j &= dt\wedge\frac{\partial \overline{\theta}^i_j}{\partial t} + (\cdots).
\end{align*}
On the other hand, by Cartan's structure equations,
\begin{align*}
\Phi^*d\omega^i &= a^j\theta^i_j\wedge dt + \Phi^* T^i + (\cdots)
&=a^j\theta^i_j\wedge dt + T_{jk}^i\circ\Phi a^j\, dt\wedge\overline{\omega}^j  + (\cdots)\\
\Phi^*d\theta^i_j &= \Phi^*R^i_j + (\cdots)\\
&=R^i_{jk\ell}\circ\Phi a^k\,dt\wedge\overline{\omega}^\ell.
\end{align*}
Equating coefficients gives
\begin{align*}
\frac{\partial\overline{\omega}^i}{\partial t} &= da^i + (T_{jk}^i\circ\Phi)a^j\overline{\omega}^k + a^j\overline{\theta}^i_j\\
\frac{\partial\overline{\theta}^i_j}{\partial t} &= (R^i_{jk\ell}\circ\Phi)a^k\overline{\omega}^\ell.
\end{align*}
With the initial conditions
\begin{align*}
\overline{\omega}^i(0;a^1,\dots,a^n) &= 0\\
\overline{\theta}^i_j(0;a^1,\dots,a^n) &= 0
\end{align*}
this defines an initial value problem for a system of first-order ordinary differential equations.

Now, the following diagram is commutative
$$\xymatrix{
&U\subset\mathbb{R}^{n+1}\ar[ld]\ar[rd]\ar[ldd]^(0.7){\Phi}\ar[rdd]_(0.7){\overline{\Phi}}&\\
U_x\subset TM_x\ar[rr]^\psi\ar[d]_{\Exp_x} && U_y\subset TM_y\ar[d]^{\Exp_y}\\
V_x\subset M\ar[rr]_{\Psi}&&V_y\subset \overline{M}
}
$$
If $\Psi^*\overline{R} = R$ and $\Psi^*\overline{T}=T$, then $\overline{\Phi}^*\omega^i, \overline{\Phi}^*\theta^i_j$ solve the same initial value problem on $\mathbb{R}^{n+1}$.  (We here use the same letters to denote the corresponding forms on $M$ and on $\overline{M}$, since it is clear from the context which is which.)  Hence $\overline{\Phi}^*\omega^i = \Phi^*\omega^i$ and $\overline{\Phi}^*\theta^i_j = \Phi^*\theta^i_j$.  By commutativity of the diagram, $\Psi^*\theta^i_j = \theta^i_j$ and $\Psi^*\omega^i=\omega^i$.  Hence
$$\Psi_*(\nabla_XX_i) = \Psi_*(\theta^j_i(X)X_j) = \theta^j_i(\Psi_*X)\Psi_*(X_j) = \nabla_{\Psi_*X}\Psi_*(X_j)$$
where in the second equality we used $\theta^j_i(X)=(\Psi^*\theta^j_i)(X) = \theta^j_i(\Psi_*X)$.

\vspace{18pt}

We end with a corollary:
\begin{corollary}\label{covariantlyconstant}
Let $M,\overline{M}$ be two affine manifolds, $x\in M$, $y\in\overline{M}$.  Let $\psi:TM_x\to T\overline{M}_y$ be a given linear isomorphism, and define $\Psi=\Exp_y\circ\psi\circ\Exp_x^{-1} : U_x\xrightarrow{\sim} U_y$, where $U_x$ is a neigbhorhood of the origin in $TM_x$ small enough that $\Psi$ is a diffeomorphism.  Suppose that $\nabla R = 0$ and $\nabla T=0$.  Then $\Psi$ is an affine diffeomorphism if and only if $\psi R_x = \overline{R}_y$ and $\psi T_x=\overline{T}_y$.
\end{corollary}
That is, when the curvature and torsion are covariantly constant, it is enough to check that they agree at a single point.

\begin{proof}
The basis $X_i$ is parallel transported along $\gamma$, and so the components of $T$ and $R$ relative to the basis are constant along $\gamma$.  This holds throughout $V_x$, and so the components of $T$ and $R$ are constant in $V_x$.  Similarly, the components of $\overline{T}$ and $\overline{R}$ are constant throughout $V_y$.  Hence $\Psi^*\overline{R}=R$ and $\Psi^*\overline{T}=T$ hold throughout the domain, because they hold at a single point.
\end{proof}

\section{Connections invariant under parallelism}
Let $M$ be a manifold with affine connection $\nabla$.  Let $a,b\in M$ be two points and $\gamma$ a piecewise $C^1$ curve from $a$ to $b$.  We say that $\pi_\gamma:TM_a\to TM_b$ (the parallel transport map) is the {\em differential of a local affine diffeomorphism} if there exist neighborhoods $U_a$ and $U_b$ of $a$ and $b$ and there is a local affine diffeomorphism $\Phi : U_a\to U_b$ such that $\Phi_{*,a}=\pi_\gamma$.   The connection $\nabla$ is said to be {\em invariant under its parallelism} if $\pi_\gamma$ is the differential of a local affine diffeomorphism for all $\gamma$.\index{Invariant under parallelism}

\begin{proposition}
The following are equivalent:
\begin{enumerate}
\item $\nabla$ is invariant under its parallelism.
\item $\pi_\gamma$ is the differential of a local affine diffeomorphism for all {\em geodesics} $\gamma$.
\item $\nabla T = \nabla R = 0$.
\end{enumerate}
\end{proposition}
\begin{proof}

$(1) \implies (2)$: Immediate from the definition of invariance under parallelism.

$(2) \implies (3)$: If $\nabla$ is invariant under parallelism, then since local affine diffeomorphisms preserve curvature and torsion, $\nabla_{\dot{\gamma}}R = \nabla_{\dot{\gamma}}T=0$ for all geodesics $\gamma$.  Therefore $\nabla R=\nabla T=0$.

$(3)\implies (1)$: Since $R$ and $T$ are covariantly constant, for any curve $\gamma$ from $a$ to $b$, $\pi_\gamma R_a = R_b$ and $\pi_\gamma T_a=T_b$.  By Corollary \eqref{covariantlyconstant}, there is a local affine diffeomorphism $\Psi:V_a\to V_b$ such that $\Psi_{*,a}=\pi_\gamma$.
\end{proof}

Suppose that $a$ and $b$ are close enough that each is the center of a star-shaped normal neighborhood containing the other, and let $[a,b]$ be the unique geodesic inside this neighborhood that connects $a$ and $b$.  Suppose that $\pi_{[a,b]} = (\tau_{a,b})_{*,a}$ where $\tau_{a,b}$ is a local affine diffeomorphism.  Then $\tau_{a,b}$ leaves the geodesics (locally near $a,b$) invariant.  As a result, if $c$ is a point on $[a,b]$ and $d$ is the point past $b$ such that $[a,c]:[b,d]=1$ (same parameter time),
$$(\tau_{a,b})_{*,c} \circ\phi_{a,c} = \pi_{b,d}\circ(\tau_{a,b})_{*,c} = \pi_{[b,d]}\circ\pi_{[a,b]}=\pi_{[a,d]}.$$
The local affine diffeomorphism $\tau_{a,b}$ is called the {\em transvection} with support $[a,b]$.\index{Transvection}

Fix $o\in M$ and let $x:I\to M$ be a geodesic through $o$.  Let $\tau_{o,x(t)}$ be the transvection with support $x(s), 0\le s\le t$.  Then, by the above,
$$\tau_{o,x(t)}\circ\tau_{o,x(s)} = \tau_{o,x(s+t)}.$$
Hence $t\mapsto \tau_{o,x(t)}$ defines a mapping $I\to \operatorname{Aff}(U_o)$, the pseudogroup of affine diffeomorphisms of $U_o$.  

\begin{definition}\index{Transvection!infinitesimal}
Let $\tau(t) = \tau_{o,x(t)}$.  Then $\left.\frac{d\tau}{dt}\right|_{t=0} = X\in\mathscr{T}^1_0U_o$ is called the infinitesimal transvection with support $o$.
\end{definition}

Let $\mathfrak{m}_o=\{\text{infinitesimal tranvections with $o$ in the support}\}\subset\mathfrak{g}_o$.  Evaluation at $o$ defines a homomorphism $\operatorname{ev}_o:\mathfrak{g}_o\to TM_o$, $\operatorname{ev}_o(X) = X_o$.  When restricting this homomorphism to $\mathfrak{m}_o$, we have
\begin{proposition}
Let $(M,\nabla)$ be an affine manifold with $\nabla$ invariant under parallelism.  Then $\operatorname{ev}_o:\mathfrak{m}_o\to TM_o$ is a linear isomorphism.
\end{proposition}
\begin{proof}
Each element $X_o\in TM_o$ generates a unique geodesic for sufficiently small time.  This geodesic then generates a transvection $\tau$ such that $\left.\frac{d\tau}{dt}\right|_{t=0} = X$ which, when evaluated at $o$ is $\operatorname{ev}_o(X) = X_o$.
\end{proof}

Let $\ker\operatorname{ev}_o = \mathfrak{k}_o\subset\mathfrak{g}_o$.   The inclusion $\mathfrak{m}_o$ splits the exact sequence
$$0\to \mathfrak{k}_o \xrightarrow{\subset} \mathfrak{g}_o \xrightarrow{\operatorname{ev}_o} TM_o \cong\mathfrak{m}_o\to 0.$$
Hence,
\begin{theorem}
Let $(M,\nabla)$ be an affine manifold with $\nabla$ invariant under parallelism.  Then $\mathfrak{g}_o=\mathfrak{k}_o\oplus\mathfrak{m}_o$.  Furthermore, $[\mathfrak{k}_o,\mathfrak{k}_o]\subset\mathfrak{k}_o$ and $[\mathfrak{k}_o,\mathfrak{m}_o]\subset\mathfrak{m}_o$.
\end{theorem}
\begin{proof}
Only the statements about the Lie bracket remain to be proven.  Firstly, since the bracket of two vector fields that vanish at $o$ also vanishes at $o$, $\mathfrak{k}_o\subset\mathfrak{g}_o$ is a Lie subalgebra: $[\mathfrak{k}_o,\mathfrak{k}_o]$.

To prove that $[\mathfrak{k}_o,\mathfrak{m}_o]\subset \mathfrak{m}_o$, let $X\in \mathfrak{k}_o$ and $Y\in \mathfrak{m}_o$.  Then $\exp(sX)$ is a local affine diffeomorphism and $\exp(sX)o=o$.  Hence $\exp(sX)_{*,o}Y$ is an infinitesimal transvection with support $o$.  So $\exp(sX)_{*,o}Y\in\mathfrak{m}_o$.  The derivative at $s=0$ is $\mathscr{L}_XY\in\mathfrak{m}_o$.  Thus $[X,Y]\in\mathfrak{m}_o$, as required.
\end{proof}

\begin{proposition}
Let $(M,\nabla)$ be an afine manifold with base point $o\in M$ and with $\nabla$ invariant under parallelism.  For $X\in\mathfrak{g}_o$, write $X=X_{\mathfrak{k}}+X_{\mathfrak{m}}$ be the splitting of $X$ into $\mathfrak{k}_o$ and $\mathfrak{m}_o$ parts.  Then for $X,Y,Z\in\mathfrak{m}_o$,
\begin{align*}
T(X_o,Y_o) &= [X,Y]_o\\
R(X_o,Y_o)Z_o &= (\mathscr{L}_{[Y,X]_{\mathfrak{k}}}Z)_o
\end{align*}
\end{proposition}

\begin{proof}
The proof relies on the following fact: if $X$ is an infinitesimal transvection at $o$ and $T$ is a tensor, then parallel transport along $X$ agrees with the pushforward along the vector flow of $X$, ad so
$$(\nabla_XT)_o = \lim_{t\to 0}\frac{1}{t}\left(\exp(-tX)_{*,\exp(tX)o}T_{\exp(tX)o}-T_o\right) = (\mathscr{L}_XT)_o.$$
Hence,
$$T(X_o,Y_o) = (\nabla_XY-\nabla_YX-[X,Y])_o = ([X,Y]-[Y,X]-[X,Y])_o = [X,Y]_o.$$
Also
\begin{align*}
R(X_o,Y_o)Z_o &= (\nabla_X\nabla_YZ - \nabla_Y\nabla_XZ - \nabla_{[X,Y]}Z)_o\\
&=(\mathscr{L}_X\nabla_YZ-\mathscr{L}_Y\nabla_X Z - \nabla_{[X,Y]}Z)_o.
\end{align*}
Now recall that for an infinitesimal affine diffeomorphism $X$, $\nabla_{[X,W]}=[\mathscr{L}_X,\nabla_W]$ for all vector fields $W$.  So continuing the above calculation of the curvature:
\begin{align*}
R(X_o,Y_o)Z_o &=(\nabla_{[X,Y]}Z + \nabla_Y\mathscr{L}_XZ - \nabla_{[Y,X]}Z - \nabla_X\mathscr{L}_YZ-\nabla_{[X,Y]}Z)_o.\\
&=(-\nabla_{[Y,X]}Z - \nabla_X\mathscr{L}_YZ + \nabla_Y\mathscr{L}_XZ)_o\\
&=(\mathscr{L}_{[Y,X]}Z - \nabla_{[Y,X]}Z)_o\\
&=(\mathscr{L}_{[Y,X]_{\mathfrak{k}}}Z - \nabla_{[Y,X]_{\mathfrak{k}}}Z)_o\qquad\!\!\text{since $\mathfrak{L}_W=\nabla_W$ for $W\in\mathfrak{m}_0$}\\
&=(\mathscr{L}_{[Y,X]_{\mathfrak{k}}}Z)_o\qquad\qquad\qquad\text{since $([X,Y]_{\mathfrak{k}})_o=0$}
\end{align*}
as required.
\end{proof}

We have the following important corollary: the structure of these Lie algebras is invariant under affine diffeomorphism.  More precisely,

\begin{corollary}
Let $(M,\nabla)$ and $(\overline{M},\overline{\nabla})$ be affine manifolds containing respective base points $o,\overline{o}$.  Suppose that $\nabla$ and $\overline{\nabla}$ are invariant under their parallelism.  Let $\Phi:U_o\to U_{\overline{o}}$ be a local affine diffeomorphism.  Then $\Phi$ induces a Lie algebra isomorphism between $\mathfrak{g}_o$ and $\mathfrak{g}_{\overline{o}}$ such that the decompositions $\mathfrak{g}_o=\mathfrak{k}_o\oplus\mathfrak{m}_o$ and $\mathfrak{g}_{\overline{o}}=\mathfrak{k}_{\overline{o}}\oplus\mathfrak{m}_{\overline{o}}$ are preserved.
\end{corollary}

\begin{proof}
Follows from the previous proposition and Theorem \ref{characterizationofaffine}.
\end{proof}

\chapter{Symmetric spaces}\label{SymmetricSpaces}
\section{Locally symmetric spaces}\index{Symmetric space!affine}
Let $(M,\nabla)$ be an affine manifold, and let $\sigma_x=-\operatorname{id}_{TM_x} \in\End(TM_x)$.  Let $s_x=\Exp_x\sigma_x\Exp_x^{-1} : U_x\to U_x$ for some normal neighborhood $U_x$ that is symmetric with respect to $s_x$ (so that $s_xU_x=U_x$).  We will sat that $(M,\nabla)$ is {\em locally symmetric at $x$} if $s_x$ is a local affine diffeomorphism for some $U_x$.  We say that $(M,\nabla)$ is {\em locally symmetric} if it is locally symmetric at all $x$.  The manifold is {\em globally symmetric} if is locally symmetric, and each $s_x$ extends to a global affine diffeomorphism at every $x$.

The local symmetry at $x$ fixes $x$ and reverses the direction of every geodesic through $x$.  Note that $x$ is the only fixed point of $s_x$ in $U_x$.  Indeed, if $f$ is an affine diffeomorphism such that $f(x)=x$ and $f(y)=y$ for some $y\in U_x$, $y\not=x$, then $f$ must preserve the geodesic $\gamma:I\to M$ connecting $x$ and $y$: $f(\gamma(t))=\gamma(t)$ for all $t\in I$.  In particular, $f_{*,x}(\dot{\gamma}(0))=\dot{\gamma}(0)$.  Thus $x$ is an isolated fixed point of a local affine diffeomorphism $f$ if and only if $1$ is not an eigenvalue of $f_{*,x}$.  As a result, $s_x$ is the only potential local affine diffeomorphism of order $2$ with $x$ an isolated fixed point.

\begin{proposition}
Let $(M,\nabla)$ be an affine manifold.  The following are equivalent:
\begin{enumerate}
\item $M$ is locally symmetric.
\item $T=0$ and $\nabla R=0$.
\item For all $x\in M$ and all symmetric normal neighborhoods $U_x$, $s_x$ is a local affine diffeomorphism.
\end{enumerate}
\end{proposition}
\begin{proof}
$(3)\implies (1)$ is the definition of locally symmetric space.

$(1)\implies (2)$: $T_x=\sigma_xT_x = -T_x$, so $T_x=0$ for all $x$.  Similarly, $\nabla R=\sigma_x\nabla R = -\nabla R$, so $\nabla R=0$ as well.

$(2)\implies (3)$: Since $T=0$ and $\nabla R=0$, the hypotheses of Corollary \ref{covariantlyconstant} are satisfied, and since $\sigma_xT_x=T_x=0$ and $\sigma_xR_x=R_x$, $s_x$ is a local affine diffeomorphism.  This holds for all symmetric normal neighborhoods because Cartan's equations in the proof of Theorem \ref{characterizationofaffine} have constant coefficients, and therefore the solution is defined on the largest possible domain of $\Exp_x^{-1}$.
\end{proof}

If $(M,\nabla)$ is a locally symmetric space, then since $T=0$ and $\nabla R=0$, the connection is invariant under parallelism.  In fact, it is possible to describe the transvections explicitly.  Suppose that $x,y\in M$ are each in a normal symmetric neighborhood of each other.  Let $[x,y]$ be the geodesic connecting $x$ and $y$.  Let $m$ be the midpoint of this geodesic (defined so that $[x,m]:[m,y]=1$).  Then
$$\tau_{[x,y]} = (s_y\circ s_m)_* = (s_m\circ s_x)_*.$$
Indeed, note first that $s_m$ the geodesic $[x,y]$ is stable under each of $s_m,s_x,s_y$.  In particular, if $T$ is a tangent vector to $[a,b]$, then $s_{m_*}T$ is also a tangent vector.  Thus if $X$ is any parallel vector along $[a,b]$, so is $s_{m,*}X$ because $s_m$ is an affine diffeomorphism and so $\nabla_Ts_{m,*}X = s_{m,*}(\nabla_{s_{m,*}T}X) = 0$.  Now, $s_{m,*}X$ is a parallel vector field along $[a,b]$ such that $s_{m,*}X_m=-X_m$.  As a result, $s_{m,*}X=-X$.  In particular $s_{m,*}X_x=-X_y$.  Thus
$$(s_y\circ s_m)_*X_x = -s_{y,*}X_y = X_y.$$
So $(s_y\circ s_m)_{*,x}=\pi_{[x,y]}$, as required.

Now, since the connection on a locally symmetric space is invariant under parallelism, the local Lie algebra $\mathfrak{g}_o\subset\mathscr{T}^1_0M_o$ of infinitesimal affine diffeomorphisms splits $\mathfrak{g}_o=\mathfrak{k}_o\oplus\mathfrak{m}_o$ as in the previous section.  The involution $s_o : U_o\to U_o$ induces a map on vector fields $s_{o,*}:\mathscr{T}^1_0M_o\to \mathscr{T}^1_0M_o$ that preserves the subspace $\mathfrak{g}_o$ as well as the splitting of $\mathfrak{g}_o$.  In fact, since $s_{o,*}\circ s_{o,*}=\operatorname{id}$, $s_{o,*}$ is a semisimple operator with eigenvalues $\pm 1$.  We have the following:

\begin{proposition}
\mbox{}
\begin{enumerate}
\item $\mathfrak{k}_o$ is the $+1$ eigenspace and $\mathfrak{m}_o$ is the $-1$ eigenspace of $s_{o,*}$.
\item The following hold:
\begin{align*}
[\mathfrak{k}_o,\mathfrak{k}_o]&\subset\mathfrak{k}_o\\
[\mathfrak{k}_o,\mathfrak{m}_o]&\subset\mathfrak{m}_o\\
[\mathfrak{m}_o,\mathfrak{m}_o]&\subset\mathfrak{k}_o
\end{align*}
\end{enumerate}
\end{proposition}
\begin{proof}
Let $X\in\mathfrak{k}_o$ generate the flow $\phi_t=\exp(tX)$.  Then $\phi_t$ is a local affine diffeomorphism such that $\phi_t(o)=o$.  Now $\phi_t$ is determined by its differential at $o$: $\phi_{t,*,o}\in \End(TM_o)$.  On the other hand, $(s_O\phi_ts_o)_{*,o} = (-\operatorname{id}_{TM_o})\phi_{t,*,o}(-\operatorname{id}_{TM_o})=\phi_{t,*,o}$, and so $s_o\phi_t s_o=\phi_t$.  Taking the derivative at $t=0$ gives $s_{o,*}X=X$.

To prove that $\mathfrak{m}_o$ is the $-1$ eigenspace, let $X\in\mathfrak{m}_o$ generate the group of transvections $\tau_{[0,x_t]} = \exp(tX)$.  Then $\tau_{[0,x_t]}=s_{m_t}\circ s_o$ and $s_o\tau_{[0,x_t]}s_o=s_os_{m_t}s_o^2=s_os_{m_t}=\tau_{[0,x_t]}^{-1}$.  Differentiating at $t=0$ gives $s_{o,*}X=-X$.

Finally, it remains only to show that $[\mathfrak{m}_o,\mathfrak{m}_o]\subset\mathfrak{k}_o$.  Let $X,Y\in\mathfrak{m}_o$.  Then
$$s_{o,*}[X,Y] = [s_{o,*}X,s_{o,*}Y] = [-X,-Y] = [X,Y].$$
So $[X,Y]$ is in the $+1$ eigenspace of $s_{o,*}$, which is $\mathfrak{k}_o$.
\end{proof}

\section{Riemannian locally symmetric spaces}\index{Symmetric space!Riemannian}
Let $(M,g)$ be a Riemannian manifold with Levi--Civita connection $\nabla$.  We say that $(M,g)$ is a Riemannian locally symmetric space if $(M,\nabla)$ is locally symmetric.

\begin{theorem}
Let $(M,g)$ be a Riemannian manifold.  The following are equivalent:
\begin{enumerate}
\item $M$ is locally symmetric.
\item $\nabla R=0$.
\item For all $x$ and all normal symmetric neighborhoods $U_x$, $s_x$ is a local affine diffeomorphism.
\item $\nabla$ is invariant under parallelism: for every geodesic $[x,y]$, there exists a local affine diffeomorphism $\phi:U_x\to U_y$ such that $\pi_{[x,y]}=\phi_{*,x}$.
\item $\pi_{[x,y]}=(s_o\circ s_x)_{*,x}$ for all $x\in U_o$ and $y=s_o(x)\in U_o$.
\item For all $o\in M$, $s_o$ is a local isometry.
\end{enumerate}
\end{theorem}

\begin{proof}
The only new item is that $(5)\implies (6)$.  This follows since $s_{o,*,x} = -(s_o\circ s_x)_{*,x} = -\pi_{[x,s_o(x)]}$ which is an isometry.
\end{proof}

We shall not prove the following theorem, due to Kobayashi (1955):
\begin{theorem}
Let $(M,\nabla)$ be a complete affine manifold that is connected and simply connected, and such that $\nabla$ is invariant under parallelism.  Then any local affine diffeomorphism can be extended (uniquely) to a global affine diffeomorphism.  In particular, if $M$ is locally symmetric, then $M$ is globally symmetric.
\end{theorem}

Isometries of Riemannian locally symmetric spaces are determined at a point:

\begin{proposition}\label{isometriesdeterminedatapoint}
Let $M,\overline{M}$ be Riemannian locally symmetric spaces and $f:TM_o\to T\overline{M}_o$ a linear isomorphism.  Then there exists a local isometry $\phi : U_o\to \overline{U}_{\overline{o}}$ such that $\phi_{*,o}=f$ if and only if $f$ is an isometry and $f_*R_o=R_{\overline{o}}$.
\end{proposition}

\begin{proof}
$\implies$: Local isometries preserve the curvature.

$\Longleftarrow$: If $f_*R_o=R_{\overline{o}}$, then Corollary \ref{covariantlyconstant} implies that there exists a local affine diffeomorphism $\phi : U_o\to \overline{U}_o$.  Let $X_i$ be a frame on $U_o$ that is orthonormal at $o$ and parallel along geodesics through $o$.  Then $X_i$ is orthonormal throughout $U_o$.  Since $\phi$ maps this orthonormal frame to another such frame, $\phi$ is an isometry.
\end{proof}

{\em Conventions:}
\begin{itemize}
\item Henceforth, ``symmetric space'' will mean ``Riemannian globally symmetric space''.  ``Locally symmetric space'' will mean ``Riemannian locally symmetric space''.
\item We will change notation.  Henceforth $\mathfrak{g}_o$ and $\mathfrak{k}_o$ refer to the Lie algebras of infinitesimal isometries and infinitesimal isometries vanishing at $o$, and $\mathfrak{g}_o^{\operatorname{aff}}$ and $\mathfrak{k}_o^{\operatorname{aff}}$ will refer to infinitesimal affine diffeomorphisms and infinitesimal affine diffeomorphisms vanishing at $o$.  Then $\mathfrak{g}_o=\mathfrak{k}_o\oplus\mathfrak{m}_o$ where
$$\mathfrak{k}_o = \mathfrak{k}_o^{\operatorname{aff}}\cap\mathfrak{g}_o,\quad \mathfrak{m}_o=\mathfrak{m}_o^{\operatorname{aff}}.$$
\end{itemize}

{\em Remark.} Infinitesimal isometries are typically known as {\em Killing fields} in the literature.

\begin{theorem}\label{affineversusriemannian}
Let $M,\overline{M}$ be (Riemannian) locally symmetric spaces.  A local isometry induces a Lie algebra isomorphism $\mathfrak{g}_o\cong\mathfrak{g}_{\overline{o}}$ preserving the decomposition and such that its restriction to $\mathfrak{m}_o\,\, (=\!TM_o)$ is an isometry to $\mathfrak{m}_{\overline{o}}\,\, (=\!T\overline{M}_{\overline{o}})$.  Conversely, a local diffeomorphism of $M$ to $\overline{M}$ with these properties is a local isometry.
\end{theorem}

\begin{proof}
$\implies$: If $\phi:M\to\overline{M}$ is a local isometry, then $\phi_*\mathfrak{g}_o=\phi_*\mathfrak{g}_{\overline{o}}$ since the composition of two isometries is an isometry.  Since $\phi_*:\mathscr{T}^1_0M\to\mathscr{T}^1_0\overline{M}$ is a Lie algebra homomorphism, it follows that $\phi_*\mathfrak{g}_o=\phi_*\mathfrak{g}_{\overline{o}}$ is an isomorphism of Lie algebras. Likewise, $\phi_*\mathfrak{k}_o=\phi_*\mathfrak{k}_{\overline{o}}$.

$\Longleftarrow$:  If $X,Y,Z\in\mathfrak{m}_o$, then $R_o(X,Y)Z=[[X,Y],Z]$.  If $\phi:M\to\overline{M}$ is a local diffeomorphism such that $\phi_*:\mathfrak{m}_o\xrightarrow{\cong}\mathfrak{m}_{\overline{o}}$ is an isometry, then $\phi_*R_o=R_{\overline{o}}$.  So from Theorem \ref{isometriesdeterminedatapoint}, $\phi$ is a local isometry.
\end{proof}

\section{Completeness}
Let $(M,g)$ be a Riemannian manifold.  Then $(M,g)$ is naturally a metric space.  Indeed, the length of a piecewise $C^1$ continuous curve $\gamma:I\to M$ is defined by
$$ L[\gamma] = \int_I g(\dot{\gamma},\dot{\gamma})^{1/2}dt.$$
The metric can be defined by
$$d(p,q) = \inf\{L[\gamma]\mid \text{$\gamma$ piecewise $C^1$ continuous curve from $p$ to $q$}\}.$$
To show that this does indeed define a metric, the only slightly tricky thing is to prove that $d(p,q)=0$ implies that $p=q$.  For this, suppose $p\not=q$ and take a relatively compact coordinate neighborhood $\phi:U\to\mathbb{R}^n$ of $p$ that excludes $q$.  Let $r = \operatorname{dist}(\phi(p),\partial \phi(U))$, and let $\lambda=\inf_{\phi(U)}\min\sigma(\phi_*g)>0$ be the smallest eigenvalue of $\phi_*g$ on $\phi(U)$.  Then, for any curve $\gamma$ from $p$ to $q$, $L[\gamma]\ge r\lambda > 0$.

There are several potential characterizations of completeness of Riemannian manifolds.  The following standard theorem of Hopf--Rinow guarantees that the most common such notions are in fact equivalent:\index{Hopf--Rinow theorem}

\begin{theorem}
Let $M$ be a connected Riemannian manifold.  Then the following are equivalent:
\begin{enumerate}
\item The Levi--Civita connection is complete, meaning that $\Exp_p$ is defined on all of $TM_p$ for every $p\in M$.
\item $(M,d)$ is a complete metric space.
\item $M$ satisfies the Heine--Borel property: a subset of $M$ is compact if and only if it is closed and bounded.
\end{enumerate}
Furthermore, any of these equivalent statements implies that any two points $p,q\in M$ can be joined by a geodesic of length $d(p,q)$.
\end{theorem}

So far, most of the results concern local isometries and local affine diffeomorphisms.  Under suitable topological conditions, local isometries can be extended to global isometries.  
\begin{theorem}\label{localglobal}
Let $(M,g) ,(\overline{M},\overline{g})$ be connected and complete Riemannian symmetric spaces with $M$ simply connected.  Then any local isometry from $M$ to $\overline{M}$ can be extended to a global isometry.
\end{theorem}

The proof requires the use of the {\em injectivity radius}:\index{Injectivity radius}
\begin{definition}
For each $x\in M$, define the injectivity radius $\inj_M(x)\in (0,\infty]$ by
$$\inj_M(x) =  \sup\{r\mid \Exp_x^{-1} : B(x,r)\to TM_x \ \text{\ is well-defined}\}.$$
\end{definition}

We note without proof that $\inj_M:M\to (0,\infty]$ is a continuous function.

\begin{proof}[Proof of Theorem]
Let $\phi$ be a local isometry with $\phi(x)=\overline{x}$.  Let $y\in M$ be another point, and let $\gamma$ be a curve of length $\ell$ in $M$ connecting $x$ and $y$.  We claim that $\phi$ can be extended to a local isometry in a neighborhood of $\gamma$.

To prove this, let $\Gamma\subset\gamma$ be the set of points in $\gamma$ that are contained in a connected open neighborhood extending $\phi$.  This is clearly an open set in the relative topology.  We claim it is also closed.  Let $U_\Gamma$ be a bounded open neighborhood of $\gamma$ on which $\phi$ extends.   Since $\phi : U_\Gamma\to \overline{M}$ is a uniformly continuous function from a metric space to a complete metric space, it extends to a continuous function on the completion (=closure) of $U_\Gamma$.  Let
$$r=\min_{u\in\operatorname{cl}U_\Gamma} \min\{\inj_M(u),\inj_{\overline{M}}(\phi(u))\}.$$
Then $r>0$ because it is the minimum of a continuous function on a compact set.  If $u$ is a limit point of $\Gamma$, then $B(u,r/2)\cap \Gamma$ is nonempty; say it contains a point $z$.  Then $B(z,r)$ is a normal neighborhood of $z$ and $B(\phi(z),r)$ is a normal neighborhood of $\phi(z)\in\overline{M}$.  By Corollary \ref{covariantlyconstant}, since $\phi_*R_z=\overline{R}_{\phi(z)}$, $\phi$ can be extended to an affine diffeomorphism $\phi:B(z,r)\to B(\phi(z),r)$ which, by virtue of Theorem \ref{affineversusriemannian}, is also an isometry.  Since $u\in B(z,r)$, it follows that $\Gamma$ is closed.  Thus $\Gamma\subset\gamma$ is both open and closed, and so $\Gamma=\gamma$.

Now the extension of $\phi$ depends only on the homotopy class of $\gamma$.  Indeed, if $\gamma_1,\gamma_2$ are homotopic curves from $x$ to $y$, then let $F:[0,1]\times [0,1]\to M$ be a homotopy for these curves.  The image $F([0,1]\times [0,1])$ of the unit square under the homotopy is a connected set, and we can consider the largest connected open subset $A\subset F([0,1]\times [0,1])$ that contains $\gamma_1$ such that $\phi$ can be extended in an open neighborhood of $A$.  Then we show by the same argument as the previous paragraph that $A$ must be closed as well, and therefore that $A=F([0,1]\times [0,1])$.

Since $M$ is simply connected by assumption, the result follows.
\end{proof}

\begin{corollary}
If $(M,g)$ is a connected, complete, simply connected Riemannian locally symmetric space, then $M$ is globally symmetric.
\end{corollary}
Indeed, the involutions $s_x$ are local isometries, which therefore extend to global isometries.  As a result of the theorem and its corollary, connected, complete, simply connected Riemannian locally symmetric spaces are globally symmetric as well, and these are completely determined by the local structure.  Moreover, the restriction of simple connectedness is not really essential:
\begin{corollary}
Let $M$ be a complete connected Riemanian locally symmetric space. The universal covering space $\widetilde{M}$ is globally symmetric.
\end{corollary}
\begin{proof}
Let $\pi:\widetilde{M}\to M$ be the covering map.  Then $\widetilde{M}$ is Riemannian with the pullback metric $\widetilde{g}=\pi^*g$.  This is a simply connected Riemannian locally symmetric space since it is locally isometric to $M$.  It is therefore enough to show that $\widetilde{M}$ is complete.

Let $\{x_n\}$ be a Cauchy sequence in $\widetilde{M}$.  Then $\{\pi x_n\}$ is Cauchy in $M$, since $d(\pi x,\pi y)\le d(x,y)$ for all $x,y\in\widetilde{M}$.  Let $\pi x_n\to X\in M$.  Let $\epsilon>0$ be small enough that the ball $B(X,\epsilon)$ is an evenly-covered neighbordhood.  So $\pi^{-1}B(X,\epsilon) = \bigcup_\alpha B(X_\alpha,\epsilon)$ is a disjoint union of balls in $\widetilde{M}$, each of which is isometric to $B(X,\epsilon)$ under the projection $\pi$.  Now, since $\{x_n\}$ is Cauchy, there exists $N>0$ such that for all $n,m>N$, $d(x_n,x_m)<\epsilon/2$.  So $\{x_n\}_{n>N}$ lies entirely within some $B(X_\alpha,\epsilon/2)$, since $\operatorname{dist}(B(X_\alpha,\epsilon/2), B(X_\beta,\epsilon/2)) \ge \epsilon$ for $\alpha\not=\beta$.  Now because $\pi:B(X_\alpha,\epsilon/2)\to B(X,\epsilon/2)$ is an isometry and $\pi x_n\to X$, it follows that $x_n\to X_\alpha$.  So $(\widetilde{M},d)$ is a complete metric space.
\end{proof}

\section{Isometry groups}  Let $(M,g)$ be a Riemannian manifold and $I(M)$ denote the group of isometries of $M$ to itself.  The group $I(M)$ is a topological group relative to the compact-open topology.\footnote{This is the topology whose basis sets are of the form $B_{K,U} = \{f\in I(M)\mid f(K)\subset U\}$ where $K,U\subset M$ with $K$ compact and $U$.} If $p\in M$, denote by $\Stab_{I(M)}(p)$ the stabilizer of $p$ in $I(M)$.  By a theorem of Myers--Steenrod (1936), $I(M)$ carries a unique differentiable structure making it a Lie group and $\Stab_{I(M)}(p)$ is a compact subgroup of $I(M)$.

Let $M$ be a Riemannian globally symmetric space with a given base point $o\in M$, and involution $s_o\in \Stab_{I(M)}(o)$.  Then $\Lie(I(M))$, the Lie algebra of $I(M)$, is $\mathfrak{g}_o$, the Lie algebra of infinitesimal isometries.  On the one hand, $\Lie(I(M))\subset\mathfrak{g}_o$.  On the other hand, since $M$ is complete, every infinitesimal isometry integrates to a local isometry, and any local isometry generates a global isometry (by Theorem \ref{localglobal}); so $\mathfrak{g}_o\subset\Lie(I(M))$.

Note that if $\phi\in\Stab_{I(M)}(o)$, then $\phi$ commutes with $s_o$.  Indeed, the differentials of $\phi$ and $s_o$ commute at $o$:
$$(\phi\circ s_o)_{*,o} = -\phi_{*,o} = (s_o\circ\phi)_{*,o}$$
and so $\phi$ and $s_o$ must also commute globally, since affine diffeomorphisms (in particular isometries) are determined by their differentials at a point.  

Now, the Lie algebra of the centralizer $C_{I(M)}(s_o)$ is the set of fixed points of the differential $s_{o,*}:\mathfrak{g}_o\to\mathfrak{g}_o$, which is just
$$\Lie(C_{I(M)}(s_o)) = \mathfrak{k}_o.$$

{\em Exercise.}  If $G$ is a topological group, denote by $G_e$ the identity component of $G$.  Show that $(C_{I(M)}(s_o))_e, (C_{I(M)_e}(s_o))_e, \Stab_{I(M)_e}(o),C_{I(M)_e}(s_o)$ all have the same Lie algebra $\mathfrak{k}_o$ and that $(C_{I(M)}(s_o))_e = (C_{I(M)_e}(s_o))_e\subset\Stab_{I(M)_e}(o)\subset C_{I(M)_e}(s_o)$.

\chapter{Orthogonal symmetric Lie algebras}
\begin{definition}
\mbox{}\index{Orthogonal symmetric Lie algebra}
\begin{enumerate}
\item A pair $(\mathfrak{g}_o,s)$ consisting of a Lie algebra $\mathfrak{g}_o$ and order $2$ Lie algebra automorphism $s:\mathfrak{g}_o\to\mathfrak{g}_o$ is called a symmetric Lie algebra.
\item A symmetric Lie algebra is called effective if $\mathfrak{k}_o=C_{\mathfrak{g}_o}(s)$, the set of fixed points of $s$, contains no nonzero ideal of $\mathfrak{g}_o$.
\item An orthogonal symmetric Lie algebra is a symmetric Lie algebra $(\mathfrak{g}_o,s)$ such that $\mathfrak{k}_o$ is a compactly embedded subalgebra of $\mathfrak{g}_o$, meaning that the Lie subgroup of $\Int(\mathfrak{g}_o)$ generated by $\ad_{\mathfrak{g_o}}\mathfrak{k}_o$ is compact.
\end{enumerate}
\end{definition}

If $(\mathfrak{g}_o,s)$ is a symmetric Lie algebra, then we can decompose $\mathfrak{g}_o=\mathfrak{k}_o\oplus\mathfrak{p}_o$ into the $+1$ and $-1$ eigenspaces for the automorphism $s$.  The eigenspaces satisfy the familiar properties:
\begin{equation}\label{commutators}
[\mathfrak{k}_o,\mathfrak{k}_o]\subset \mathfrak{k}_o,\quad [\mathfrak{k}_o,\mathfrak{p}_o]\subset\mathfrak{p}_o,\quad [\mathfrak{p}_o,\mathfrak{p}_o]\subset\mathfrak{k}_o.
\end{equation}

The following criterion explains the use of the term {\em effective}:
\begin{lemma}
A symmetric Lie algebra $(\mathfrak{g}_o,s)$ is effective if and only if $\ad: \mathfrak{k}_o\to\mathfrak{gl}(\mathfrak{p}_o)$ is injective.
\end{lemma}
\begin{proof}
If $\ad:\mathfrak{k}_o\to\mathfrak{gl}(\mathfrak{p}_o)$ were not injective, then its kernel $\mathfrak{i}$ would be an ideal of $\mathfrak{k}_o$.  But then $\mathfrak{i}$ would also be an ideal of $\mathfrak{g}_o$, since if $X=X_{\mathfrak{k}}+X_{\mathfrak{p}}\in\mathfrak{g}_o$ and $Y\in \mathfrak{i}$, then $[X,Y] = [X_{\mathfrak{k}}+X_{\mathfrak{p}}, Y] \in \mathfrak{i}$ because $[X_{\mathfrak{p}}, Y]=0$ since $\mathfrak{i}$ commutes with $\mathfrak{p}_o$ and $[X_{\mathfrak{k}},Y]\in \mathfrak{i}$ since $\mathfrak{i}$ is an ideal of $\mathfrak{k}_o$.

Conversely, if there were an ideal $\mathfrak{i}$ of $\mathfrak{g}_o$, $\mathfrak{i}\subset \mathfrak{k}_o$, then $[\mathfrak{i},\mathfrak{p}_o]\subset \mathfrak{i}$ (being an ideal) and $[\mathfrak{i},\mathfrak{p}_o]\subset\mathfrak{p}_o$ (by \eqref{commutators}).  Hence $[\mathfrak{i},\mathfrak{p}_o]\subset \mathfrak{i}\cap\mathfrak{p}_o = 0$.  Thus $\mathfrak{i}\subset \ker [\ad:\mathfrak{k}_o\to\mathfrak{gl}(\mathfrak{p}_o)]$.
\end{proof}

{\em Remark.} Helgason defines an effective symmetric Lie algebra as a Lie algebra such that $\mathfrak{k}_o\cap Z(\mathfrak{g}_o)=0$.  This is a weaker condition equivalent to the adjoint representation of $\mathfrak{k}_o$ on all of $\mathfrak{g}_o$ being effective.  Our definition is from Kobayashi and Nomizu.

\vspace{24pt}

The following criterion explains the use of the term {\em orthogonal}:
\begin{lemma}
A symmetric Lie algebra is orthogonal if and only if there exists a positive definite bilinear form on $\mathfrak{g}_o$ that is invariant under $\ad(\mathfrak{k}_o)$.
\end{lemma}
\begin{proof}
For one direction, take any positive definite bilinear form on $\mathfrak{g}_o$ and average it over the (compact) Lie subgroup of $\Int(\mathfrak{g}_o)$ generated by $\ad_{\mathfrak{g}_o}(\mathfrak{k}_o)$.  Conversely, if $Q$ is a $\mathfrak{k}_o$-invariant positive-definite form on $\mathfrak{g}_o$, then $\mathfrak{k}_o\xrightarrow{\ad}\mathfrak{so}(\mathfrak{g}_o,Q)\subset\mathfrak{int}(\mathfrak{g}_o)$ is a compact embedding.
\end{proof}

{\em Remark.} A symmetric Lie algebra is orthogonal if and only if $\mathfrak{k}$ is a compact Lie algebra that leaves invariant a positive definite form on $\mathfrak{p}_o$.

\section{Structure of orthogonal symmetric Lie algebras}
Henceforth, all orthogonal symmetric Lie algebras are assumed to be effective.
\begin{lemma}
Let $(\mathfrak{g}_o,s)$ be an orthogonal symmetric Lie algebra.  Then $\mathfrak{g}_o$ is semisimple if and only if $C_{\mathfrak{g}_o}(\mathfrak{p}_o)=0$.
\end{lemma}
\begin{proof}
If $(\mathfrak{g}_o,s)$ is semisimple, then $C_{\mathfrak{g_o}}(\mathfrak{p}_o)$ is invariant under $s$, and therefore splits as the direct sum
$$C_{\mathfrak{g_o}}(\mathfrak{p}_o)=(C_{\mathfrak{g_o}}(\mathfrak{p}_o)\cap\mathfrak{k}_o)\oplus (C_{\mathfrak{g_o}}(\mathfrak{p}_o)\cap\mathfrak{p}_o).$$
However, $C_{\mathfrak{g_o}}(\mathfrak{p}_o)\cap\mathfrak{p}_o$ is an abelian ideal of $\mathfrak{g}_o$, and must therefore be zero by semisimplicity.  On the other hand, $C_{\mathfrak{g_o}}(\mathfrak{p}_o)\cap\mathfrak{k}_o=\ker\ad(\mathfrak{k}_o)|_{\mathfrak{p}_o}=0$ by effectiveness.

Conversely, assume that $C_{\mathfrak{g_o}}(\mathfrak{p}_o)=0$.  Let $\mathfrak{r} = \ker B_{\mathfrak{g}_o}$ be the radical  of $\mathfrak{g}_o$.  Then since $B_{\mathfrak{g}_o}$ is $s$-invariant, $\mathfrak{r}=(\mathfrak{r}\cap\mathfrak{k}_o)\oplus(\mathfrak{r}\cap\mathfrak{p}_o)$. Note that $B_{\mathfrak{g}_o}|_{\mathfrak{k}_o\times\mathfrak{k}_o}$ is the Killing form of the adjoint representation $\ad(\mathfrak{k}_o)\subset \mathfrak{so}(\mathfrak{g}_o,Q)$, which is negative definite.\footnote{A Lie algebra is compact if and only if its Killing form is negative definite.}  Thus $\mathfrak{r}\cap \mathfrak{k}_o=0$ which implies that $\mathfrak{r}\subset\mathfrak{p}_o$.  But then $[\mathfrak{r},\mathfrak{p}_o]\subset \mathfrak{r}\cap\mathfrak{k}_o = 0$, so $\mathfrak{r}\subset C_{\mathfrak{g}_o}(\mathfrak{p}_o)=0$, and it follows that $\mathfrak{g}_o$ is semisimple.
\end{proof}

\begin{definition}
Let $(\mathfrak{g}_o,s)$ be an orthogonal symmetric Lie algebra.  An extension of $(\mathfrak{g}_o,s)$ is an orthogonal symmetric Lie algebra $(\overline{\mathfrak{g}}_o,\overline{s})$ such that
\begin{enumerate}
\item $\mathfrak{g}_o\subset\overline{\mathfrak{g}}_o$ is a subalgebra.
\item $\overline{s}|_{\mathfrak{g}_o} = s$
\item $\overline{\mathfrak{p}}_o=\mathfrak{p}_o$
\end{enumerate}
An orthogonal symmetric Lie algebra is {\em maximal} if it has no proper extensions.\index{Orthogonal symmetric Lie algebra!maximal}
\end{definition}

\begin{lemma}
Let $(\mathfrak{g}_o,s)$ be orthogonal symmetric.  If $\mathfrak{g}_o$ is semisimple, then it is maximal.  Furthermore, in that case $\mathfrak{k}_o=[\mathfrak{p}_o,\mathfrak{p}_o]$.
\end{lemma}

\begin{proof}
Let $(\overline{\mathfrak{g}}_o,s)$ be an extension of $(\mathfrak{g}_o,s)$.  Note that the Killing criterion implies that $\overline{\mathfrak{g}}_o$ is semisimple as well.  Let $\ell = [\mathfrak{p}_o,\mathfrak{p}_o]\oplus\mathfrak{p}_o$.  This is an ideal in $\mathfrak{g}_o,\overline{\mathfrak{g}}_o$ which is semisimple and $s$-invariant.  Hence, since $\ell$ is semisimple,
$$\mathfrak{g}_o = \ell\oplus C_{\mathfrak{g}_o}(\ell), \quad \overline{\mathfrak{g}}_o = \ell\oplus C_{\overline{\mathfrak{g}}_o}(\ell).$$

Now, $C_{\mathfrak{g}_o}(\ell)\subset C_{\mathfrak{g}_o}(\mathfrak{p}_o)=0$.  Also, by effectiveness, $C_{\overline{\mathfrak{g}}_o}(\mathfrak{p}_o)\cap\mathfrak{k}_o=\ker\ad(\mathfrak{k}_o)|_{\mathfrak{p}_o}=0$.  Hence, $C_{\overline{\mathfrak{g}}_o}(\ell)\subset C_{\overline{\mathfrak{g}}_o}(\mathfrak{p}_o)\subset \mathfrak{p}_o\subset\ell$. So $C_{\overline{\mathfrak{g}}_o}(\ell) = 0$ as well.
\end{proof}

\begin{theorem}\label{decomposition}\index{Orthogonal symmetric Lie algebra!irreducible}
Let $(\mathfrak{g}_0,s)$ be an orthogonal symmetric Lie algebra.  Then there is a direct sum decomposition
$$(\mathfrak{g}_o,s) = (\mathfrak{h}_f,s_f)\oplus (\mathfrak{h}_1,s_1)\oplus\cdots\oplus (\mathfrak{h}_t,s_t)$$
such that $(\mathfrak{h}_f,s_f)$ is flat and $(\mathfrak{h}_i,s_i)$ is irreducible.  The decomposition is unique up to ordering.
\end{theorem}

\begin{proof}
The forms $Q, B_{\mathfrak{g}_o}|_{\mathfrak{p}_o\times\mathfrak{p}_o}$ are both $\ad(\mathfrak{k}_o)$-invariant, so there exists $A:\mathfrak{p}_o\to\mathfrak{p}_o$, symmetric with respect to $Q$, such that $B(x,y)=Q(Ax,y)$.  So $A$ is orthogonally diagonalizable over $\mathbb{R}$.  Let $0,\lambda_1,\dots,\lambda_n$ be the eigenvalues and $V_0,V_1,\dots,V_n$ be the corresponding eigenspaces, which are orthogonal with respect to $Q$ (as well as $B$), and are $\ad(\mathfrak{k}_o)$-invariant.

We claim that $[V_i,V_j]=0$ for all $i\not=j$, and $[V_0,V_0]$.  Indeed, if $x\in V_i$ and $y\in V_j$, then $B([x,y],[x,y])=B(x,[y,[x,y]])=0$ since $x\in V_i$ and $[y,[x,y]]\in V_j$.  Thus $[x,y]\in V_0\cap\mathfrak{k}_o=0$.

Now $A$ commutes with $\ad(\mathfrak{k}_o)$, and so each $V_i\, (i\not=0)$ decomposes into a sum of simple $\mathfrak{k}_o$-modules, so that
$$V_1\oplus\cdots\oplus V_n = \mathfrak{p}_1\oplus\cdots\oplus\mathfrak{p}_t.$$
Now $\mathfrak{p}_o=V_0\oplus \mathfrak{p}_1\oplus\cdots\oplus\mathfrak{p}_t$ and $[\mathfrak{p}_i,\mathfrak{p}_j]=0$ for $i\not=j$. Let $\mathfrak{h}_i=[\mathfrak{p}_i,\mathfrak{p}_i]\oplus\mathfrak{p}_i$.  This is an ideal in $\mathfrak{g}_o$, since it is a $\mathfrak{k}_o$-module that stable under brackets with each of the the $\mathfrak{p}_j$.  Because $B_{\mathfrak{h}_i}=B_{\mathfrak{g}_o}|_{\mathfrak{h}_i\times\mathfrak{h}_i}$ is non-degenerate, $\mathfrak{h}_i$ is a semisimple ideal.  Hence there is a complete reduction
$$\mathfrak{g}_o=C_{\mathfrak{g}_o}(\mathfrak{h})\oplus(\mathfrak{h}_1\oplus\cdots\oplus\mathfrak{h}_t)$$
where we have denoted $\mathfrak{h}=\mathfrak{h}_1\oplus\cdots\oplus\mathfrak{h}_t$.  Each factor $(\mathfrak{h}_i,s|_{\mathfrak{h_i}})$ is a symmetric orthogonal Lie algebra.  Let $\mathfrak{h}_f = C_{\mathfrak{g}_o}(\mathfrak{h})$.  This is an $s$-invariant ideal and so
$$\mathfrak{h}_f=C_{\mathfrak{g}_o}(\mathfrak{h})=(C_{\mathfrak{g}_o}(\mathfrak{h})\cap\mathfrak{k}_o)\oplus(C_{\mathfrak{g}_o}(\mathfrak{h})\cap\mathfrak{p}_o) = (C_{\mathfrak{g}_o}(\mathfrak{h})\cap\mathfrak{k}_o)\oplus V_o.$$

Hence
$$\mathfrak{k}_o=(C_{\mathfrak{g}_o}(\mathfrak{h})\cap\mathfrak{k}_o)\oplus[\mathfrak{p}_1,\mathfrak{p}_1]\oplus\cdots\oplus[\mathfrak{p}_t,\mathfrak{p}_t].$$
Moreover, the action of $\mathfrak{k}_o$ on $\mathfrak{p}_i$ factors through $\ad([\mathfrak{p}_i,\mathfrak{p}_i])$:
$$
\xymatrix{
\mathfrak{k}_o\ar[rr]^\ad\ar[dr]&&\mathfrak{gl}(\mathfrak{p}_i)\\
&[\mathfrak{p}_i,\mathfrak{p}_i]\ar[ur]^\ad&
}
$$
So if $\mathfrak{p}_i$ were reducible as a $[\mathfrak{p}_i,\mathfrak{p}_i]$-module, then it would also be reducible as a $\mathfrak{k}_o$-module.  However it is simple by construction, and therefore each $(\mathfrak{h}_i,s_i)$ is an irreducible symmetric orthogonal Lie algebra.
\end{proof}

\begin{theorem}
Let $(\mathfrak{g}_o,s)$ be an irreducible orthogonal symmetric Lie algebra.  Then $B=cQ$ with
\begin{itemize}
\item $c=0$: flat
\item $c>0$: $\mathfrak{g}_o$ is noncompact and semisimple.  In this case $\mathfrak{k}_o$ is a maximal compact subalgebra.
\item $c<0$: $\mathfrak{g}_o$ is compact and $s$ is an order $2$ automorphism.
\end{itemize}
\end{theorem}

{\em Remark.}  In the $c<0$ case, if $\mathfrak{g}_o$ is not simple, $\mathfrak{g}_o=\bigoplus_{i=1}^t\mathfrak{g}_i$ where $\mathfrak{g}_i\trianglelefteq\mathfrak{g}_o$ are simple ideals.  The involution $s$ must interchange them, say $s(\mathfrak{g}_1)=\mathfrak{g}_2$.  In particular $\mathfrak{g}_1\cong\mathfrak{g}_2$.  So $(\mathfrak{g}_1\oplus\mathfrak{g}_2,s)$ is an orthogonal symmetric Lie algebra, and therefore by irreducibility, $\mathfrak{g}_o=\mathfrak{g}_1\oplus\mathfrak{g}_2$.  Thus, up to isomorphism,
$$\mathfrak{g}_o = \mathfrak{r}_o\oplus\mathfrak{r}_o$$
where $\mathfrak{r}_o$ is a compact simple Lie lagebra and $s$ is the flip map $s(x\oplus y)=y\oplus x$.

{\em Notation.}  Henceforth we fix the following notation.  $\mathfrak{g}_o$ will denote a real Lie algebra.  $\mathfrak{g}=(\mathfrak{g}_o)_{\mathfrak{C}}$ is the complexification of $\mathfrak{g}_o$.  Finally, $\mathfrak{g}^{\mathbb{R}}$ is $\mathfrak{g}$ regarded as a real Lie algebra.

\subsection{Real forms}\index{Real form} The real Lie algebra $\mathfrak{g}_o$ is called a real form of the complex Lie algebra $\mathfrak{g}$. Associated to any complex semisimple Lie algebra $\mathfrak{g}$ there are precisely two real forms up to isomorphism.  To describe these, let $\mathfrak{t}\subset\mathfrak{g}$ be a maximal toral subalgebra.  Let $\mathfrak{t}_o$ be the real abelian Lie algebra such that $\mathfrak{t}_o\otimes\mathbb{C}=\mathfrak{t}$.  Let $\Phi\subset\mathfrak{t}^*$ be a set of roots for $\mathfrak{g}$, so that each $\alpha\in\Phi$ is an eigenvalue for the action of $\mathfrak{t}$ on $\mathfrak{g}$ and there is an eigenspace decomposition
$$\mathfrak{g}=\mathfrak{t}\oplus\bigoplus_{\alpha\in\Phi}\mathfrak{g}_\alpha.$$
Then each $\mathfrak{g}_\alpha$ is one-dimensional, each spanned by a single element $x_\alpha$.  Let $h_\alpha$, $\alpha\in\Phi$, be elements that are dual to $\alpha$ under the Killing form of $\mathfrak{g}$.  The $x_\alpha$ can be chosen so that $[x_\alpha,x_{-\alpha}]=h_\alpha$ and $[x_\alpha,x_\beta]=N_{\alpha,\beta}x_{\alpha+\beta}$ whenever $\alpha+\beta\in\Phi$ where $N_{\alpha,\beta}\in\mathbb{R}$ satisfy $N_{\alpha,beta}=N_{-\alpha,-\beta}$.

Divide the roots into positive and negative roots $\Phi=\Phi^+\cup\Phi^-$ (relative to some generic hyperplane in $\mathfrak{t}^*$).  Let $\mathfrak{k}_o = \bigoplus_{\alpha\in\Phi^+}\mathbb{R}(x_\alpha + x_{-\alpha})$ and $\mathfrak{p}_o=\mathfrak{t}_o\oplus\bigoplus_{\alpha\in\Phi^+}\mathbb{R}(x_\alpha - x_{-\alpha}).$  The Lie algebra
$$\mathfrak{g}_o = \mathfrak{k}_o\oplus i\mathfrak{p}_o$$
is a compact real form of $\mathfrak{g}$.  Indeed, its Killing form is the restriction to $\mathfrak{g}_o$ of the Killing form on $\mathfrak{g}$.  But, since $[x_\alpha,x_{-\alpha}]=h_\alpha$, $B(x_\alpha,x_{-\alpha})=1$, so
\begin{align*}
B(x_\alpha-x_{-\alpha},x_\alpha - x_{-\alpha})&<0\\
B(i(x_\alpha+x_{-\alpha}),i(x_\alpha + x_{-\alpha}))&<0\\
B(x_\alpha-x_{-\alpha},i(x_\alpha + x_{-\alpha}))&=0\\
B(ih_\alpha,ih_\alpha)&<0.
\end{align*}
Hence $B|_{\mathfrak{g}_o\times\mathfrak{g}_o}$ is negative-definite.

The Lie algebra
$$\widetilde{\mathfrak{g}}_o = \mathfrak{k}_o\oplus \mathfrak{p}_o$$
is the noncompact real form of $\mathfrak{g}_o$.  Note that $\mathfrak{k}_o$ is the maximal compact subalgebra of $\widetilde{\mathfrak{g}}_o$ since the restriction of the Killing form to $\mathfrak{k}_o$ is negative semidefinite and the restriction to $\mathfrak{p}_o$ is positive definite.

The functor $\mathfrak{g}_o\mapsto \widetilde{\mathfrak{g}}_o$ interchanges the compact and non-compact forms of a given real orthogonal symmetric Lie algebra $\mathfrak{g}_o$.  The Lie algebrs $\mathfrak{g}_o$ and $\widetilde{\mathfrak{g}}_o$ are said to be dual to each other.

\section{Rough classification of orthogonal symmetric Lie algebras}
The initial classification of orthogonal symmetric Lie algebras is as follows:
\begin{enumerate}[Type I]\index{Orthogonal symmetric Lie algebra!Type I--IV}
\item $\mathfrak{g}_o$ is compact and simple
\item $(\mathfrak{g}_o,\mathfrak{g}_o, s)$ where $s$ is the flip and $\mathfrak{g}_o$ is compact and simple.
\item $\mathfrak{g}_o$ is simple, non-compact, and $\mathfrak{k}_o$ is a maximal compact subalgebra.
\item $\mathfrak{g}_o = \mathfrak{g}^{\mathbb{R}}$ is the underlying real Lie algebra of a complex simple Lie algebra.  $s: \mathfrak{g}^{\mathbb{R}}\to  \mathfrak{g}^{\mathbb{R}}$ is complex conjugation.
\end{enumerate}

Of these, types I and II are compact and types III and IV are noncompact.  The duality functor $\mathfrak{g}_o\mapsto\widetilde{\mathfrak{g}}_o$ interchanges types I and III and types II and IV.  For types I and II, $Q=-B_{\mathfrak{g}_o}|_{\mathfrak{p}_o\times\mathfrak{p}_o}$; for types III and IV, $Q = B_{\mathfrak{g}_o}|_{\mathfrak{p}_o\times\mathfrak{p}_o}$.

\section{Cartan involutions}
\begin{definition}\index{Cartan involution}
Let $\mathfrak{g}_o$ be a real Lie algebra.  An order $2$ automorphism $\theta : \mathfrak{g}_o\to\mathfrak{g}_o$ is said to be a Cartan involution if $B_\theta(X,Y) =  -B_{\mathfrak{g}_o}(X,\theta T)$ is positive definite.
\end{definition}

{\em Remark.}
\begin{enumerate}
\item Let $\theta$ be a Cartan involution.  Then $B_{\mathfrak{g}_o}$ is $\theta$-invariant:
$$B_{\mathfrak{g}_o}(\theta X,\theta Y) = \tr(\ad\theta X\,\ad\theta Y) = B_{\mathfrak{g}_o}(X,Y)$$
since $\theta$ is an automorphism of $\mathfrak{g}_o$.
\item Let $\mathfrak{g}$ be complex semisimple and $\mathfrak{u}_o$ a compact real form.  Then the complex conjugattion map
$$\mathfrak{g} = \mathfrak{u}_o\oplus i\mathfrak{u}_o\xrightarrow{\tau} \mathfrak{u}_o\oplus i\mathfrak{u}_o$$
is a Cartan involution.  Indeed, if $X,Y\in\mathfrak{u}_o$, then
\begin{align*}
B_\tau(X+iY,X+iY)&=-2B_{\mathfrak{g}^{\mathbb{R}}}(X+iY,X-iY) = -2\operatorname{Re}B_{\mathfrak{g}}(X+iY,X-iY)\\
&=-2\operatorname{Re}\left(B_{\mathfrak{g}}(X,X)+B_{\mathfrak{g}}(Y,Y)\right)>0.
\end{align*}
\end{enumerate}

\begin{lemma}
Let $\mathfrak{g}_o$ be a real semisimple Lie algebra, $\theta$ a Cartan involution and $\tau$ an involution.  Then there exists an interior automorphism $\phi\in\operatorname{Int}(\mathfrak{g}_o)$ such that $\phi\theta\phi^{-1}$ and $\tau$ commute.
\end{lemma}

\begin{proof}
The group $\operatorname{Int}(\mathfrak{g}_o)$ is generated by the exponential of the adjoint representation:
$$ \operatorname{Int}(\mathfrak{g}_o) = \left\langle e^{\ad X}\mid X\in\mathfrak{g}_o\right\rangle = \operatorname{Aut}(\mathfrak{g}_o)_e$$
since $\mathfrak{g}_o$ is semisimple.  Let $\omega=\tau\theta$ so that $\omega^{-1}=theta\tau$.  Since this is a Lie algebra automorphism, $B_{\mathfrak{g}_o}$ is $\omega$-invariant.  As a result, $\omega$ is symmetric with respect to $B_\theta$: indeed,
\begin{align*}
B_\theta(\omega X,Y) &= -B(\omega X,\theta Y) = -B(X,\omega^{-1}\theta Y) = -B(X,(\theta\tau)\theta Y)\\
&=-B(X,\theta\omega Y) = B_\theta(X,\omega Y).
\end{align*}
Since $B_\theta$ is positive definite, $\omega$ is diagonalizable over $\mathbb{R}$.  Let $\rho=\omega^2$.  This is diagonalizable over $\mathbb{R}$ with positive eigenvalues $\lambda$.  Hence $\rho^r$ are diagonalizable with positive eigenvalues $\lambda^r$.  Moreover, by the spectral mapping theorem,
$$\rho^r\omega = \omega\rho^r$$
for all $r\in\mathbb{R}$.  Also, since $\rho\theta = \tau\theta\tau = \theta\rho^{-1},$
the action of $\theta$ sends the eigenspace $V_\lambda$ to $V_{1/\lambda}$ for each eigenvalue $\lambda$ of $\rho$.  Hence we have for all $r$ that
$$\rho^r\theta = \theta\rho^{-r}.$$
Set $\phi=\rho^{1/4}$.  Then
\begin{align*}
\phi\theta\phi^{-1}\tau &= \rho^{1/4}\theta\rho^{-1/4}\tau = \rho^{1/2}\theta\tau = \rho^{1/2}\omega^{-1}\\
&=\rho^{-1/2}\rho\omega = \rho^{-1/2}\omega = \omega\rho^{-1/2}\\
&=\tau\theta\rho^{-1/2} = \tau\rho^{1/4}\theta\rho^{-1/4} = \tau\phi\theta\phi^{-1}.
\end{align*}
So $\phi\theta\phi^{-1}$ commutes with $\tau$.  Moreover, $r\mapsto\rho^r$ is a 1-parameter subgroup of $\operatorname{Aut}(\mathfrak{g}_o)$, so $\phi$ is connected to the identity in $\operatorname{Aut}(\mathfrak{g}_o)$, and therefore $\phi\in\operatorname{Int}(\mathfrak{g}_o)$.
\end{proof}

\begin{proposition}
Let $\mathfrak{g}_o$ be a real semisimple noncompact Lie algebra.  Then $\mathfrak{g}_o$ has a Cartan involution.
\end{proposition}

\begin{proof}
Let $\mathfrak{g}_o\subset\mathfrak{g}^{\mathbb{R}}$, and let $\mathfrak{u}_o$ be a compact real form of $\mathfrak{g}$.  Let $\tau$ be conjugation with respect to $\mathfrak{g}_o$ and let $\theta$ be conjugation with respect to $\mathfrak{u}_o$.  Then $\theta$ is a Cartan involution.  So there exists $\phi\in\operatorname{Int}(\mathfrak{g}^{\mathbb{R}})=\operatorname{Int}(\mathfrak{g})$ such that $\phi\theta\phi^{-1}$ commutes with $\tau$.  As a result, $\phi\theta\phi^{-1}$ preserves the set of fixed points of $\tau$, and so
$$\phi\theta\phi^{-1} : \mathfrak{g}_o\to\mathfrak{g}_o$$
is a Lie algebra involution.  Furthermore, $\phi(\mathfrak{u}_o)\subset\mathfrak{g}$ is a compact form and
$$\phi\theta\phi^{-1} : \mathfrak{g}^{\mathbb{R}}\to\mathfrak{g}^{\mathbb{R}}$$
is conjugation with respect to $\phi(\mathfrak{u}_o)$.  So $\phi\theta\phi^{-1}$ is a Cartan involution of $\mathfrak{g}^{\mathbb{R}}$, and so $B_{\phi\theta\phi^{-1}}$ is positive definite.
\end{proof}

Any two Cartan involutions are conjugate by an interior automorphism:

\begin{proposition}
Let $\mathfrak{g}_o$ be a real semisimple non-compact Lie algebra and let $\theta_1,\theta_2$ be Cartan involutions.  Then there exists $\phi\in\operatorname{Int}(\mathfrak{g}_o)$ such that
$$\phi\theta_1\phi^{-1} = \theta_2.$$
\end{proposition}

\begin{proof}
There exists $\phi\in\operatorname{Int}(\mathfrak{g}_o)$ such that $\phi\theta_1\phi^{-1}$ commutes with $\theta_2$.  In particular, $\phi\theta_1\phi^{-1}$ and $\theta_2$ are simultaneously diagonalizable.  We are done if we show that each eigenvector of one is an eigenvector of the other with respect to the same eigenvalue.  Suppose by contradiction that there is $X\in\mathfrak{g}_o$, $X\not=0$, such that
$$\phi\theta_1\phi^{-1}(X) = X,\quad \theta_2(X) = -X.$$
Then
\begin{align*}
0 &< B_{\phi\theta_1\phi^{-1}}(X,X) = -B(X,\phi\theta_1\phi^{-1}(X)) = - B(X,X)\\
0 &< B_{\theta_2}(X,X) = -B(X,\theta_2X) = B(X,X),
\end{align*}
a contradiction.  Hence $\phi\theta_1\phi^{-1}$ and $\theta_2$ have the same eigenspaces, and therefore $\phi\theta_1\phi^{-1}=\theta_2$.
\end{proof}

\begin{corollary}
Let $\mathfrak{g}$ be a complex semisimple Lie algebra.  Then any two compact real forms of $\mathfrak{g}$ are conjugate with respect to $\operatorname{Int}(\mathfrak{g})$.  Any Cartan involution of $\mathfrak{g}$ is conjugation with respect to some compact real form.
\end{corollary}

Let $\mathfrak{g}_o$ be a noncompact real semisimple Lie algebra and $\theta$ a Cartan involution on $\mathfrak{g}_o$.  Under the adjoint map, $\mathfrak{g}_o$ is a subalgebra of $\mathfrak{gl}(n)$ for some $n$:
$$\mathfrak{g}_o\xrightarrow{\cong}\ad(\mathfrak{g}_o)\subset\mathfrak{gl}(\mathfrak{g}_o).$$
Relative to the form $B_\theta$ on $\mathfrak{g}_o$, define the transpose of an endomorphism of $\mathfrak{g}_o$ by
$$B_\theta(A^Tx,y) = B_\theta(x,Ay),\quad \text{for all $x,y\in\mathfrak{g}_o$}.$$
Notice that $\ad(X)^T = \ad(-\theta X)$: indeed,
\begin{align*}
B_\theta(ad(X)^Tx,y) &= B_\theta(x,\ad(X)y) = -B_{\mathfrak{g}_o}(x,\theta\ad(X)y)\\
&=-B_{\mathfrak{g}_o}(x,\ad(\theta X)\theta y) = B_{\mathfrak{g}_o}(\ad(\theta X)x,\theta y)\\
&=-B_\theta(\ad(\theta X)x,y).
\end{align*}
Thus $\ad(\mathfrak{g}_o)\subset\mathfrak{gl}(\mathfrak{g}_o)$ is invariant under the transpose.  Moreover, the form $B_\theta$ is
$$B_\theta(X,Y) = -\tr_{\mathfrak{g}_o}(\ad X,\ad(\theta Y)) = \tr_{\mathfrak{g}_o}(\ad X \ad(Y)^t).$$

To summarize:

{\em Remark:}  We can confine attention to linear Lie algebras $\mathfrak{g}_o\subset\mathfrak{gl}(n,\mathbb{R})$ that are closed under the transpose.  The form $(X,Y)\to \tr(XY^t)$ is positive-definite on $\mathfrak{g}_o$.

\section{Flat subspaces}
\subsection{Regular elements}  Let $\mathfrak{g}$ be a complex Lie algebra, $\pi : \mathfrak{g}\to\mathfrak{gl}(V)$ a representation of $\mathfrak{g}$.  For $x\in\mathfrak{g}$, let $V_0(x) = \ker \pi(x)^\infty$ be the maximal subspace on which $\pi(x)$ is nilpotent (the generalized eigenspace for the eigenvalue $0$).  The dimension of $V_0(x)$ can be read off the characteristic polynomial of $\pi(x)$:
$$\det(\lambda-\pi(x)) = \lambda^2-\sum_j\lambda^ja_j(x)$$
where the $a^j$ are some polynomial functions on $\mathfrak{g}_o$ depending only on $\pi$.  Then $\dim V_0(x)$ is the minimum $j$ such that $a_j(x)\not=0$. Define
$$\ell_G(V) = \min_{x\in\mathfrak{g}}\dim V_0(x).$$
Call an element $x\in\mathfrak{g}$ $\pi$-regular if $\dim V_0(x) = \ell_G(V)$.  This is the complement of an algebraic set in $\mathfrak{g}$.  In particular, it is a dense open set.

\begin{definition}\index{Regular element}
An element $x\in\mathfrak{g}$ is called {\em regular} if it is $\ad$-regular.  The set of regular elements of $\mathfrak{g}$ is denoted $\mathfrak{g}'$.  It is a Zariski open subset of $\mathfrak{g}$.
\end{definition}

\subsection{Cartan subalgebras}
Let $\mathfrak{g}$ be a complex reductive Lie algebra.  A subalgebra $\mathfrak{h}\subset\mathfrak{g}$ is a Cartan subalgebra if it satisfies any of the following equivalent conditions
\begin{enumerate}
\item $\mathfrak{h}$ is a nilpotent Lie algebra such that $N_{\mathfrak{g}}(\mathfrak{h})=\mathfrak{h}$.
\item $\mathfrak{h}=C_{\mathfrak{g}}(x)$ where $x\in\mathfrak{g}'$ is a regular element.  
\end{enumerate}
If $\mathfrak{g}$ is also semisimple, then these are equivalent to
\begin{enumerate}[(3)]
\item $\mathfrak{h}$ is a maximal abelian subalgebra, and $\ad(H)$ is diagonalizable for all $H\in\mathfrak{h}$.
\end{enumerate}

The conjugacy theorem, which we shall not prove (see Humphreys, {\em Introduction to Lie algebras and representation theory}) is
\begin{theorem}
Any two Cartan subalgebras are conjugate with respect to $\operatorname{Int}(\mathfrak{g}).$
\end{theorem}

\subsection{Flat subspaces}\index{Flat subspace $\mathfrak{a}_o$}  Let $(\mathfrak{g}_o,s)$ be an orthogonal symmetric Lie algebra.  A {\em flat subspace} is a maximal abelian subalgebra of $\mathfrak{p}_o$.  The terminology is justified by the following situation.  Let $(G,K,\sigma)$ be a Riemannian symmetric space with $\Lie(G)=\mathfrak{g}_o$, $\Lie(K)=\mathfrak{k}_o$, and $\sigma_{*,e}=s_o$.  The exponential map defines a local diffeomorphism $\Exp_e : \mathfrak{p}_o\to G/K$.  This local mapping sends every flat subspace of $\mathfrak{p}_o$ to a maximal totally geodesic submanifold on which the curvature tensor vanishes.

\begin{lemma}
Let $\mathfrak{a}_o$ be a flat subspace.  Then $\mathfrak{a}_o$ acts semisimply on $\mathfrak{g}_o$.
\end{lemma}
\begin{proof}
By the decomposition Theorem \ref{decomposition}, it is enough to consider the case where $\mathfrak{g}_o=[\mathfrak{p}_o,\mathfrak{p}_o]\oplus\mathfrak{p}_o$.  By duality, we can take $\mathfrak{g}_o$ to be compact.  But then $\ad x$ is semisimple with imaginary eigenvalues for any $x\in\mathfrak{g}_o$.
\end{proof}

The following theorem characterizes flat subspaces:

\begin{theorem}\label{RegularElementsFlatSubspaces}
Let $\mathfrak{g}_o$ be a reductive orthogonal symmetric Lie algebra.  Then $\mathfrak{a}_o\subset\mathfrak{p}_o$ is a flat subspace if and only if $\mathfrak{a}_o=C_{\mathfrak{p}_o}(x)$ for some regular element $x\in\mathfrak{p}_o\cap\mathfrak{g}_o'$.
\end{theorem}

For this reason, flat subspaces are often known in the literature as Cartan subalgebras of $\mathfrak{p}_o$.\footnote{Cf. Borel, {\em Semisimple groups and Riemannian symmetric spaces}, 1998, \S II.2.}  They are also known as {\em maximal} flat subspaces, but we shall not be interested in any flat subspaces that are not maximal.

\begin{proof}
The proof requires the following additional notation.  If $\mathfrak{h}\subset\mathfrak{g}_o$ is a subalgebra, define the nil radical of $\mathfrak{h}$ by
$$\mathfrak{r}_{\mathfrak{g}_o}(\mathfrak{h}) = \{ x\in\mathfrak{g} \mid (\ad\mathfrak{h})^\infty x=0 \}.$$
Note that in the special case where $\ad\mathfrak{h}$ acts semisimply, $\mathfrak{r}_{\mathfrak{g}_o}(\mathfrak{h})=C_{\mathfrak{g}_o}(\mathfrak{h})$.

Let $\mathfrak{a}_o$ be a flat subspace.  Let $\mathfrak{h}_o$ be a Cartan subalgebra of $C_{\mathfrak{g}_o}(\mathfrak{a}_o)$.  We claim that $\mathfrak{h}_o$ is a Cartan subalgebra of $\mathfrak{g}_o$ as well.  Indeed, since $\ad_{\mathfrak{g}_o}\mathfrak{a}_o$ acts semisimply,
$$N_{\mathfrak{g}_o}(\mathfrak{h}_o) \subset \mathfrak{r}_{\mathfrak{g}_o}(\mathfrak{h}_o) \subset \mathfrak{r}_{\mathfrak{g}_o}(\mathfrak{a}_o) = C_{\mathfrak{g}_o}(\mathfrak{a}_o).$$
Hence
$$N_{\mathfrak{g}_o}(\mathfrak{h}_o) = N_{\mathfrak{g}_o}(\mathfrak{h}_o)\cap C_{\mathfrak{g}_o}(\mathfrak{a}_o) = N_{C_{\mathfrak{g}_o}(\mathfrak{a}_o)}(\mathfrak{h}_o) = \mathfrak{h}_o$$
by hypothesis.  Hence $\mathfrak{h}_o$ is a Cartan subalgebra of $\mathfrak{g}_o$.

Now, since $\mathfrak{a}_o$ acts semisimply on $\mathfrak{g}_o$, there is an $\ad_{\mathfrak{g}_o}\mathfrak{a}_o$-regular element $x\in\mathfrak{a}_o$.  So $\mathfrak{h}\subset C_{\mathfrak{g}_o}(\mathfrak{a}_o) = C_{\mathfrak{g}_o}(x)$.  But the inclusion $C_{\mathfrak{g}_o}(x)\subset\mathfrak{h}$ must also hold, and so the two are equal.  Thus $x\in\mathfrak{g}_o'$.

Conversely, if $\mathfrak{a}_o=C_{\mathfrak{p}_o}(x)$ for $x$ a regular element, then $\mathfrak{a}_o$ is maximal abelian, for
\begin{align*}
C_{\mathfrak{g}_o}(x)&=C_{\mathfrak{k}_o}(x)\oplus C_{\mathfrak{p}_o}(x) =C_{\mathfrak{k}_o}(x)\oplus \mathfrak{a}_o\\
C_{\mathfrak{g}_o}(\mathfrak{a}_o)&=C_{\mathfrak{k}_o}(\mathfrak{a}_o)\oplus C_{\mathfrak{p}_o}(\mathfrak{a}_o).
\end{align*}
But these are both equal for a Cartan subalgebra, and so $ \mathfrak{a}_o=C_{\mathfrak{p}_o}(\mathfrak{a}_o)$, and $\mathfrak{a}_o$ is maximal abelian.
\end{proof}

Any two flat subspaces are conjugate via an automorphism from $K$.  More precisely,

\begin{proposition}
Let $\mathfrak{g}$ be orthogonal symmetric and $\mathfrak{a}_1,\mathfrak{a}_2$ flat subspaces.  Let $K\subset \operatorname{GL}(\mathfrak{p}_o)$ be the connected Lie subgroup with Lie algebra $\mathfrak{k}_o$: $K=\left\langle e^{\ad X}\mid X\in\mathfrak{k}_o\right\rangle$.  Then there exists a $\phi\in K$ such that $\mathfrak{a}_2=\phi(\mathfrak{a}_1)$.
\end{proposition}

\begin{proof}
Write $\mathfrak{a}_1=C_{\mathfrak{p}_o}(X_1)$ and $\mathfrak{a}_2=C_{\mathfrak{p}_o}(X_2)$ with $X_1,X_2\in \mathfrak{p}_o\cap\mathfrak{g}_o'$.  Define a mapping $K\to\mathbb{R}$ by
$$k \mapsto B_{\mathfrak{g}_o}(X_1,\Ad(k)X_2).$$
Since $K$ is compact, this has an extremum at some $k\in K$.  Fix such a $k$, and then consider for a fixed $Y\in\mathfrak{k}_o$,
$$\phi(t) = B_{\mathfrak{g}_o}(X_1,e^{t\ad Y}\Ad(k)X_2).$$
Then $\phi'(0)=0$, so
$$0 = B_{\mathfrak{g}_o}(X_1,[Y,\Ad(k)X_2]) = -B_{\mathfrak{g}_o}(Y,[X_1,\Ad(k)X_2])$$
by invariance.  Since this holds for all $Y\in\mathfrak{k}_o$, non-degeneracy of $B_{\mathfrak{g}_o}|_{\mathfrak{k}_o\times\mathfrak{k}_o}$ implies that $[X_1,\Ad(k)X_1]=0$.  So
$$\mathfrak{a}_1 = C_{\mathfrak{p}_o}(X_1) = C_{\mathfrak{p}_o}(\Ad(k)X_2) = \Ad(k)\mathfrak{a}_2.$$
\end{proof}

\begin{corollary}\label{coverbyflats}
Let $\mathfrak{a}_o\subset\mathfrak{p}_o$ be a flat subspace.  Then
$$\mathfrak{p}_o = \bigcup_{k\in K} \Ad(k)\mathfrak{a}_o.$$
\end{corollary}

\section{Symmetric spaces from orthogonal symmetric Lie algebras}
Let $(\mathfrak{g}_o,s_o)$ be an orthogonal symmetric Lie algebra.  Write
$$\mathfrak{g}_o=\mathfrak{k}_o\oplus\mathfrak{p}_o.$$
Let $G$ be a connected Lie group with $\Lie(G)=\mathfrak{g}_o$.  Assume for now that
\begin{enumerate}
\item There exists an involution $s:G\to G$ such that $s_{*,e}=s_o$.  Let $K\le G$ be such that $\Lie(K)=\mathfrak{k}_o$; i.e., $G_{s,e} \le K\le G_s=\{g\mid s(g)=g\}$.  Note that because $G_{s,e}$ is compact, it is closed in $G$.
\item $K\le G$ is closed.
\end{enumerate}
(Assumption (1) holds for instance when $G$ is simply connected; assumption (2) always holds, see below.)

It follows that $G/K$ is a manifold.  Let $o=eK$ be the coset of the identity in $G/K$.  Under the submersion $G\to G/K$, $\mathfrak{p}_o\xrightarrow{\cong} T(G/K)_o$.  The form $Q$ on $\mathfrak{p}_o$ is an $\ad_{\mathfrak{g}_o}$-invariant bilinear form.  This induces  corresponding form on $T(G/K)_o$.  Extend this to a $G$-invariant Riemannian form on $G/K$.  In this way $G/K$ becomes a Riemannian manifold.

Since $G$ acts by isometries on $G/K$, there is a Lie group morphism $\pi : G\to I(G/K)$, $\phi(g)(xK) = gxK$.  Moreover, restricting to $K$, we have
$$\pi:K\to I(G/K)_o = \{ \phi\in I(G/K)\mid \phi(o)=o\},$$
the stabilizer of $o$.  The involution $s:G\to G$ is compatible with the quotient in the sense that we may define $s:G/K\to G/K$ by $s(gK)=s(g)K$: this is well-defined since $K$ is fixed by $s$.  Thus the diagram
$$\xymatrix{
G \ar[r]\ar[d]^{\cong}_{s} & G/K\ar@{-->}^{\cong}_s[d]\\
G \ar[r]&G/K
}$$
commutes: $s\pi(g) = \pi(s(g)).$

Now let $g\in G$ be fixed.  Let $\tau(g):G/K\to G/K$ be the right translation by $G$.  We have
$$\xymatrix{
T_{gK}G/K \ar[r]^{s_{*,gK}}\ar[d]^{\tau(g^{-1})}&T_{s(g)K}G/K\ar[d]^{\tau(s(g)^{-1})}\\
T_oG/K\ar[r]^{s_*,o}& T_oG/K
}$$
Hence $s_{*,gK}$ is an isometry, and so $s\in I(G/K)_o$.  The geodesic symmetry at $gK$ is given by $s_{gK} = \tau(g)\circ s\circ\tau(g^{-1})$.  Hence $G/K$ is a Riemannian globally symmetric space.

Write $\Lie(I(G/K))=\widetilde{\mathfrak g}_o$ and $\Lie(I(G/K)_o)=\widetilde{\mathfrak k}_o$.  The involution $s$ on $G/K$ defines an involution on $I(G/K)$ by conjugation: $c(s)\phi = s\phi s^{-1}$ for $\phi\in I(G/K)$.  Then $(\widetilde{g}_o,c(s)_{*,o})$ is an orthogonal symmetric Lie algebra.  Decompose it as
$$\widetilde{\mathfrak g}_o = \widetilde{\mathfrak k}_o\oplus\widetilde{\mathfrak{p}}_o.$$
Then there is a natural isomorphism of $G$ modules of $\widetilde{\mathfrak p}_o$ with $\mathfrak{p}_o$, since both of them are isomorphic to $T_oG/K$.  Because of the map $G\to I(G/K)$ there is a mapping of Lie algebras $\mathfrak{g}_o\to\widetilde{\mathfrak g}_o$, which may or may not be surjective.\footnote{For example, if $G=\mathbb{R}^n$ and $K=\{e\}$, then $I(G/K)$ is the Euclidean group, and the mapping is not surjective.}

\subsection{Transvections and the exponential map} The transvections of $G/K$ are of the form $s_{gK}\circ s$.  This is the transvection along the geodesic from $o$ to $g^2K$ passing through $gK$.   Let $\exp : \mathfrak{p}_o\to G$ be the Lie algebra exponential.  Since for any $g\in G$, we have $s_{gK}=\tau(g)s\tau(g^{-1})$, the isometry $\tau(\exp X)s \tau(\exp(-X)) s$ is the transvection sending $o$ to $\exp (2X)$.  Now, $\exp:\mathfrak{p}_o\to G\to G/K$ is a local diffeomorphism onto a neighborhood of $o$, so the geodesics through $o$ are all of the form $t\mapsto \exp(tX)K$ for some $X\in\mathfrak{p}_o$.  As a result, the Lie algebra exponential and Riemannian exponential coincide, modulo the identification of $T_oG/K$ with $\mathfrak{p}_o$: $\exp = \Exp$.  Note that since $G/K$ is complete, $\Exp : T_oG/K\to G/K$ is surjective onto the connected component of $o$, and so $\exp : \mathfrak{p}_o\to G\to G/K$ is as well.

In particular, as a result ee have $G=K\exp(\mathfrak{p}_o)=\exp(\mathfrak{p}_o)K = K\exp(\mathfrak{p}_o)K$.  In terms of this decomposition, we must have
$$s(k\exp(X)) = k\exp(-X).$$

\subsection{$K$ is closed}  In these next paragraphs, we dispense with assumptions (1) and (2) above.  First, we deal with assumption (2), that $K$ was assumed to be closed.

Let $\mathfrak{a}_o\subset\mathfrak{p}_o$ be a flat subspace.  Then $\mathfrak{p}_o=\bigcup_{k\in K}\Ad(k)\mathfrak{a}_o$.  This holds in particular for $K=G_{s,e}$.  Take $K_e\le G_{s,e}\le K\le G_s$.  We argue that $|K/K_e| < \infty$, so that $K$ is closed (i.e., the assumption (2) above is always satisfied.)  Indeed, let $x\in K\le G=\exp(\mathfrak{p}_o)K_e$ so that it is possible to write  $xK_e = \exp Z K_e$ for some $Z\in\mathfrak{p}_o$.  The involution $s$ is the identity on $K$ and coincides with inversion on $\exp\mathfrak{p}_o$, so that since $\exp Z\in K\cap \exp\mathfrak{p}_o$, it follows that $\exp(Z)^2=e$.

Write $Z=\Ad(k)W$ for some $k\in K_e$ and $W\in\mathfrak{a}_o$.  Then
$$xK_e = \exp(Z)K_e = k\exp(W) K_e = \exp(W)\underbrace{(\exp(-W)k\exp(W))}_{\in\, K_e}K_e.$$
It follows that $\exp(W)^2=e$ as well, and so $\exp(W)\in K$.  Since $W\in\mathfrak{p}_o$ already, $\exp(W)\in K\cap \exp\mathfrak{p}_o$.  To summarize then
$$|K/K_e| \le |\{ g\in \exp\mathfrak{a}_o\mid g^2=e\}|.$$
On the other hand, $\exp\mathfrak{a}_o$ is a closed subgroup of $G$, since $\operatorname{cl}\exp\mathfrak{a}_o$ is a connected abelian subgroup (a torus) whose Lie algebra must be equal to $\mathfrak{a}_o$ (by maximality).  But the set of elements of order two in a connected abelian Lie group is finite.  Thus $|K/K_e|<\infty$, proving that $K$ is closed.

\subsection{Structure of Riemannian symmetric spaces}  Suppose that we have a Riemannian symmetric pair $(G,s)$.  Then $\operatorname{Lie}(I(G/K))=\widetilde{\mathfrak{g}}_o$ and $\Lie(I(G/K)_o) = \widetilde{\mathfrak{k}}_o$ and there is a map
$$\mathfrak{g}_o\to \widetilde{g}_o$$
such that the restriction to $\mathfrak{p}_o$ is an isomorphism with $\widetilde{\mathfrak{p}}_o$.  If $(\mathfrak{g}_o,s)$ is semisimple and irreducible, then $\mathfrak{k}_o$ is simple and so $\mathfrak{g}_o\hookrightarrow\widetilde{\mathfrak{g}}_o$.  But if $\mathfrak{g}_o$ is semisimple, then it is maximal, and so in that case $\mathfrak{g}_o=\widetilde{\mathfrak{g}}_o$.

Assumption (1) above holds when $G$ is simply-connected.  In that case, we claim that $G/G_{s,e}$ is simply-connected as well.  Indeed, if $\gamma$ is a loop in $G/G_{s,e}$ based at $o$, then we can lift it to a curve $\tilde{\gamma}$ in $G$ from $e$ to some point $x\in G_{s,e}$.  Since $G_{s,e}$ is connected, there is a path $\tilde{\mu}$ lying entirely inside $G_{s,e}$ connecting $x$ back to $e$.  The compositum $\tilde{\gamma}\cup\tilde{\mu}$ is a loop in $G$.  Since $G$ is simply connected, this loop is contractible, and so its image $\gamma$ under $G\to G/G_{s,e}$ is also contractible.  Consequently, $G/G_{s,e}$ is a simply connected Riemannian globally symmetric space.

\begin{theorem}
Let $M$ be a simply-connected Riemannian globally symmetric space.  Then there is a unique decomposition (up to ordering)
$$M=M_f \times M_1\times\cdots\times M_t$$
with $M_f$ flat and $M_i$ simply-connected irreducible globally Riemannian symmetric with a semisimple group of isometries.  The sectional curvature\footnote{Given by $K(X,Y)=\frac{1}{c}B([X,Y],[X,Y])$ where $c$ is the constant of proportionality relating the invariant form on $\mathfrak{p}_o$ to the Killing form on $\mathfrak{g}_o$: $B(X,Y)=cQ(X,Y)$.} of $M_i$ is positive if $M_i$ is compact, negative if $M_i$ is non-compact.
\end{theorem}

\begin{proof}
Let $\mathfrak{g}_o=\operatorname{Lie}(I(M))$.   Conjugation by $\sigma$ defines the involution on $\mathfrak{g}_o$ in terms of which it becomes an orthogonal symmetric Lie algebra.  There is the decomposition
$$(\mathfrak{g}_o,s_o) = (\mathfrak{h}_f,s_f)\oplus (\mathfrak{h}_1,s_1)\oplus\cdots\oplus  (\mathfrak{h}_t,s_t).$$
Associate to each $(\mathfrak{h}_i,s_i)$ its simply connected globally Riemannian symmetric space $M_i$.  Then $M$ is locally isometric to $M_f\times M_1\times\cdots\times M_t$, and therefore is globally isometric, since they are both simply-connected globally Riemannian symmetric spaces.
\end{proof}

{\em Exercise.}  If $M$ is a semisimple irreducible Riemannian symmetric space, show that every affine diffeomorphism is an isometry, $I(M)=\operatorname{Aff}(M)$.

{\em Remarks:}
\begin{enumerate}
\item Let $(\mathfrak{g}_o,s_o)$ be an orthogonal symmetric Lie algebra and $G$ a connected Lie group $\Lie(G)=\mathfrak{g}_o$.  Suppose that there exists an involution $s:G\to G$ lifting $s_o$; i.e., such that $s_{*,e}=s_o$.  If $K$ is an open subgroup of $G_s$, then
$$\widetilde{G/K} = \widetilde{G}/\widetilde{K} \to G/K$$
is a local isometry of Riemannian globally symmetric spaces, and $G/K$ decomposes in the same way as $\widetilde{G/K}$, except the factors are replaced by the appropriate non-simply connected quotient factors.
\item  If there is no lift to $s:G\to G$, but there is a closed subgroup $K\subset G$ with $\Lie(K)=\mathfrak{k}_o$, then $G/K$ and $\widetilde{G}/\widetilde{G}_{s,e}$ are locally isometric and $G/K$ becomes a locally isometric space.  In that case, it only decomposes locally in a neighborhood of each point.
\item Suppose that $M$ is a Riemannian manifold and $\widetilde{M}\xrightarrow{\pi} M$ is the universal cover.  The group of covering transformations $\Gamma$ is the subgroup of $I(\widetilde{M})$ that cover the identity diffeomorphism of $M$.  This is a discrete group that acts freely and properly discontinuously on $\widetilde{M}$, so there is a canonical identification $M=\widetilde{M}/\Gamma$.  Suppose now that $M,\widetilde{M}$ are globally symmetric, with symmetries $s,\tilde{s}$.  (It follows by rigidity that $\pi\circ\tilde{s}=s\circ\pi$.)  Let $\gamma\in \Gamma$ and $\tilde{m}\in\widetilde{M}$.  Then
\begin{align*}
\pi(\tilde{s}(\gamma.\tilde{m}) &= s(\pi(\gamma,\tilde{m})) = s(\pi(\tilde{m}))\\
&=\pi(\tilde{s}(\tilde{m}))
\end{align*}
so $\tilde{s}(\gamma.\tilde{m})$ and $\tilde{s}(\tilde{m})$ are in the same fiber.  Because covering transformations act transitively on the fiber, there exists a $\gamma_{\tilde{m}}\in\Gamma$ such that $\tilde{s}\gamma\tilde{s}.\tilde{m} = \gamma_{\tilde{m}}.\tilde{m}$.  This establishes a continuous map $\tilde{m}\mapsto \gamma_{\tilde{m}}$:
$$\widetilde{M}\to \Gamma.$$
But $\Gamma$ is discrete, so this map is constant: $\gamma_{\tilde{m}}$ does not depend on $\tilde{m}$.  Thus $\Gamma$ is normalized by the group generated by all local symmetries.  If $\widetilde{M}$ is semisimple, then $I(\widetilde{M})_e\le \{\text{group generated by local symmetries}\}$, so $I(\widetilde{M})_e$ normalizes $\Gamma$.  For each $\gamma\in\Gamma$, the conjugation map
\begin{align*}
\phi &\mapsto \phi\gamma\phi^{-1}\\
I(\widetilde{M})_e&\to\Gamma
\end{align*}
is a continuous map from a connected space to a discrete space, and must therefore be constant.  Hence $\Gamma$ must be central: $\Gamma\le C_{I(\widetilde{M})}(I(\widetilde{M})_e)$.

Conversely, if $\Gamma\le C_{I(\widetilde{M})}(I(\widetilde{M})_e)$ is discrete, then it acts freely and properly discontinuously, and $\widetilde{M}/\Gamma$ is a Riemannian globally symmetric space.
\end{enumerate}

\section{Cartan immersion}\index{Cartan immersion}

\begin{lemma}
Let $(\mathfrak{g}_o,s_o)$ be an orthogonal symmetric Lie algebra.  Let $G$ be a connected Lie group with $\Lie(G)=\mathfrak{g}_o$ and let $s:G\to G$ be an involution such that $s_{*,e}=s_o$, and let $K$ be an open subgroup of $G_s$.  Then $\exp\mathfrak{p}_o$ is closed in $G$.  If $\mathfrak{a}_o$ is a flat subspace of $\mathfrak{p}_o$, then $\exp\mathfrak{a}_o$ is closed in $G$ as well.
\end{lemma}

\begin{proof}
Let $g_n$ be a Cauchy sequence in $\exp\mathfrak{p}_o$, say $g_n=\exp X_n$, such that $g_n\to g$ in $G$.  Then $\|X_n\|=d(o,g_n\cdot o)$.  Since $\|X_n\|$ is bounded, it has a subsequence that converges in $\mathfrak{p}_o$ to an element $X$.   By continuity of the exponential, $g=\exp X$.  For the second assertion, $\overline{\exp\mathfrak{a}_o}$ is a connected closed abelian subroup of $\exp\mathfrak{p}_o$.  So its Lie algebra $\mathfrak{a}_1$ is abelian and satisfies $\mathfrak{a}_o\subset\mathfrak{a}_1\subset\mathfrak{p}_o$.  But by maximality, $\mathfrak{a}_o=\mathfrak{a}_1$, and hence $\overline{\exp\mathfrak{a}_o}=\exp\mathfrak{a}_o$, as required.
\end{proof}

Under the same hypotheses as the Lemma, the mapping $G\times G\to G$, $(x,g)\mapsto x\cdot g := xgs(x)^{-1}$ is a continuous group action.  Let
$$P=G\cdot e = \{ xs(x)^{-1} \mid x\in G\}.$$
Since any $x\in G$ can be written $x=\exp(X)k$ for some $X\in\mathfrak{p}_o$ and $k\in K$, relative to this decomposition we have $s(x) = \exp(-X)k$ and in particular $xs(x)^{-1} = \exp(2X)$.  Thus
$$P = \{\exp(2X)\mid X\in\mathfrak{p}_o\} = \exp\mathfrak{p}_o$$
But since this is also an image of the map $x\mapsto xs(x)^{-1}$, it is a closed subset of $G$.  Let $G_s$  be the set of fixed points of $s$.  Then
$$G/G_s \xrightarrow{\cong} P = \exp \mathfrak{p}_o$$
is a diffeomorphism.  Note that $G/G_s$ is the symmetric space and $P$ is a closed submanifold of $G$.  The inclusion of the symmetric space into $G$ in this manner is called the {\em Cartan immersion}.  

The geodesics are of the form $t\mapsto\exp(2tX)$ for $X\in\mathfrak{p}_o$.  If $\mathfrak{a}_o$ is a flat subspace of $\mathfrak{p}_o$, then $\exp\mathfrak{a}_o$ is a closed submanifold of $P$.  It is flat, since the bracket, which defines the curvature, vanishes on $\mathfrak{a}_o$.  It is totally geodesic since the transvection from a point $A_1 = \exp (a_1)$ to $A_2=\exp(a_2)$ is the one parameter group $\exp(t(a_2-a_1))$ which remains tangent to $\exp\mathfrak{a}_o$ for all $t$.

\chapter{Examples}


\section{Flat examples}
\subsection{$\mathbb{R}^n$ with the standard Euclidean metric $\sum_i dx^i\otimes dx^i$} The geodesics are straight lines.  The symmetry at the point $0$ is $s_0(x) = -x$.  The geodesic transvections are the translations, and since these commute with each other, the symmetric space is flat.
\subsection{$\mathbb{R}_+\times\mathbb{R}_+$ with the metric $\frac{dx^2}{x^2}+\frac{dy^2}{y^2}$}  The geodesics are $t\mapsto (e^{ta},e^{tb})$ and so the symmetry at the point $(1,1)$ is $s_{(1,1)}(x,y) = \left(\frac{1}{x},\frac{1}{y}\right)$.  The geodesic transvections are $(x,y)\mapsto (e^{ta},ye^{tb})$.  The symmetric space is flat, since the geodesic transvections commute.  Note that this is symmetric space is isometric to $\mathbb{R}^2$ via $x\mapsto\log x$ and $y\mapsto\log y$.
\subsection{$\mathbb{C}^*$ with the metric $\frac{|dz|^2}{|z|^2}$}  The geodesics are the logarithmic spirals $e^{t(a+ib)}$.  The symmetry at the point $z=1$ is $s_1(z) = \frac{1}{z}$.  The geodesic transvections are $z\mapsto e^{t(a+ib)}$.  These commute, and therefore the space is flat.  Note that $\mathbb{C}(=\mathbb{R}^2)$ is the universal cover of this symmetric space via $z\mapsto e^z$.
\subsection{$S^1$ with the metric $d\theta^2$.}  The isometry group is $I(S^1)=O(2)$ with identity component $I(S^1)_e = SO(2) = U(1) = S^1$ (the circle group).  The geodesic inversion at the point $z=1$ is $z\mapsto -\overline{z}$.  In terms of the group $SO(2)$, this is $\sigma_1:A\mapsto A^{-1}$.  On the Lie algbra,
\begin{align*}
\mathfrak{so}(2) &\xrightarrow{s_1} \mathfrak{so}(2) \\
A &\mapsto -A.
\end{align*}

\subsection{$\mathbb{T}^2$}
The flat torus $\mathbb{T}^2=S^1\times S^1$ is a symmetric space.  Regard $\mathbb{T}^2$ as the quotient space $\mathbb{R}^2/\mathbb{Z}^2$.  The automorphisms of $\mathbb{T}^2$ fixing the point $o$ lift to isometries of $\mathbb{R}^2$ that fix the origin and stabilize the lattice $\mathbb{Z}^2$.  Thus such automorphisms lie within $SL(2,\mathbb{Z})\cap O(2) = \{\pm\operatorname{id}\}$.  Hence the symmetry at $o$ is $\sigma_o(x)=-x$, and likewise the symmetry through any other point is $\sigma_y(x)=y-x$.  The connected group of automorphisms of the torus consists entirely of transvections, so as a symmetric space we have $\mathfrak{g}_o=\mathfrak{p}_o$ is the two-dimensional abelian Lie algebra.

\section{Spheres}
\subsection{$S^2$ with the induced metric from $\mathbb{R}^3$.}   Let $o\in S^2$ be the north pole $(0,0,1)$.  The involution at $o$ is the rotation through an angle of $\pi$ about the axis through $o$, since this reverses the direction of all geodesics through $o$.  This is the reflection in the $o$-axis, given in terms of the Euclidean dot product by $x\mapsto 2(x\cdot o)o-x$.  The isometry group of the sphere is $I(S^2)=O(3)$, with identity component $SO(3)$, and the stabilizer of $o$ is a subgroup isomorphic to $SO(2)$.  The Lie algebra of $SO(3)$ is
$$\mathfrak{g}_o = \mathfrak{so}(3) = \left\{\begin{pmatrix}A & \xi\\ -\xi^t & 0 \end{pmatrix}\mid A\in \mathfrak{so}(2), \xi\in\mathbb{R}^2\right\}$$
The isotropy algebra of $o$ is
$$\mathfrak{k}_o = \left\{\begin{pmatrix}A & 0\\ 0 & 0 \end{pmatrix}\mid A\in \mathfrak{so}(2), \xi\in\mathbb{R}^2\right\}.$$
Hence, the involution in the Lie algebra is
$$\begin{pmatrix}A & \xi\\ -\xi^t & 0 \end{pmatrix} \mapsto\begin{pmatrix}A & -\xi\\ \xi^t & 0 \end{pmatrix}.$$
The $-1$ eigenspace is therefore
$$\mathfrak{p}_o =  \left\{\begin{pmatrix}0 & \xi\\ -\xi^t & 0 \end{pmatrix}\mid \xi\in\mathbb{R}^2\right\}.$$

To describe the involution at the group level, it is convenient to work in the spin group $SU(2)$.  This is the group of $2\times 2$ complex matrices satisfying $A^*A=I$; or equivalently the group of unit quaternions $Sp(1)$ (see below).  The Lie algebra $\mathfrak{g}$ of $SU(2)$ is the space of trace-free antihermitian matrices.  Then $SU(2)$ acts on $\mathfrak{g}$ by the adjoint representation
$$\Ad(x) : H\mapsto  xvx^*, \quad x\in SU(2), \ v\in\mathfrak{g}.$$
The Killing form on $\mathfrak{g}$ induces the Euclidean norm $\|v\| = \sqrt{\det v}$, in terms of which $\mathfrak{g}$ is identified with the standard three-dimensional Euclidean space $\mathbb{R}^3$.  Since $\Ad(SU(2))$ is a connected $3$-dimensional Lie group preserving the Euclidean metric on $\mathfrak{g}$, $\Ad(SU(2))=SO(3)$.  Thus $\Ad:SU(2)\to SO(3)$ is a cover of $SO(3)$.

The involution of $SU(2)$ is defined by complex conjugation of matrices $\sigma_o:x\mapsto \bar{x}$.  The fixed points in $SU(2)$ are the real matrices of $SO(2)\subset SU(2)$.  However, $\Ad(\sigma x)y = \Ad(x)y$ for all $y\in\mathfrak{g}$ if and only if $\sigma x = \pm x$.  Thus, either $x\in SO(2)$ or $x$ is of the form
$$x = \begin{pmatrix} iu& iv\\ iv& -iu\end{pmatrix}, \quad u,v\in\mathbb{R},\quad u^2+v^2=1.$$ 
So the set $SO(3)_{\sigma_o}$ of elements of $SO(3)$ fixed by the involution $\sigma_o$ has two connected components.

Note the familiar fact that $S^2\cong\mathbb{CP}^1$ carries a complex structure (and is a K\"{a}hler manifold).  This can be seen in a variety of ways from the symmetric space structure:
\begin{itemize}
\item It is easy to show that if a manifold of dimension $2n$ has restricted holonomy in $U(n)$, then it is K\"{a}hler, and in particular is a complex manifold (see Kobayashi \& Nomizu, Vol. 2).\footnote{As a result, any orientable surface can be given a complex structure that makes it K\"{a}hler.}  In this case, the holonomy algebra is $[\mathfrak{p}_o,\mathfrak{p}_o]\subset \mathfrak{k}_o$ and the isotropy group is $SO(2)\cong U(1)$, so $S^2$ is K\"{a}hler.
\item Define
$$J_o = \begin{pmatrix}0&1\\-1&0\end{pmatrix} \in SO(2) = K.$$
This is a complex structure in the tangent space at $o$, meaning that $J_o^2=-\operatorname{id}$, and $J_o$ is metric-preserving (being a member of $K$).  Use transvections to transport the complex structure $J$ to every other point of $S^2$.  Since $K$ is abelian, $J_o$ commutes with $K$, and so $J$ is invariant under $SU(2)$.  With this complex structure, the metric on $S^2$ is K\"{a}hler, and so $S^2$ becomes a complex manifold.
\item $SO(2)\subset SU(2)$ is a maximal torus, and any two maximal tori are conjugate.  Therefore $SU(2)/SO(2)\cong SU(2)/U(1)$ by an inner automorphism of $SU(2)$.  But in the latter homogeneous space, $SU(2)$ acts on $\mathbb{C}^2$ by unitary transformations and $U(1)$ is the stabilizer of a complex line.  But $SU(2)$ is transitive on the set of complex lines (one-dimensional subspaces), and so $SU(2)/U(1)$ is identified with the space of complex lines $\mathbb{CP}^1$.
\end{itemize}

The significant point here in the second item above is that  the complex structure lies in the center of $K$ (see Chapter \ref{HermitianSymmetricSpaces}).  For irreducible symmetric spaces, the center is one-dimensional, and the complex structure is effectively unique (the only ambiguity is in replacing $J$ by $-J$).  We note also that the involution on $\mathfrak{su}(2)$ can be written instead as conjugation by $J_o$:
$$\sigma_o(x) = J_oxJ_o^{-1}.$$
This is also a general feature of the Hermitian symmetric spaces, that the symmetry is conjugation by the complex structure, that is {\em a fortiori} an inner automorphism.

\subsection{$S^3$ with the induced metric.}
The unit sphere $S^3\subset\mathbb{R}^4$ naturally carries a group structure induced by its realization as the group of unit quaternions $Sp(1)$.  Let $\mathbb{H}$ denote the quaternions.  This carries a real Euclidean inner product $\langle p, q\rangle = \operatorname{re}(p\overline{q})$.  The set of unit quaternions are the solutions of the equation $p\overline{p}=1$.  This is a group, since $(pq)\overline(pq) = pq\overline{q}\overline{p} = 1$ if $p,q\in Sp(1)$.  The group $Sp(1)\times Sp(1)$ acts on $\mathbb{H}$ via
$$\rho(p,q)x = px\overline{q}.$$
Observe that $\rho : Sp(1)\times Sp(1)\to SO(4)$.  Both groups are compact connected and have the same dimension, and $\ker\rho = \{(1,1),(-1,-1)\}$.  So $\rho$ is a double cover.

The involution on $Sp(1)\times Sp(1)$ is $\sigma_o:(p,q)\mapsto (q,p)$.  The group of fixed points for the involution is $K_o=\Delta(Sp(1))$, the diagonal subgroup.  This descends under the quotient to the involution $s_o(p)=\overline{p}=p^{-1}$ on $S^3$ itself.  There is a natural section of the quotient map $Sp(1)\times Sp(1)\to (Sp(1)\times Sp(1))/\Delta(Sp(1)) = S^3$ given by $\tau : p\mapsto (p,p^{-1})$.  Note $\tau\circ s_o = \sigma_o\circ\tau$.

\subsection{$S^4$ with the induced metric.}
The group $SO(5)$ is covered by the symplectic group
$$Sp(2) = \left\{A\in GL(2,\mathbb{H})\mid A^*A=I\right\}$$
where $*$ is the quaternionic conjugate-transpose.  This is the compact real form of the symplectic group $Sp(4,\mathbb{C})$.  With respect to the natural action on $\mathbb{H}^2$, the orbit of the vector $(1\ \ 0)^T$ under the action of $\operatorname{Sp}(2)$ is the unit sphere $S^7$, and the stabilizer of  $(1\ \ 0)^T$ is $SU(2)\cong S^3$. So topologically, $\operatorname{Sp}(2)$ is an $S^3$ bundle over $S^7$, and therefore is simply connected by the long exact homotopy sequence.

The Lie algebra is
$$\mathfrak{sp}(2) = \left\{\begin{pmatrix}x&z\\-\bar{z} & y\end{pmatrix} \mid x,y\in\im\mathbb{H}, z\in\mathbb{H}\right\}.$$
The adjoint representation of $Sp(2)$ on $\mathfrak{sp}(2)$ leaves invariant the positive-definite quadratic form $Q(X) = \tr(X^*X)$.  Thus $Sp(2)\to SO(5)$.  This mapping is a double cover as well, with kernel $\pm I\in Sp(2)$.

The involution is conjugation by the matrix $S=\begin{pmatrix}-1&0\\ 0&1\end{pmatrix}$.  Note that this is not an inner involution.

{\em Digression: The Grassmannian of Lagrangian $2$-planes in $\mathbb{C}^4$.}  Note that one family of involutions on $Sp(2)$ are given by the quaternionic conjugation maps $\sigma_o(x) = ux\overline{u}$ where $u$ is a fixed imaginary unit quaternion.  Any involution of this form gives a complex structure on the quotient.

For definiteness, there is no loss of generality in taking $u=j$.  Multiplication by $j$ induces a complex structure on $\mathbb{H}^2$; so we shall identify $\mathbb{H}^2$ with $\mathbb{C}^4$ relative to this complex structure.  The isotropy group $\operatorname{Sp}(2)_{\sigma_o}$ is the subgroup
$$U(2) = \left\{\begin{pmatrix} \alpha & \beta\\ -\overline{\beta} & \overline{\alpha} \end{pmatrix} \mid \alpha,\beta\in\mathbb{R}[j], \ |\alpha|^2+|\beta|^2 = 1\right\}.$$

With a choice of complex structure $j$, $\mathfrak{sp}(2)$ can be identified with the Lie algebra of $4\times 4$ complex matrices
$$\mathfrak{sp}(2) = \left\{ \begin{pmatrix}U & V\\ -\overline{V} &\overline{U} \end{pmatrix}\mid \text{$U, V\in \mathfrak{gl}_2(\mathbb{C})$,\quad $U$ antihermitian, $V$ symmetric}\right\} $$
where the bar denotes ordinary complex conjugation.  This is the algebra of antihermitian matrices on $\mathbb{C}^4$ that preserve the complex symplectic form
$$\omega = e_1\wedge e_3 + e_2\wedge e_4.$$

A $2$-plane $P$ in $\mathbb{C}^4$ is called {\em Lagrangian} if $\omega|_P=0$.  The group $\operatorname{Sp}(2)$ acts transitively on the Lagrangian $2$-planes in $\mathbb{C}^4$, and $U(2)$ is the isotropy group of a particular $2$-plane.  Thus the symmetric space $\operatorname{Sp}(2)/U(2)$ is identified with the set of Lagrangian $2$-planes through the origin in $\mathbb{C}^4$.  This is an algebraic subvariety of the Grassmannian $\operatorname{Gr}_2(\mathbb{C}^4)$, and naturally has a K\"{a}hler structure.

\subsection{$S^n$}  Exercise: discuss $S^n$ as a symmetric space.

\section{Projective spaces}
\subsection{$\mathbb{RP}^n$} 
Let $\tau:S^n\to S^n$ be the antipodal map $\tau(x) = -x$.  Note that $\tau$ is a fixed point free isometry, so the quotient $\mathbb{RP}^n := S^n/\tau$ carries (uniquely) the structure of a Riemannian manifold in such a way that $\tau$ becomes a local isometry.  This is the real projective space, and it is identified as a point set with the space of lines in $\mathbb{R}^{n+1}$ passing through the origin, since the line through the origin and a point $x\in\mathbb{R}^{n+1}$ passes through $x/\|x\|\in S^n$ and $\tau(x/\|x\|)\in S^n$ and conversely the line through the origin is uniquely determined by the two antipodal points that it passes through on the sphere.

Since it is the quotient of a symmetric space by a discrete group, $\mathbb{RP}^n$ is a locally symmetric space.  But in fact, it is easy to see that it is symmetric as well.  Consider the symmetry $\sigma_o:S^n\to S^n$ on the $n$-sphere about the north pole.  This was given by
$$\sigma_o(x) = 2(o\cdot x)o - x.$$
Since $\sigma_o\circ\tau = \tau\circ\sigma_o$, $\sigma_o$ descends to an isometry of the quotient space $\mathbb{RP}^n$.  The group $SO(n+1)$ acts transitively on $\mathbb{RP}^n$ by isometries.  The stabilizer of the point $o$ is $O(n)$.  (Note the difference with the sphere.)

Of particular interest with $\mathbb{RP}^n$ is that it has a nontrivial {\em antipodal set}.\index{Antipodal set}  This is the set of fixed points of $\sigma_o$.  On $S^n$, the fixed points of $\sigma_o$ are $o$ and $-o$.  But the fixed points of $\sigma_o$ of $\mathbb{RP}^n$ correspond to points on $S^n$ that are either fixed or sent to their negatives.  But $\sigma_o(x) = -x$ if and only if $o\cdot x=0$, which is the equation of the equatorial sphere $S^{n-1}\subset S^n$.  Thus the antipodal set of $\sigma_o:\mathbb{RP}^n\to\mathbb{RP}^n$ is $\mathbb{RP}^{n-1}=S^{n-1}/\tau\subset S^n/\tau = \mathbb{RP}^n$.

In any symmetric space of rank one (like $\mathbb{RP}^n$), the antipodal set consists of all points a maximum distance from $o$.  If the isotropy group acts transitively on the unit sphere of the tangent space at $o$, then it acts transitively on the antipodal set.  The antipodal set is the cut locus for the exponential map at $o$, and its complement is a geodesic ball. This fact is significant because it gives a {\em cell decomposition} of $\mathbb{RP}^n$: topologically, $\mathbb{RP}^n$ is the result of attaching an $n$-cell (the $n$-ball) to $\mathbb{RP}^{n-1}$.\index{Cell decomposition}

As a result of the cell decomposition, it is possible to compute the cohomology of $\mathbb{RP}^n$.  Let $U=B_n$ (a ball) and let $V$ be a thickening of the antipodal set $\mathbb{RP}^{n-1}$ so that $\mathbb{RP}^{n-1}$ is a deformation retract of $V$.  Then $S^{n-1}$ is a deformation retract of $U\cap V$.  The Mayer--Vietoris sequence is\index{Cohomology}\footnote{All cohomology groups are real.}
$$H^{q-1}(S^{n-1})\to H^q(\mathbb{RP}^n)\to H^q(\mathbb{RP}^{n-1})\oplus H(\mathbb{R}^n) \to H^q(S^{n-1}).$$
Now
\begin{align*}
H^q(S^{n-1}) &= \begin{cases}
\mathbb{R} & q=0,n-1\\
0 & \text{otherwise,}
\end{cases}\\
H^q(\mathbb{R}^n) &= \begin{cases}
\mathbb{R} & q=n\\
0 & \text{otherwise.}
\end{cases}
\end{align*}
So we determine inductively that
$$H^q(\mathbb{RP}^n,\mathbb{R})=\begin{cases}
\mathbb{R} & q=0\\
0 & 0<q<n\\
\mathbb{R} & \text{$q=n$ odd}\\
0 & \text{$q=n$ even}
\end{cases}.$$\footnote{A different way to compute the cohomology is via the de Rham isomorphism.  Let $(\Omega S^n)^\tau$ be the ring of $\tau$-invariant differential forms on $S^n$.  The exterior derivative commutes with $\tau$ and with the pullback along $S^n\to\mathbb{RP}^n$, and so there is an isomorphism of cochain complexes $((\Omega S^n)^\tau,d)\cong (\Omega\mathbb{RP}^n,d)$.  At the level of cohomology, $H^*((\Omega S^n)^\tau,d) \cong H^*(\mathbb{RP}^n)$.  The inclusion of $(\Omega S^n)^\tau\subset \Omega S^n$ induces is a natural map on cohomology $H^*((\Omega S^n)^\tau,d) \to H^*(S^n,d)$.  This is injective since an invariant form $\omega\in(\Omega S^n)^\tau$ is exact if and only if $\omega = d\sigma$ for some form $\sigma$, but in that case $\omega=\tfrac12 d(\sigma + \tau^*\sigma)$ is the exterior derivative of an invariant form.  Thus the cohomology of $\mathbb{RP}^n$ must be a subspace of the cohomology of $S^n$ in each degree.  Moreover, $\tau : S^n\to S^n$ is orientation preserving if and only if $n$ is odd, so the top degree cohomology class (represented by a volume form) is invariant if and only if $n$ is odd.}
The homology has more structure over $\mathbb{Z}_2$.  Here the homology algebra is the truncated polynomial algebra over $\mathbb{Z}_2$ generated by a single characteristic class $b\in H^1(\mathbb{RP}^n,\mathbb{Z}_2)$ subject to the single relation $b^{n+1}=0$.

The general feature to observe here is that the antipodal set of a compact symmetric space $M$ of rank one is also a symmetric space, and $M$ arises by attaching a cell to its antipodal set.  (This is the analog of the Schubert decomposition of a generalized flag manifold.)  In principle, this allows the cohomology of $M$ to be computed combinatorially.

\subsection{$\mathbb{CP}^n$} Let $\mathbb{CP}^n$ be the space of complex lines through the origin in $\mathbb{C}^{n+1}$.  Introduce coordinates $(Z_0,Z_1,\dots,Z_n)$ on $\mathbb{C}^{n+1}$.  Introduce an equivalence relation on $\mathbb{C}^{n+1}\setminus\{0\}$ via $(Z_0,Z_1,\dots,Z_n)\sim (Z'_0,Z'_1,\dots,Z'_n)$ if there is $\lambda\in\mathbb{C}$, $\lambda\not=0$, such that $Z'_i=\lambda Z_i$ for each $i$.  Then
$$\mathbb{CP}^n = \left.\mathbb{C}^{n+1}\setminus\{0\}\right/\sim.$$
Denote by $Z\mapsto [Z]$ the quotient map $\mathbb{C}^{n+1}\setminus\{0\}\to \mathbb{CP}^n$.

In fact, $\mathbb{CP}^n$ is a complex manifold.  A set of holomorphic coordinate charts is furnished by the open sets $U_i=\{ [Z]\in \mathbb{CP}^n\mid Z_i\not=0\}$.  In each of these, we can choose a scale so that $Z_i=1$, and use the remaining $Z_j$, $j\not=i$, as coordinate on $U_i$.  For instance, in $U_0$, let $z_1=Z_1/Z_0, \dots, z_n=Z_n/Z_0$.  These define $n$ complex-valued coordinate functions on $U_0$.  It is easily seen that the transition functions on the overlaps are holomorphic, and therefore this atlas gives $\mathbb{CP}^n$ the structure of a complex manifold.

Concretely, we can relate $\mathbb{CP}^n$ to the sphere $S^{2n+1}$ by means of the {\em Hopf fibration}.  On $\mathbb{C}^{n+1}$, introduce the Hermitian norm
$$\|Z\|^2 = |Z_0|^2 + |Z_1|^2 + \cdots + |Z_n|^2.$$
The locus $\|Z\|=1$ is the sphere $S^{2n+1}$.  A scaling argument shows that each complex line must intersect $S^{2n+1}$.  In fact, any complex line intersects $S^{2n+1}$ in a circle, since if $Z$ is on $S^{2n+1}$, then $\lambda Z$ is also on $S^{2n+1}$ if (and only if) $|\lambda|=1$.  Identifying $U(1)$ with the group of $\lambda\in\mathbb{C}$ such that $|\lambda|=1$, $Z\mapsto\lambda Z$ defines a group action of $U(1)$ on $S^{2n+1}$, and the above argument shows that (as a point set, at least) $\mathbb{CP}^n = S^{2n+1}/U(1)$.  Since $U(1)$, a compact Lie group, acts effectively, this quotient carries the structure of a differentiable manifold.  Furthermore, $U(1)$ acts by isometries on $S^{2n+1}$, and so the metric on $S^{2n+1}$ descends to a metric on $\mathbb{CP}^n$, the {\em Fubini--Study metric}.

The Fubini--Study metric is K\"{a}hler.  In the homogeneous coordinates $Z_i$, the K\"{a}hler form of the metric is $\Omega = i\partial\overline\partial\log{\|Z\|^2}$.  This 2-form is independent of the choice of holomorphic section of the bundle $\mathbb{C}^{n+1}\setminus\{0\}\to \mathbb{CP}^n$, and it is easy to verify that it is closed (from the decomposition $d=\partial + \overline{\partial}$.  The K\"{a}hler metric itself is then $g(X,Y) = \omega(X,iY).$  

At the group level, $\mathbb{CP}^n$ is a homogeneous space for the group $U(n+1)$ consisting of the isometries of $\mathbb{C}^{n+1}$ that preserve the complex structure.  This acts transitively on the unit sphere $S^{2n+1}$.  If $x\in U(n+1)$ fixes a point $o\in S^{2n+1}$, then $x$ also stabilizes the orthogonal complement of $x$, and so the stabilizer of a point is $U(n)$.  Consequently, the stabilizer of a point is $U(1)\times U(n)\subset U(n+1)$.  Thus\footnote{The isotropy representation of $S(U(1)\times U(n))$ at the base point $o$ is the usual action of $U(n)$ on the tangent space, so this is sometimes written $\mathbb{CP}^n=SU(n+1)/U(n)$, although this is something of an abuse of notation.}
$$\mathbb{CP}^n = \frac{U(n+1)}{U(1)\times U(n)} = \frac{SU(n+1)}{S(U(1)\times U(n))}.$$

Take the base point $o$ to be the coset of the north pole $[1\ 0\ \cdots\ 0]^t\in S^{2n+1}$.  Let $L$ be the complex line through $o$.  Define a linear operator $S:\mathbb{C}^{n+1}\to \mathbb{C}^{n+1}$ such that
$$S|_L = \operatorname{id},\qquad S|_{L^\perp} = -\operatorname{id}.$$
That is,
$$S(Z_0,Z_1,\dots,Z_n) = (Z_0,-Z_1,\dots,-Z_n).$$
Because it is linear, $S_L$ factors through the quotient to $\mathbb{CP}^n$ to define the involution $s_o:\mathbb{CP}^n\to\mathbb{CP}^n$.  The subgroup $S(U(1)\times U(n))\subset SU(n+1)$ is the subset of $(n+1)\times (n+1)$ special unitary matrices that are fixed by conjugation by the element
$$S = \begin{pmatrix}
1 & 0 & 0 &\cdots & 0\\
0& -1 & 0 &\cdots & 0\\
0& 0 & -1 &\cdots & 0\\
\vdots&\vdots&\vdots&\ddots &\vdots\\
0&0&0&\cdots&-1
\end{pmatrix}.$$
Note that, as we saw with the special case of  $S^2\cong\mathbb{CP}^1$, this is an inner automorphism of $SU(n+1)$.  A complex structure that gives rise to the same inner automorphism is $J_o=\operatorname{diag}(i,-i,\dots,-i)$.

A point of $\mathbb{CP}^n$ is fixed by $s_o$ if and only if
$$(Z_0,Z_1,\dots,Z_n) = (\lambda Z_0,-\lambda Z_1,\dots,-\lambda Z_n)$$
which can ony happen if $\lambda=1$ and $Z_1=Z_2=\cdots=Z_n=0$, or $\lambda=-1$ and $Z_0=0$.  The first case corresponds to the base point $o$ itself, while the second case is the antipodal set.  The antipodal set is a cell that is isomorphic to $\mathbb{CP}^{n-1}$.

As in the real case, knowing the antipodal set allows us to compute the cohomology of the complex projective space.  The Mayer--Vietoris sequence is
$$H^{q-1}(S^{2n-1})\to H^q(\mathbb{CP}^n)\to H^q(\mathbb{CP}^{n-1})\oplus H(\mathbb{C}^n) \to H^q(S^{2n-1}).$$
Inductively, we find
$$
H^q(\mathbb{CP}^n) = 
\begin{cases}
\mathbb{R} &\text{for $q=0,2,\dots,2n$}\\
0 & \text{otherwise.}
\end{cases}
$$
A more detailed analysis shows that the cohomology ring of $\mathbb{CP}^n$ is the truncated polynomial algebra generated by a single Chern class $c_1\in H^2(\mathbb{CP}^n,\mathbb{Z})$ subject to only one relation $c_1^2=0$.

\subsection{$\mathbb{HP}^n$}  Let $\mathbb{HP}^n$ be the space of quaternionic lines through the origin in $\mathbb{H}^{n+1}$.  This is the quotient of $\mathbb{H}^{n+1}\setminus\{0\}$ by the equivalence relation $Z\sim \lambda Z$ whenever $Z=(Z_0,\dots,Z_n)\in\mathbb{H}^{n+1}\setminus\{0\}$ and $\lambda$ is a nonzero quaternion.  The unit sphere $S^{4n+3}\subset\mathbb{H}^{n+1}$ is the set of points such that $Z\cdot\overline{Z}=1$, where $\cdot$ is the usual dot product and the bar denotes the quaternion conjugation.  A quaternionic line in $\mathbb{H}^{n+1}$ must intersect the unit sphere.  In fact, if $Z\in S^{4n+3}$, then $\lambda Z$ is also in $S^{4n+3}$ if and only if $|\lambda|=1$.  Indentifying $Sp(1)$ with the group of unit quaternions, this argument shows that (as a point set) $\mathbb{HP}^n=S^{4n+3}/Sp(1)$.  Since $Sp(1)$ is a compact Lie group acting effectively, $\mathbb{HP}^n$ carries the structure of a differentiable manifold.

At the group level, let $Sp(n+1)$ denote the set of $(n+1)\times (n+1)$ quaternionic matrices preserving the quaternion Hermitian form $X\cdot\overline{Y}=X_0\overline{Y}_0+\cdots+X_n\overline{Y}_n$.  This acts transitively on the unit sphere in $\mathbb{H}^{n+1}$.  The stabilizer of the line through $[1\ 0\ cdots\ 0]^t$ is $Sp(1)Sp(n)$.  Indeed, take the base point $o$ to be the coset of the north pole $[1\ 0\ \cdots\ 0]^t\in S^{4n+3}$.  Let $L$ be the quaternionic line through $o$.  Define a linear operator $S:\mathbb{H}^{n+1}\to \mathbb{H}^{n+1}$ such that
$$S|_L = \operatorname{id},\qquad S|_{L^\perp} = -\operatorname{id}.$$
That is,
$$S(Z_0,Z_1,\dots,Z_n) = (Z_0,-Z_1,\dots,-Z_n).$$
Because it is linear, $S_L$ factors through the quotient to $\mathbb{HP}^n$ to define the involution $s_o:\mathbb{HP}^n\to\mathbb{HP}^n$.  The subgroup $Sp(1)Sp(n)\subset Sp(n+1)$ is the subset of $(n+1)\times (n+1)$ quaternionic unitary matrices that are fixed by conjugation by the element
$$S = \begin{pmatrix}
1 & 0 & 0 &\cdots & 0\\
0& -1 & 0 &\cdots & 0\\
0& 0 & -1 &\cdots & 0\\
\vdots&\vdots&\vdots&\ddots &\vdots\\
0&0&0&\cdots&-1
\end{pmatrix}.$$
This is an inner automorphism of $Sp(n+1)$.  However, unlike the case of $SU(n+1)$, it cannot be written as a conjugation by a complex structure.  The group $Sp(1)Sp(n)$ has trivial center, and so there is no compatible complex structure on $\mathbb{HP}^n$.  (Note that the complex structure we wrote down for $SU(n+1)$ is not in the center of $Sp(n+1)$.)

The space $\mathbb{HP}^n$ carries what is known instead as a quaternion-K\"{a}hler structure.  This is a triple of local almost-complex structures $I,J,K$ satisfying the quaternion relations $IJ=K,JK=I,KI=J$, compatible with $g$ such that the subbundle {\em spanned by} $I,J,K$ is invariant under parellism.  This is in contrast to so-called hyper-K\"{a}hler structures, in which there exist local almost complex structures $I,J,K$ satisfying the quaternion relations compatible with the metric that are each individually parallel.  Compact hyper-K\"{a}hler manifolds are very difficult to construct; see Joyce, {\em Compact manifolds with special holonomy}.  In particular, there are no compact symmetric spaces of this kind.

We can work out the homology groups of $\mathbb{HP}^n$ in the same way that we did with the other projective spaces.  It turns out that
$$H^k(\mathbb{HP}^n) =
\begin{cases}
\mathbb{R} & 0\le k\le n, \quad k\equiv 0\pmod{4}\\
0 & \text{otherwise.}
\end{cases}$$
In fact, the cohomology ring is generated by a Chern class $c_2\in H^4(\mathbb{HP}^n,\mathbb{Z})$ with the single relation $c_2^{n+1}=0$.

\subsection{$\mathbb{OP}^2 \cong F_4/\operatorname{Spin}(9)$ (the Cayley plane)}

See Baez \cite{Baez}.  For a synthetic construction, see Besse \cite{BesseClosedGeodesics}.

\section{Negatively curved spaces}
\subsection{The hyperbolic plane}
Let $H^2=\{x+iy\in\mathbb{C}\mid y>0\}$ be the upper half-plane, with the metric defined by
\begin{equation}\label{hyperbolicmetric}
ds^2 = \frac{dx^2+dy^2}{y^2} = \frac{dz\,d\bar{z}}{\operatorname{Im}(z)^2}.
\end{equation}
One can show easily that the group
$$SL(2,\mathbb{R})=\left\{ \begin{pmatrix}a&b\\c&d\end{pmatrix}\mid ad-bc=1, \ a,b,c,d\in\mathbb{R}\right\}$$
when acting by fractional linear transformations
$$Tz = \frac{az+b}{cz+d}$$
preserves the upper half-plane; that is, for any $T\in SL(2,\mathbb{R})$, $\operatorname{Im}(Tz)>0$ if $\operatorname{Im}(z)>0$.  Conversely, if $T\in SL(2,\mathbb{C})$ stabilizes the upper half-plane, then $T\in SL(2,\mathbb{R})$.  Any element $T\in SL(2,\mathbb{R})$ can be written as a composition of operations:
\begin{itemize}
\item Translation: $z\mapsto z+a$, $a\in\mathbb{R}$
\item Dilation: $z\mapsto bz$, $b>0$
\item Inversion: $z\mapsto -1/z$
\end{itemize}
These each act by isometries on the metric \eqref{hyperbolicmetric}.  Moreover, the inversion is an involution with fixed point $i$, and $SL(2,\mathbb{R})$ acts transitively on $H^2$. Thus $H^2$ is a symmetric space, with base point $o=i$, and symmetry at $i$ given by $s_o=-1/z$.  The isotropy group of $o$ is $K=SO(2)\subset SL(2,\mathbb{R})$.  Thus
$$H^2=SL(2,\mathbb{R})/SO(2).$$
The isometry group is
$$I(H^2) = PSL(2,\mathbb{R})\ltimes\mathbb{Z}_2,$$
where $\mathbb{Z}_2$ acts by reflection in the imaginary axis.  In particular,
$$I(H^2)_e = PSL(2,\mathbb{R}).$$

We can work out the geodesics (and their lengths) in $H^2$ as follows.  The curve of shortest length between two points $iy_0, iy_1$ on the imaginary axis is the straight line connecting them, since $y^{-1}\sqrt{dx^2+dy^2} \ge y^{-1}dy$ with equality if and only if $dx=0$. The length of this curve is calculated as
$$\int_{y_0}^{y^1} \frac{dy}{y} = \log\frac{y_1}{y_0}.$$

Now, the group $PSL(2,\mathbb{R})$ is conformal, preserves the real axis, and sends circles to circles.  It follows that the geodesics are circles perpendicular to the real axis.  The length is somewhat awkward to compute in these coordinates (it is simpler in the disc model).

At the Lie algebra level, $s_o=\begin{pmatrix}0&-1\\1&0\end{pmatrix},$ and $\mathfrak{sl}(2,\mathbb{R})$ decomposes via
$$\mathfrak{sl}(2,\mathbb{R}) = \underbrace{\left\{\begin{pmatrix}0&b\\-b&0\end{pmatrix}\mid b\in\mathbb{R}\right\}}_{\mathfrak{so}(2) = \mathfrak{k}_o} \oplus \underbrace{\left\{\begin{pmatrix}a&c\\-a&c\end{pmatrix}\mid a,c\in\mathbb{R}\right\}}_{\mathfrak{p}_o}.$$
The Cartan embedding is
$$H^2\xrightarrow{\cong} P=\exp\mathfrak{p}_o=\left\{\text{positive definite symmetric unimodular matrices}\right\}.$$
Explicitly, this is the embedding
\begin{align*}
x+iy &\mid\!\rightarrow\begin{pmatrix}y + x^2/y & x/y\\ x/y & 1/y\end{pmatrix}\\
\frac{B}{D}+\frac{i}{D} &\leftarrow\!\mid\begin{pmatrix}A & B\\ B & D\end{pmatrix}.
\end{align*}

There is also the {\em Cayley transform}
$$z\mapsto \frac{z-i}{z+i}$$
that maps $H^2$ to $D$, the unit disc.  This becomes an isometry when $D$ is equipped with the hyperbolic metric
$$ds_D^2 = \frac{dz\,d\bar{z}}{4(1-|z|^2)^2}.$$
The inverse Cayley transform is
$$w\mapsto -i\frac{w+1}{w-1}.$$

The hyperbolic plane is particularly important in arithmetic.  The arithmetic subgroup $SL(2,\mathbb{Z})\subset SL(2,\mathbb{R})$ acts properly discontinuously on $H^2$.  The quotient
$$SL(2,\mathbb{Z})\backslash H^2$$
is a locally symmetric space.  The geometrical significance of this quotient is as follows.  Let $L\subset\mathbb{C}$ be a (unimodular) lattice.  Then $\mathbb{C}/L$ is an elliptic curve.  Modulo rotation, any unimodular lattice is equivalent to a lattice of the form $L_\tau=\mathbb{Z}+\mathbb{Z}\tau$ where $\tau\in H^2$, and so any elliptic curve can be written
$$E_\tau = \mathbb{C}/L_\tau.$$
Any two unimodular lattices are conjugate under $SL(2,\mathbb{R})$, and $SL(2,\mathbb{Z})$ is the stabilizer of $L_i$.  Thus the moduli space of isomorphism classes of elliptic curves is
\begin{align*}
SL(2,\mathbb{Z})\backslash H^2 &\leftrightarrow E_\tau = \mathbb{C}/L_\tau \\
\tau &\leftrightarrow E_\tau.
\end{align*}

\subsection{Hyperbolic space}  Define the quadratic form $Q$ on $\mathbb{R}^{n+2}$ by
$$Q(x) = -x_0^2 + x_1^2+\cdots + x_{n+1}^2$$
and let $B(x,y)$ be the nondegenerate bilinear form induced by polarization. Let $o=(1,0,\dots,0)$.  Hyperbolic $n$-space is defined to be the locus
$$H^n = \{x\in\mathbb{R}^{n+2}\mid Q(x) = -1,\quad B(x,o)<0\}.$$
This is a connected component of the hyperboloid of two sheets $Q(x)=-1$.  The quadratic form $-B$ induces a Riemannian metric on $H^n$.  Indeed, the normal vector to $H^n$ in $\mathbb{R}^{n+2}$ is timelike ($Q(n)<0$), so by Sylvester's law of inertia $Q|_{n^\perp}$ is a positive-definite form.

\begin{theorem}
The group $SO(Q)_e$ acts transitively on $H^n$.
\end{theorem}
To prove transitivity of the full group $O(Q)$, let $p,q$ be two points on $H^n$, and let $W$ be the $2$-plane through the origin containing $p,q$.  Then $Q|_W$ defines a nonsingular quatradic form on $W$, and the locus $Q|_W=-1$ is clearly an orbit of the one-dimensional Lorentz group
$$t\mapsto \begin{pmatrix}\cosh t & \sinh t \\ \sinh t & \cosh t\end{pmatrix}.$$
Thus transitivity follows by Witt's theorem:
\begin{theorem}[Witt's theorem]
Let $V$ be a vector space over a field and $Q$ a nondegenerate quadratic form on $V$.  Let $W\subset V$ be a linear subspace.  Then a function $f:W\to V$ can be extended to an isometry $f:V\to V$ if and only if $Qf(w)=Q(w)$ for all $w\in W$.
\end{theorem}
The group $SO(Q)_e=SO(n,1)_e$.  The stabilizer of the point $o$ is $SO(n)$.  Hence
$$H^n = SO(n,1)_e/SO(n).$$

The metric on $H^n$ can be described concretely as follows.  Let $S=(-1,0,0,\dots,0)$ be the ``south pole''.  Let $\rho_S:H^n\to \{0\}\times\mathbb{R}^n\subset\mathbb{R}^n$ be the stereographic projection with respect to $S$.  This is given by
$$\rho_S(x) = S - 2\frac{(x-S)}{\langle x-S,x-S\rangle}.$$
This is a diffeomorphism of $H^n$ onto the ball $\{\xi\in\mathbb{R}^n\mid |\xi|^2<1\}$, and the induced metric is
$$\frac{4}{(1-|\xi\|^2)^2}d\xi^i\otimes d\xi^i.$$
This is the Poincar\'{e} ball model for hyperbolic space.  The geodesics are circles perpendicular to the boundary of the ball.

Related is the Klein model, which is the stereographic projection of $H^n$ onto $TH^n_o$ with respect to the origin in $\mathbb{R}^{n+1}$.  Since the geodesics on $H^n$ are the intersections of $H^n$ with $2$-planes through the origin of $\mathbb{R}^{n+1}$, the geodesics go over to straight lines in the Klein model.

\subsection{Siegel space}
Let $B\subset\mathbb{R}^n$ be an open subset.  The {\em tube over $B$}, $T_B$, is the subset of $\mathbb{C}^n$ given by
$$T_B = \{ z=x+iy\in\mathbb{C}^n\mid y\in B\}.$$
When $B$ is an open convex cone in $\mathbb{R}^n$, $T_B$ is a natural domain for the Hardy space of several variables.  A special case of this is when $B\subset\mathbb{R}^{n\times n}$ is the convex cone of positive-definite matrices.  Then $T_B:=\mathfrak{H}_n$ is called the {\em Siegel space}, and plays an important role in the study of abelian varieties and modular forms: the quotient of the Siegel space by an arithmetic subgroup of the symplectic group is the moduli space of polarized abelian varieties.

We can identify $\mathfrak{H}_n$ with the space of (positive definite) K\"{a}hler structures on $\mathbb{R}^{2n}$ compatible with a given symplectic structure.  Thus as a homogeneous space, $\mathfrak{H}_n$ is the symmetric space $Sp(n,\mathbb{R})/U(n)$.

\chapter{Noncompact symmetric spaces}
We return to the general study of noncompact symmetric spaces.  We shall here (without loss of generality) confine attention to closed subgroups $G\le GL(n,\mathbb{R})$ that are closed under the transpose.  In particular, with this simplification, we can identify the Lie algebra of $G$ with a subalgebra of $\mathfrak{gl}(n,\mathbb{R})$:
$$\mathfrak{g}_o=\{X\in\mathfrak{gl}(n,\mathbb{R}) \mid \exp tX\in G \quad\text{for all $t\in R$}\}.$$
This is of course also closed under the transpose. Here the exponential is just the usual matrix exponential.

To study such groups, it is necessary to understand first the case when $G=GL(n,\mathbb{R})$ itself and $K=O(n)$.  As a point-set the resulting symmetric space $G/K$ is already well understood from linear algebra.  Indeed, the {\em polar factorization} of a nonsingular real matrix $A$ asserts that there is a unique symmetric matrix $X$ and orthogonal matrix $k$ such that $A=ke^X$.  To understand $G/K$ (or equivalently, to understand the symmetric space $SL(n,\mathbb{R})/SO(n,\mathbb{R})$\footnote{A somewhat detailed account of this symmetric space can be found in J\"{u}rgen Jost, {\em Riemannian geometry and geometric analysis}, 5th ed., \S 5.5.}, we shall establish that the polar decomposition is a diffeomorphism.  Let $\operatorname{Symm}(n,\mathbb{R})$ denote the space of $n\times n$ real symmetric matrices and denote by $\Phi : O(n)\times \operatorname{Symm}(n,\mathbb{R})$ the polar factorization $\Phi(k,X) = k\exp X$.

The following Lemma is a topological constraint on the kinds of 1-parameter subgroups that can arise
\begin{lemma}[Chevalley's lemma]\index{Chevalley's lemma}
Let $f$ be be a polynomial function on $n\times n$ matrices and suppose that $X$ is a real $n\times n$ symmetric matrix.  If $f(\exp mX) = 0$ for infinitely many integers $m$, then $f(\exp tX)=0$ for all $t$.
\end{lemma}
\begin{proof}
We may assume that $X$ is diagonal and that $f$ is a polynomial in the diagonal entries $d_i$.  Then
$$f(\exp tX) = f(e^{td_1},\dots,e^{td_n}).$$
For large $t$, $ f(e^{td_1},\dots,e^{td_n})\not=0$ unless $ f(e^{td_1},\dots,e^{td_n}) = 0$ for all $t$.
\end{proof}

\begin{proposition}
$\Phi : O(n)\times \operatorname{Symm}(n,\mathbb{R})\to GL(n,\mathbb{R})$ is a diffeomorphism.
\end{proposition}

\begin{proof}
We first show (1) that $\Phi$ is regular, then (2) that it is surjective, and finally (3) that it is injective.

\begin{enumerate}
\item For regularity, consider $\exp :  \operatorname{Symm}(n,\mathbb{R})\to GL(n,\mathbb{R})$.  We compute the differential of $\exp$ at $X\in \operatorname{Symm}(n,\mathbb{R})$:
\begin{align*}
d\exp_X Y &= \left.\frac{d}{dt}\right|_{t=0}\exp(X+tY)\\
&=\left.\frac{d}{dt}\right|_{t=0}\sum_{n=0}^\infty \frac{1}{n!}(X+tY)^n\\
&=\sum_{n=0}^\infty \frac{1}{n!}(X^{n-1}Y + X^{n-2}YX + \cdots + XYX^{n-2}+YX^{n-1})\\
&=\sum_{n=0}^\infty \frac{1}{n!}(L_X^{n-1} + L_X^{n-2}R_X + \cdots + L_XR_X^{n-2}+R_X^{n-1})(Y)\\
&\qquad\text{where $L_X$ (resp. $R_X$) denotes the multiplication on the left (resp. right) by $X$.}\\
&=\sum_{n=0}^\infty \frac{1}{n!}\frac{L_X^n-R_X^n}{L_X-R_x}(Y)\\
&=\frac{1}{L_X-R_X}(L_{\exp X} - R_{\exp X})(Y)= \frac{1-e^{-\ad X}}{\ad X}(Y).
\end{align*}

Now let $k\in O(n)$, $X,Z\in\operatorname{Symm}(n,\mathbb{R})$, $Y\in \mathfrak{o}(n)$ (skew-symmetric matrices).  Then
\begin{align*}
d\Phi_{(k,X)}(YZ) &= \left.\frac{d}{dt}\right|_{t=0} k\exp(tY)\exp(X+tZ)\\
&=k(Y\exp X + d\exp_X Z).
\end{align*}
If $d\Phi_{k,X}(Y,Z)=0$, then
$$Y\exp X + d\exp_X Z = 0.$$
Now $d\exp_X Z$ is symmetric, and therefore $Y\exp X$ is also symmetric.  Hence, because $X$ is symmetric and $Y$ is skew, $Y\exp(X) = -\exp(X)Y.$  Hence $\exp(X)Y\exp(-X)=-Y$, so if $Y$ were nonzero, then it would be an eigenvector of $\Ad\exp(X)$ with eigenvalue $-1$.  But suppose that $X$ has eigenvalues $a_1,\dots,a_n$.  Then $\ad X$ has eigenvalues $a_i-a_j$, and $\Ad\exp X = \exp \ad X$ has eigenvalues $e^{a_i-a_j}$.  But, for real $a_i$, we have $e^{a_i-a_j}\not=-1$.  Hence $Y=0$.

Thus we have $d\exp_XZ=0$.  Using the above calculation,
$$d\exp_xZ = \underbrace{\sum_{\text{$k$ odd}} \frac{(-1)^{k-1}}{k!}\ad^{k-1}(X)Z}_{\text{symmetric}}+\underbrace{\sum_{\text{$k$ even}} \frac{(-1)^{k-1}}{k!}\ad^{k-1}(X)Z}_{\text{skew-symmetric}}=0.$$
Since the first term above is the symmetric part of $d\exp_XZ$, and the second is the skew part, each one must individually vanish.  In particular,
$$\sum\frac{1}{(2k+1)!} \ad^{2k}(X)Z = 0$$
so $Z$ is an eigenvector for $\sum_{k=0}^\infty\frac{\ad^{2k}(X)}{(2k+1)!}$ with eigenvalue $0$.  But this is a sum of squares
$$\operatorname{id} + \frac{1}{6}\ad^2(X) + \cdots$$
which has no nontrivial eigenvectors with eigenvalue $0$.  Hence $Z=0$ as well.  This shows that $\Phi$ is regular.

\item For surjectivity, let $g\in GL(n,\mathbb{R})$.  Since $g^tg$ is symmetric, positive-definite, $g^tg=\exp(2X)$ for some $X\in\operatorname{Symm}(n,\mathbb{R})$.  Then $k=g\exp(-X)$ satisfies $k^tk=\exp(-X)g^tg\exp(-X) = e$, so $k\in O(n)$ and $g=\Phi(k,X)$.

\item For injectivity, suppose that $g=k_1\exp X_1=k_1\exp X_2$.  Then $g^tg=\exp 2X_1=\exp 2X_2$.  It is therefore sufficient to show that $[X_1,X_2]=0$, for then $e=\exp 2X_1\exp 2X_2 = \exp(2(X_1-X_2))$ which implies that $X_1=X_2$.

By raising to a power, $\exp (2mX_1)$ commutes with $\exp (2X_2)$ for all integers $m$.  But this is a polynomial relation in the entries of $\exp(2mX_1)$.  Hence by Chevalley's lemma, $\exp (tX_1)$ commutes with $\exp (2X_2)$ for every real $t$.  Differentiating the one parameter group $\exp(tX_1)$ then implies that $X_1$ commutes with $\exp(2X_2)$.  So $X_1$ commutes with $\exp(2mX_2)$ for every integer $m$, and so again by Chevalley's lemma, $X_1$ commutes with the one-parameter group $\exp(tX_2)$, and so $[X_1,X_2]=0$. 

\end{enumerate}
\end{proof}

\begin{proposition}
Let $G_{\mathbb{C}}\subset GL(\mathbb{C},n)$ be a linear algebraic group and let $G_{\mathbb{R}}=G_{\mathbb{C}}\cap GL(\mathbb{R},n)$ be a closed subgroup that is closed under the transpose. Let $G_o=(G_{\mathbb{R}})_e, K_o=G_o\cap O(n)$, $\mathfrak{p}_o=\mathfrak{g}_o\cap\operatorname{Symm}(n,\mathbb{R})$.  Then
\begin{align*}
K_o\times\mathfrak{p}_o &\xrightarrow{\Phi} G_o\\
(k,X)&\mapsto k\exp X
\end{align*}
is a diffeomorphism.
\end{proposition}

\begin{proof}
Since $\Phi$ is the restriction of a regular injective mapping on $GL(\mathbb{R},n)$ to a submanifold, $\Phi$ is regular and injective, so it is enough to prove surjectivity.  Let $g\in G_o$ be given.  Then $g^tg\in G_{\mathbb{R}}$ can be written $g^tg=\exp(2X)$ for some $X\in\operatorname{Symm}(n,\mathbb{R})$.  Also $\exp(2mX)\in G_{\mathbb{C}}$ for all integers $m$.  Since the algebraic group $G_{\mathbb{C}}$ is the locus of a set of polynomial equations, Chevalley's lemma implies that the one-parameter group $\exp(tX)\in G_{\mathbb{C}}$.  Differentiating at $t=0$ gives $X\in\mathfrak{g}_{\mathbb{C}}$ as well.  Thus $X\in\operatorname{Symm}(n,\mathbb{R})\cap\mathfrak{g}_{\mathbb{C}}=\mathfrak{p}_o$.

Since $X\in\mathfrak{p}_o$, $\exp X\in G_{\mathbb{R}}$, and so $k:=g\exp(-X)\in O(n)\cap G_{\mathbb{R}} = K_o$.  We have $g=k\exp X=\Phi(k,X)$ for $k\in K_o$ and $X\in\mathfrak{p}_o$, as required.
\end{proof}

\begin{proposition}
Let $G_o$ and $K_o$ be as above.  Let $\rho : G\to G_o$ be a covering map and let $K=\rho^{-1}(K_o)$.  Then $K\times\mathfrak{p}_o\xrightarrow{\Phi} G, \Phi(k,X)=k\exp X$, is a difeomorphism.
\end{proposition}
\begin{proof}
\begin{enumerate}
\item Note that $\exp_G$ covers $\exp_{G_o}$:
$$\xymatrix{
\mathfrak{p}_o\ar[d]^{=} \ar[r]^{\exp_G} & G\ar[d]^\rho\\
\mathfrak{p}_o\ar[r]^{\exp_{G_o}} & G_o
}
$$

Hence $\Phi$ covers the diffeomorphism $\Phi_o:K_o\times\mathfrak{p}_o \to G_o$:
$$\xymatrix{
K\times\mathfrak{p}_o\ar[d]^{\rho\times\operatorname{id}}\ar[r]^{\Phi}&G\ar[d]^\rho\\
K_o\times\mathfrak{p}_o\ar[r]^{\Phi_o}&G
}$$

The bottom arrow $\Phi_o$ is a diffeomorphism, and the vertical arrows are covering maps.  Since regularity is local, $\Phi$ is regular because the corresponding result is true of $\Phi_o$.  
\item For injectivity, suppose that $k_1\exp_G X_1=k_2\exp_G X_2$.  Then
$$\rho(k_1)\exp_{G_o}X_1 = \rho(k_2)\exp_{G_o} X_2$$
by injectivity of $\Phi_o$, it follows that $X_1=X_2$.  But therefore we must also have $k_1=k_2$.
\item For surjectivity, let $g\in G$.  Then there exists $k_o\in K_o$ and $X\in \mathfrak{p}_o$ such that
$$\rho(g) = k_o\exp_{G_o}X.$$
As a result $g\exp_G(-X)\in K$, say $g\exp_G(-X)=k$.  Then
$$g = k\exp_G X = \Phi(k,X),$$
as required.
\end{enumerate}
\end{proof}

\begin{theorem}[Cartan decomposition]\index{Cartan decomposition}
Let $(\mathfrak{g}_o,s_o)$ be a symmetric orthogonal noncompact semisimple Lie algebra, and let $G$ be a connected Lie group with $\Lie G = \mathfrak{g}_o$.  Let $K=\Ad^{-1}(\operatorname{Int}(\mathfrak{g}_o)\cap O(n))$.  Then
\begin{align*}
K\times\mathfrak{p}_o &\xrightarrow{\Phi} G\\
(k,X)&\mapsto k\exp X
\end{align*}
is a diffeomorphism.
\end{theorem}

\begin{proof}
Without loss of generality, we can identify $\mathfrak{g}_o$ with its image under the adjoint map $\ad : \mathfrak{g}_o\xrightarrow{\subset}\mathfrak{gl}(\mathfrak{g}_o)$.  Moreover, we can assume without loss of generality that $\mathfrak{g}_o$ is closed under the transpose, and $s_oX=-X^t$.  With these identifications, $\operatorname{Int}(\mathfrak{g}_o)$ is identified with a subgroup of $\operatorname{Aut}(\mathfrak{g}_o)$.  Now
$$\operatorname{Int}(\mathfrak{g}_o) = \left\langle e^{\ad X}\mid X\in\mathfrak{g}_o\right\rangle$$
and $(e^{\ad X})^t = e^{\ad X^t} = e^{-\ad s_oX}$.  So $\operatorname{Int}(\mathfrak{g}_o)$ is closed under transpose.  Furthermore $\operatorname{Int}(\mathfrak{g}_o)=\operatorname{Aut}(\mathfrak{g}_o)_e$ and $\mathfrak{g}_o=\Lie(G)=\Lie(\operatorname{Int}\mathfrak{g}_o)$.  Thus
$$G \xrightarrow{Ad} \operatorname{Int}(\mathfrak{g}_o)$$
is a covering Lie group morphism.  Hence, the groups
$$\xymatrix{ 
G\ar[d]^{\Ad} & \ar[l]^{\le} K\ar[d]\\
G_o=\operatorname{Int}\mathfrak{g}_o & K_o=\operatorname{Int}\mathfrak{g}_o\cap O(n)
}
$$
satisfies the hypotheses of the previous proposition.  Thus $K\times\mathfrak{p}_o\to G$ is a diffeomoprhism.
\end{proof}

\begin{corollary}
Let $(G,K)$ be as above.  Then $K$ is connected and $G/K$ is simply connected.
\end{corollary}
Indeed, $G/K$ is diffeomorphic to $\mathfrak{p}_o$, and so $K$ is a deformation retract of $G$.

\begin{proposition}  Let $(G.K)$ be as above and let $(\widetilde{G},\widetilde{K})$ cover the pair $(G,K)$ with $\widetilde{G}$ simply-connected.  Then $Z(\widetilde{G})\le\widetilde{K}$.
\end{proposition}

\begin{proof}
Because $\widetilde{G}$ is simply-connected, there exists a Lie group morphism $s:\widetilde{G}\to\widetilde{G}$ such that $s_{*,e}=s_o$.  Since $s$ is an automorphism, $s(Z(\widetilde{G}))=Z(\widetilde{G})$.  Suppose that $k\exp X = g\in Z(\widetilde{G})$ where $k\in\widetilde{K}$ and $X\in\mathfrak{p}_o$.  Then, since $s|_{\widetilde{K}}=\operatorname{id}_{\widetilde{K}}$ and $s|_{\mathfrak{p}_o}=-\operatorname{id}_{\mathfrak{p}_o}$,
$$s(g) = k\exp(-X)\in Z(\widetilde{G}).$$
Thus $s(g)^{-1}g=\exp 2X\in Z(\widetilde{G}) = \ker\operatorname{Ad}$.  By Chevalley's lemma, $\exp(tX)\in Z(\widetilde{G})$ for all $t$.  Hence $X\in Z(\mathfrak{g}_o)\cap\mathfrak{p}_o = 0$ by effectiveness.
\end{proof}

\begin{corollary}
Let $(G,K)$ be as in the theorem.  Then there exists a Lie group involution $s:G\to G$ such that $s_{*,e}=s_o$ and $G_s=K$.
\end{corollary}

\begin{proof}
Consider the diagram
$$\xymatrix{
\widetilde{G}\ar[r]^s\ar[d]^\pi&\widetilde{G}\ar[d]^\pi\\
G\ar@{-->}[r]^s&G
}
$$
We wish to show that there is a (unique) mapping $s$ for which this is commutative.  Note that $\ker\pi\le\widetilde{G}$ is a discrete normal subgroup, which is therefore central.  Thus, by the proposition, $\ker\pi\le Z(\widetilde{G})\le\widetilde{K}$.  Since $s=\operatorname{id}$ on $\widetilde{K}$, $s=\operatorname{id}$ on $\ker\pi$.  Hence $s:\widetilde{G}\to\widetilde{G}$ factors through a unique a group morphism $s : G\to G$.  

Since $s_{*,e}=\operatorname{id}$ on $\mathfrak{k}_o$, $s|_K=\operatorname{id}$.  Thus $K\subset G_s$.  For the reverse inclusion, if $g=k\exp X\in G_s$, then $s(g) = k\exp(-X) = g = k\exp X$, from which we obtain $X=0$, and so $g=k\in K$.
\end{proof}

\begin{corollary}
Let $(G,K)$ be as before.  Then $Z(G)\le K$.
\end{corollary}

\begin{theorem}
Let $M$ be a Riemannian irreducible noncompact connected globally symmetric space.  Then $M$ is simply connected.  Moreover $I(M)_e = \operatorname{Int}\mathfrak{g}_o$, $I(M) = \operatorname{Aut}(\mathfrak{g}_o)$.
\end{theorem}

{\em Remark.}  It is interesting to compare this result to the {\em Cartan--Hadamard theorem} which states that a complete Riemannian manifold of nonpositive sectional curvature is covered by a vector space.

\begin{proof}
The isometry group $I(M)$ acts transitively on $M$.  Let $\sigma_o$ be the geodesic symmetry on $M$, and $s_o : I(M)\to I(M)$ the conjugation map by $\sigma_o$: $\phi\mapsto \sigma_o\phi\sigma_o$.  Note that an isometry commutes with $\sigma_o$ if and only if it fixes $o$, so $I(M)_o=I(M)_{s_o}$.

The group $I(M)$ is semisimple, noncompact, and
$$M\cong I(M)_e/I(M)_{e,o}.$$
Hence $Z(I(M)_e)\subset I(M)_{e,o}$.  In fact, this is true for all $o\in M$.  But that is only possible if $Z(I(M)_e)=\operatorname{id}_M$.  Thus $I(M)_e=\operatorname{Int}(\mathfrak{g}_o)$.  There is, moreover, a map $I(M)\to\operatorname{Aut}\mathfrak{g}_o$, given as follows.  Let $\phi\in I(M)$ and suppose that $\phi(o)=\overline{o}$.  We can write $\overline{o}=\exp(X).o$ for some unique $X\in\mathfrak{p}_o$ (since $\mathfrak{p}_o$ is diffeomorphic to $M$ via the map $\exp(\cdot).o$).  The map $\phi \mapsto \exp(-X)\circ\phi$ defines the map $I(M)\to\operatorname{Aut}\mathfrak{g}_o$.  This mapping is compatible with the Cartan decomposition
$$I(M)_o\times\mathfrak{p}_o\xrightarrow{\cong} I(M).$$
Now 
$$I(M)_o = \{ f : \mathfrak{g}_o\xrightarrow{\cong}\mathfrak{g}_o \mid fs_o=s_of\} = \operatorname{Aut}(\mathfrak{g}_o)\cap O(n).$$
By applying the Cartan decomposition to $\operatorname{Aut}(\mathfrak{g}_o)$, there is a diffeomorphism
\begin{equation}\label{AutCartan}
(\operatorname{Aut}(\mathfrak{g}_o)\cap O(n))\times\mathfrak{p}_o\xrightarrow{\cong}\operatorname{Aut}(\mathfrak{g}_o).
\end{equation}
Hence this shows that $I(M)_o$ is a maximal compact subgroup of $\operatorname{Aut}(\mathfrak{g}_o)$ and therefore also a maximal compact subgroup of $I(M)$.  Likewise, $I(M)_{e,o}$ is a maximal compact subgroup of $I(M)_e$.  Together with \eqref{AutCartan}, this implies that $I(M)_e = \operatorname{Int}\mathfrak{g}_o$, $I(M) = \operatorname{Aut}(\mathfrak{g}_o)$.
\end{proof}

\chapter{Compact semisimple Lie groups}\label{CompactSemisimple}
\section{Introduction}
Recall that Riemannian globally symmetric spaces of type II are pairs $(G,s_0)$ where $G$ is a connected compact semisimple Lie groups equipped with its unique bi-invariant Riemannian metric of unit volume, and $s_0$ is the involution at the identity $s_0(g) = g^{-1}$.  Let $G_{sc}$ denote the simply connected cover of $G$.  This is also compact, and the identity component of the isometry group $\operatorname{I}(G)_e$ is covered by $G_{sc}\times G_{sc}$ which acts via $(g,h)\cdot x = gxh^{-1}$.  The Cartan involution in $G_{sc}\times G_{sc}$ is the flip map $s_e(g,h)=(h,g)$.  The Lie algebra of the isometry group is the same as that of $G_{sc}\times G_{sc}$, and this is $\mathfrak{g}_0\oplus\mathfrak{g}_0$.  The flip map is $s_{e,*}(x\oplus y) = y\oplus x$, and the eigenspaces are
$$\mathfrak{k}_0 = \Delta(\mathfrak{g}_0) = \{(x,x)\mid x\in\mathfrak{g}_0\},\quad \text{and}\quad\mathfrak{p}_0 = \{(x,-x)\mid x\in\mathfrak{g}_0\}.$$

A classification of Riemannian symmetric spaces of type II is tantamount to a classification of the compact connected semisimple Lie groups.  This classification is achieved by first knowing the classification of complex semisimple Lie groups.  Each complex semisimple Lie group has a unique compact real form, which has a unique simply connected universal cover (also compact).  Finally, any compact $G$ is the quotient of $G_{sc}$ by a discrete subgroup, which is necessarily in the center.  Therefore, understanding type II symmetric spaces ultimately rests on being able to identify the center in a compact simply-connected semisimple Lie group.

The approach is to study the geometry of maximal tori in $G$ and $G_{sc}$.  These tori give rise in a natural manner to various lattices, and a systematic study of these lattices reveals the center of $G_{sc}$.  The Weyl group acts by conjugation on these tori, and this group action lifts to an action of the so-called affine Weyl group on the Lie algebra of the torus.

Many questions can be resolved with a careful study of the fundamental alcove of the affine Weyl group.  For instance, any two maximal tori in $G$ are conjugate, and any element of $G$ must lie inside some maximal torus, so $G=\cup_{g\in G} g^{-1}Tg$ for any maximal torus $T$.  Therefore, any element of $G$ is represented up to conjugacy by a unique element of the fundamental alcove.  This gives a normal form for elements of $G$ when acting by conjugation on itself---compare with the Jordan form of a matrix in $\operatorname{GL}(n,\mathbb{C})$.

\section{Compact groups}
Let $G$ be a compact connected Lie group.  The group $G\times G$ acts naturally on $G$ via the left action
$$(a,b)x = axb^{-1}.$$
As a global Riemannian symmetric space, there is an isomorphism
$$G\cong G\times G/\Delta(G)$$
where $\Delta(G)$ is the diagonal subgroup $\Delta(G) = \{(x,x)\mid x\in G\}$.  The involution in this symmetric space is the {\em flip} mapping
$$s:G\times G \to G\times G,\qquad s(x,y) = (y,x).$$
Let $\mathfrak{g}_o=\Lie(G)$.  Then $\Lie(G\times G)=\mathfrak{g}_o\oplus\mathfrak{g}_o$ and the differential of $s$ at the identity is
$$s_o := s_{*,e} : x\oplus y \mapsto y\oplus x.$$
As usual, decompose $\Lie(G\times G)$ into the eigenspaces for $s_o$
$$\mathfrak{g}_o\oplus\mathfrak{g}_o = \underbrace{\Delta(\mathfrak{g}_o)}_{\mathfrak{k}_o}\oplus \underbrace{\{ (X,-X)\mid X\in\mathfrak{g}_o\}}_{\mathfrak{p}_o}.$$

The composite $\mathfrak{g}_o\xrightarrow{X\mapsto (X,-X)} \mathfrak{p}_o\xrightarrow{\exp} G\times G\xrightarrow{(x,y)\mapsto (x,y)e} G$ sends $X\mapsto (X,-X)\mapsto (\exp(X),\exp(-X))\mapsto \exp(X)\exp(-X)^{-1} = \exp(2X)$.  Thus in particular $\mathfrak{g}_o\xrightarrow{\exp} G$ is surjective.

Let $\mathfrak{t}_o\subset\mathfrak{g}_o$ be a maximal toral subalgebra.  Then
$$\mathfrak{a}_o = \{(X,-X)\mid X\in\mathfrak{t}_o\}\subset\mathfrak{p}_o$$
be a flat subspace.  The subgroup
$$T = \exp\mathfrak{t}_o = \exp\mathfrak{a}_o\cdot e \subset G$$
is closed.  So $T\le G$ is an abelian, closed, connected subgroup (a {\em torus}).  Furthermore, $T$ is a maximal torus.

By Corollary \ref{coverbyflats}, $\mathfrak{p}_o = \bigcup_{k\in \Delta(G)} \Ad(k)\mathfrak{a}_o$ and so
$$ \mathfrak{g}_o = \bigcup_{g\in G} \Ad(g)\mathfrak{t}_o$$
which, at the Lie group level goes over to
\begin{equation}\label{coverbytori}
G = \bigcup_{g\in G} gTg^{-1}.
\end{equation}

\subsection{The center of a compact group}
\begin{definition}
Let $K$ be a compact abelian Lie group.  We say that $x\in K$ is a topological generator if $K = \operatorname{cl} \langle x^n\mid n\in \mathbb{Z}\rangle $.
\end{definition}

{\em Exercise.}  For example, a topological generator of the standard torus $\mathbb{R}^2/\mathbb{Z}^2$ is any point with irrational coordinates.

{\em Remarks:}  \begin{enumerate}
\item If $K$ is connected, then $K$ has topological generators.  Indeed, $\exp : \mathfrak{k}\to K$ is a group homomorphism, since $K$ is abelian, and so $\ker\exp\subset\mathfrak{k}$ is a lattice.  Any element that is irrational with respect to this lattice is a topological generator.
\item If $K/K_e$ is cyclic, then $K$ has topological generators.  Indeed, suppose that
$$K_e =  \operatorname{cl} \langle k_o^n\mid n\in \mathbb{Z}\rangle $$
and let $k_1\in K$ be such that $k_1K_e$ is a generator of $K/K_e$ (a finite cyclic group of order $m$).  Choose $a\in K_e$ such that $a^m=k_1^mk_o\in K_e$ (possible since the exponential map is surjective onto $K_e$).  Let $b=a^{-1}k_1\in K$.  Now $b^m=a^{-m}k_1^{m} = k_o$.  So $\operatorname{cl}\langle b^n\mid n\in \mathbb{Z}\rangle$ contains $K_e$.  In particular, it contains $a$, and hence also $k_1=ab$, which along with $K_e$ generates $K$.  Thus $\operatorname{cl}\langle b^n\mid n\in \mathbb{Z}\rangle = K$.
\end{enumerate}

Let $S\subset G$ be a torus, and $x\in C_G(S)$.  Then $K=\operatorname{cl}\langle S,x\rangle$ is  closed abelian subgroup of $G$ (and therefore compact).  Note that $K/K_e$ is cyclic, being generated by $xK_e$.  So $K$ has a topological generator $g$.  Since $g\in T$ for some maximal torus $T$, we have
$$C_G(S) =\hspace{-18pt} \bigcup_{\begin{matrix}T\supset S\\ \text{$T$ maximal torus}\end{matrix}}\hspace{-18pt} T.$$

\begin{corollary}
\mbox{}
\begin{enumerate}
\item $\displaystyle{C_G(S) =\hspace{-18pt} \bigcup_{\begin{matrix}T\supset S\\ \text{$T$ maximal torus}\end{matrix}}\hspace{-18pt} T}$ \quad is connected.

\vspace{12pt}

\item $C_G(x)$ is connected for all $x\in G$.
\item If $T$ is a maximal torus, then $C_G(T) = T$.  In particular, $T$ is a maximal abelian subgroup of $G$.
\end{enumerate}
\end{corollary}

\begin{proposition}
$Z(G) = \bigcap_{\text{$T$ maximal torus}}T$.
\end{proposition}

\begin{proof}
For one inclusion, let $z\in (G)$ be given.  There exists a maximal torus $T$ such that $z\in T$.  But for any $g\in G$, $z=gzg^{-1}=gTg^{-1}$, and so $z\in \bigcap_{g\in G} g^{-1}Tg$.  Since any two maximal tori are conjugate by an element of $G$, this implies $Z(G) \subset \bigcap_{\text{$T$ maximal torus}}T$.

For the opposite inclusion, let $x\in \bigcap_{\text{$T$ maximal torus}}T$ and $y\in G$.  Then $y\in T$ form some maximal torus $T$, and yet also $x\in T$ since it is in every maximal torus.  Thus $xy=yx$, and so $x\in Z(G)$.
\end{proof}

\section{Weyl's theorem}\index{Weyl's theorem}
In the next section, for a compact connected semisimple group $G$, we shall study the adjoint group $\Ad G=G_{\ad}$ and the universal covering group $G_{sc}$.  These are compact groups, by a theorem of Hermann Weyl
\begin{theorem}
Let $G$ be a compact connected semisimple Lie group.  Then the fundamental group of $G$ is finite.\end{theorem}
\begin{proof}[Proof sketch]\footnote{This proof is loosely based on an idea in Chevalley and Eilenberg, Cohomology theory of Lie groups and Lie algebras, {\em Trans. AMS}, 63 (1):86--124, 1948.  I am grateful to Alexey Solovyev for pointing it out.}
The fundamental group is the kernel of the universal covering map $G_{sc}\to G$.  It is a discrete subgroup of $G_{sc}$, which is therefore central, since for each fixed $\gamma\in\pi_1(G)$, the mapping $g\mapsto g\to g\gamma g^{-1}$ is a continuous map from $G\to \pi_1(G)$, but $G$ is connected and $\pi_1(G)$ is discrete, so this must reduce to the constant map. We shall show that $\pi_1(G)=H_1(G,\mathbb{Z})$ is pure torsion, and the result will follow since the fundamental group of a compact topological space is finitely-generated.  For this, it is enough to show that the first cohomology group $H^1(G,\mathbb{R})$ vanishes, since there is a perfect pairing between $H_1(G,\mathbb{R})=H_1(G,\mathbb{Z})\otimes_{\mathbb{Z}}\mathbb{R}$ and $H^1(G,\mathbb{R})$.

Let $\omega$ be a closed differential form on $G$.  Then there is a left $G$-invariant closed differential form in the same de Rham cohomology class as $\omega$.  Indeed, first note that for every $g\in G$, $L_g^*\omega$ and $\omega$ lie in the same cohomology class, since $g$ can be connected to the identity by a smooth path.  Thus $\omega$ is cohomologous to its average over the group with respect to the left Haar measure, which is left invariant.

Let $\omega^i$ be a basis of left-invariant one-forms on $G$.  The equation of Maurer--Cartan is
$$d\omega^i = -c_{jk}^i \omega^j\wedge\omega^k$$
where $c_{jk}^i$ are the structure constants of $\mathfrak{g}$ relative to the dual basis of $\omega^i$.  A left-invariant one-form $\omega = \sum_i a_i\omega^i$ is closed if and only if
$$\sum_i a_i c^i_{jk} = 0.$$
This holds if and only if $\omega([X,Y])=0$ for all $X,Y\in\mathfrak{g}$.  That is, $\omega([\mathfrak{g},\mathfrak{g}])=0$.  But $\mathfrak{g}$ is semisimple, and so is equal to its derived subalgebra.  Hence $\omega=0$.
\end{proof}

The following alternative approach is more in the spirit of representation theory\commentx{ and Chapter \ref{Analysis}}.  See Hodge, {\em The theory and applications of harmonic integrals}:

{\em Exercise.}  Prove this in the following different way.  Fix a bi-invariant metric on $G$, and show that the Hodge Laplacian on 1-forms is the Casimir operator for the representation of $G$ on $L^2(G,\mathfrak{g}^*)$.  Show that a 1-form on $G$ is invariant if and only if it is harmonic.  Then apply the Hodge theorem and the argument in the last paragraph to show that $H^1(G,\mathbb{R})=0$.

\section{Weight and root lattices}\index{Weight lattice}\index{Root lattice}\index{Unit lattice}
Let $G$ be a compact connected semisimple Lie group and $T\subset G$ a maximal torus with Lie algebra $i\mathfrak{t}_0$.  The exponential map $\exp : i\mathfrak{t}_0\to T$ is a homomorphism of the additive group $i\mathfrak{t}_0$ to the multiplicative group $T$.  The kernel is a lattice, $\ker\exp = 2\pi i L_G \subset i\mathfrak{t}_0$, called the {\em unit lattice}. The dual lattice to $L_G$ is defined by 
$$\Lambda_G = \{\alpha\in \mathfrak{t}_0^*\mid \alpha(x) \in\mathbb{Z} \text{\ for all $x\in L_G$}\},$$ 
and is called the {\em weight lattice}.

Let $G_{sc}\xrightarrow{\pi} G$ be the simply connected cover of $G$, and let $G_{\ad}\subset GL(\mathfrak{g}_0)$ be the image of the adjoint representation (also a compact Lie group).  Since $\ker\pi$ is a discrete normal subgroup of $G_{sc}$, it must be central.\footnote{Indeed, let $\Gamma=\ker\pi$.  If $\gamma\in\Gamma$, then $\gamma(g) = g^{-1}\gamma g\in\Gamma$ defines a continuous function from the connected space $G$ to the discrete space $\Gamma$.  It must therefore be constant: $\gamma(g)=g^{-1}\gamma g\equiv\gamma$ for all $g$.  So $\gamma\in Z(G)$.}  Therefore there are surjective maps
$$G_{sc}\xrightarrow{\pi} G \xrightarrow{\ad} G_{\ad}.$$
All of these groups have the same Lie algebra and $\pi$ and $\ad$ commute with the exponential maps.  Moreover $\pi^{-1}(T)$ and $\ad(T)$ are both maximal tori in $G_{sc}$ and $G_{\ad}$, respectively.  Hence, relative to these tori, there is an inclusion of the kernels of the exponential maps:
$$L_{G_{sc}}\subset L_G\subset L_{G_{\ad}}$$
and dually
$$\Lambda_{G_{\ad}}\subset\Lambda_G\subset\Lambda_{G_{sc}}.$$
The center of $G$ and fundamental group of $G$ can be determined from these various lattices:
\begin{align*}
Z(G) &= L_G/L_{G_{\ad}} = \Lambda_{G_{\ad}}/\Lambda_G\\
\pi_1(G) &= L_{G_{sc}}/L_G = \Lambda_G/\Lambda_{G_{sc}}.
\end{align*}
So, to classify all possible compact groups with a given Lie algebra, it is enough to study these lattices.

Here we fix some notation that will help to relate these lattices to the standard ones that arise in representation theory.  If $\Phi=\{\alpha_1,\dots,\alpha_r\}$ is any root system in a Euclidean space $E$, then let $L(\Phi)$ be the lattice generated by $\Phi$ (as a $\mathbb{Z}$-module).  Let $L(\Phi^\vee)$ denote the lattice generated by the dual root system $\Phi^\vee = \{\alpha_1^\vee,\dots,\alpha_r^\vee\}$, where $\alpha^\vee\in E^*$ is defined by
$$\alpha^\vee(\beta) = 2\frac{(\alpha,\beta)}{(\alpha,\alpha)}.$$
Finally, let $\widehat{L}(\Phi)$ be the {\em dual lattice}, defined by
$$\widehat{L}(\Phi) = \{v\in E^* \mid v(\alpha) \in\mathbb{Z}\ \ \text{for all $\alpha\in\Phi$}\}.$$

Specialize now to the case where $\Phi = \Phi(\mathfrak{g}_0,\mathfrak{t}_0)\subset i\mathfrak{t}_0^*$ is  the root system for the adjoint representation of $\mathfrak{g}_0$ on $\mathfrak{t}_0$.  Then
\begin{align*}
Q &= \Lambda_{G_{ad}} = L(\Phi)\\
P &= \Lambda_{G_{sc}} = \widehat{L}(\Phi^\vee)\\
Q^\vee &= L_{G_{sc}} = L(\Phi^\vee)\\
P^\vee &= L_{G_{ad}} = \widehat{L}(\Phi).
\end{align*}
The letters $P$ for the weight lattice and $Q$ for the root lattice here are traditional.  The last equation in particular means that the coroot lattice $P^\vee$ is generated as follows.  Let $\Delta=\{\alpha_1,\dots,\alpha_r\}$ be a base of the root system $\Phi$.  The {\em fundamental coweights} are $\lambda_{\alpha_1}^\vee,\dots,\lambda_{\alpha_r}^\vee$ defined by
$$\lambda_{\alpha_i}^\vee(\alpha_j) = \delta_{ij}.$$
Then $P^\vee$ is generated by $\{\lambda_{\alpha_1}^\vee,\dots,\lambda_{\alpha_r}^\vee\}$.

\section{The Weyl group and affine Weyl group}\index{Weyl group}
The normalizer of $T$ in $G$ acts on $T$ by conjugation, and this action lifts to the adjoint action on $i\mathfrak{t}_0$ which preserves the lattice $2\pi i L_G$.  Thus $\Ad : N_G(T) \to GL(2\pi i L_G)$.  Since the latter group is discrete, the identity component $N_G(T)_e$ acts trivially of $2\pi i L_G$ and therefore also on $2\pi i L_G\otimes_{\mathbb{Z}}\mathbb{R}=i\mathfrak{t}_0$.  Thus $N_G(T)_e\subset C_G(T)$.  Since $T$ was a maximal torus, $C_G(T)=T$, and so $N_G(T)_e=T$.

The {\em Weyl group} of $G$ is defined by
$$W_G = N_G(T)/T.$$
This is a finite group, since it is discrete and $G$ is compact.  The Weyl group can also be described as a group generated by reflections.  For $\alpha\in\mathfrak{t}_0$, let $\sigma_\alpha$ be the reflection in the hyperplane orthogonal to $\alpha$.

Equivalence of this definition with the usual one involving reflections can be found in Helgason, Ch. VII, \S 2, Corollary 2.13.
\begin{proposition}
Let $\Phi = \Phi(\mathfrak{g}_0,\mathfrak{t}_0)$ be a set of roots for $G$.  Then $W_G$ is generated by the reflections $\sigma_\alpha$ for $\alpha\in \Phi$.
\end{proposition}



\subsection{Affine Weyl group}\index{Weyl group!affine}
The affine Weyl group of $G$ is defined by\footnote{See Macdonald, {\em Symmetric functions and orthogonal polynomials}, AMS University Lecture Series, Volume 12, 1998. pp. 39--40.}
$$\widetilde{W}_G = W_G\ltimes\tau(L_G)$$
where $\tau(L_G)$ is the group of translations of $L_G$ on $\mathfrak{t}_0$.  The affine Weyl group is generated by the reflections $\sigma_{\alpha,n}$ in the affine hyperplane
$$H_{\alpha,n} = \{x\in\mathfrak{t}_0 \mid \alpha(x) = n\}$$
for $\alpha\in\Phi(\Lambda_G)$, a root system of $\Lambda_G$, and $n\in\mathbb{Z}$.  Indeed, for each such $\alpha$
$$\sigma_{\alpha,n} = \tau(n\alpha^\vee)\circ\sigma_{\alpha,0}\quad\text{and}\quad \tau(n\alpha^\vee) = \sigma_{\alpha,n}\circ\sigma_{\alpha,0}.$$

The {\em alcoves} of the affine Weyl group are the connected components of
$$\mathfrak{t}_0 \setminus \bigcup_{\alpha\in\Phi(\Lambda_G),n\in\mathbb{Z}} H_{\alpha,n}.$$
Each alcove is a fundamental domain for the affine Weyl group, meaning that each orbit of $\widetilde{W}_G$ either lies on a {\em wall} (one of the hyperplanes $H_{\alpha,n}$) or contains exactly one element of the alcove.  There is a unique fundamental alcove of highest weight relative to choice of positive roots $\Phi^+(\Lambda_G)$ for $\Lambda_G$:
$$\Delta_G = \{x\in\mathfrak{t}_0\mid 0<\alpha(x)<1\ \ \text{for all $\alpha\in\Phi^+(\Lambda_G)$}\}$$

\subsection{Example: $SO(3)$ and $SU(2)$}  We consider the examples $SO(3)$ and its simply connected cover $SU(2)$, both of which have the same Lie algebra $\mathfrak{su}(2)$ which is generated by the Pauli matrices
$$\frac{1}{\sqrt{2}}\begin{pmatrix}0&1\\-1&0\end{pmatrix},\quad\frac{1}{\sqrt{2}}\begin{pmatrix}0&i\\i&0\end{pmatrix}, \quad\frac{1}{\sqrt{2}}\begin{pmatrix}i&0\\0&-i\end{pmatrix}.$$
A maximal toral subgroup of $SU(2)$ is the group of diagonal matrices
$$ \begin{pmatrix}e^{i\theta}&0\\0&e^{-i\theta}\end{pmatrix}$$
with Lie algebra spanned by the matrix
$$iv=\frac{i}{\sqrt{2}}\begin{pmatrix}1&0\\0&-1\end{pmatrix}.$$
With this generator, $\ker\exp\cong 2\pi i\sqrt{2}\mathbb{Z}$ and so $Q^\vee = L_{SU(2)}=\sqrt{2}\mathbb{Z}$.  The affine Weyl group is generated by reflections in the hyperplanes $H_{\sqrt{2},\,n} = n/\!\sqrt{2}$.  So the fundamental domain for the action of the affine Weyl group is the interval $[0,1/\!\sqrt{2}]$.

With the same conventions, since $SU(2)$ is a double cover of $SO(3)$, the lattice $L_{SO(3)} = \mathbb{Z}/\!\sqrt{2}$.  The affine Weyl group is generated by reflections in the hyperplanes $H_{2\!\sqrt{2},\,n}=\frac{1}{2\!\sqrt{2}}$.  The fundamental domain is $[0,1/2\!\sqrt{2}]$.

\subsection{Example: $\operatorname{SO}(5)$ and $\operatorname{Sp}(2)$} We consider next the orthogonal group $SO(5)$ in 5-dimensions and its simply connected double-cover $\operatorname{Sp}(2)$ which is the group of $2\times 2$ quaternionic matrices $A$ preserving the Hermitian form $\langle x,y\rangle = x_1\overline{y}_1+x_2\overline{y}_2$ on $\mathbb{H}^2$.  Equivalently, 
$$\operatorname{Sp}(2) = \{ A\in \operatorname{GL}(2,\mathbb{H}) \mid A^*A=I\}$$
where $*$ is the quaternionic conjugate-transpose.  This is the compact real form of the symplectic group $\operatorname{Sp}_4(\mathbb{C})$.  With respect to its natural action on $\mathbb{H}^2$, the orbit of the vector $(1\ \ 0)^T$ under the action of $\operatorname{Sp}(2)$ is the unit sphere $S^7$, and the stabilizer of  $(1\ \ 0)^T$ is $SU(2)\cong S^3$. So topologically, $\operatorname{Sp}(2)$ is an $S^3$ bundle over $S^7$, and therefore is simply connected by the long exact homotopy sequence.

The Lie algebra is
$$\mathfrak{sp}(2) = \left\{\begin{pmatrix}x&z\\-\bar{z} & y\end{pmatrix} \mid x,y\in\im\mathbb{H}, z\in\mathbb{H}\right\}.$$
A maximal toral subalgebra is (using the standard $i,j,k$ generators of $\mathbb{H}$):
$$i\mathfrak{t}_0 = \left\{\begin{pmatrix}ix&0\\0 & iy\end{pmatrix} \mid x,y\in\mathbb{R}\right\}.$$
Now on $i\mathfrak{t}_0$, the kernel of the exponential map is $\ker\exp = 2\pi i L_{\operatorname{Sp}(2)}$ where
$$L_{\operatorname{Sp}(2)} = \begin{pmatrix} 1 & 0 \\ 0 & 0\end{pmatrix}\mathbb{Z}\,\oplus\, \begin{pmatrix} 0 & 0 \\ 0 & 1\end{pmatrix}\mathbb{Z}.$$
This is written as the lattice spanned by a single short coroot and long coroot:
$$Q^\vee = L_{\operatorname{Sp}(2)} = \begin{pmatrix} 1 & 0 \\ 0 & 0\end{pmatrix}\mathbb{Z}\,\oplus\, \begin{pmatrix} -1 & 0 \\ 0 & 1\end{pmatrix}\mathbb{Z} = \alpha^\vee\mathbb{Z}\oplus\beta^\vee\mathbb{Z}.$$
The root system is isomorphic to $C_2$. It is convenient to normalize the Euclidean structure so that the short coroot $\alpha^\vee$ has length $\sqrt{2}$.  Then $Q^\vee$ can be realized as the lattice in Euclidean $\mathbb{R}^2$ by setting $\alpha^\vee = (\sqrt{2},0)$ and $\beta^\vee=(-\sqrt{2},\sqrt{2})$.  The reciprocal lattice $Q=(Q^\vee)^\vee$ is generated by $\alpha=\alpha^\vee=(\sqrt{2},0)$ and $\beta=(-\sqrt{2}/2,\sqrt{2}/2)$.  The dual of this lattice is $P^\vee$ which is generated by $\lambda_\alpha^\vee =(\sqrt{2}/2,\sqrt{2}/2)$ and $\lambda_\beta^\vee=(0,\sqrt{2})$.  Finally, $P$ is generated by $\theta=\lambda_\alpha=(\sqrt{2},\sqrt{2})$ (the {\em highest root}) and $\lambda_\beta=\lambda_\beta^\vee = \theta_s = (0,\sqrt{2})$ (the {\em highest short root}).

The affine Weyl group for $\operatorname{Sp}(2)$ is generated by the affine reflections in the roots of $P$: $\sigma_{\theta,n}, \sigma_{\beta,n}$.  The first of these is reflection across the line $x+y=n/\!\sqrt{2}$ and the second is reflection across the line $y=n/\!\sqrt{2}$.  The fundamental alcove is the region in the first quadrant bounded by $y=x$ and $y=-x+1/\!\sqrt{2}$.

The affine Weyl group for $\operatorname{SO}(5)$ is generated by the affine reflections in the roots of $Q$: $\sigma_{\alpha,n}$ in the line $x=n/\!\sqrt{2}$ and $-x+y=n\sqrt{2}$, $\sigma_{\beta,n}$.  The fundamental alcove is the region in the first quadrant bounded by $y=x$ and $y=1/2\sqrt{2}$.

\section{Determining the center of $G$}\index{Center}
\begin{definition} 
Let $\Delta = \{\alpha_1,\dots,\alpha_n\}$ be a base for the root system $\Phi$ and let $\theta\in i\mathfrak{t}_o^*$ be the highest root.  A coweight $\lambda_{\alpha_i}^\vee\in P^\vee$ is called {\em minuscule} if $\lambda_{\alpha_i}^\vee(\theta)=1$
\end{definition}

In other words, the minuscule coweights are $\lambda_{\alpha_i}^\vee$ for those $\alpha_i$ whose coefficient in the expansion of the heighest root is one.  Geometrically, the minuscule coweights are the fundamental coweights $\lambda_{\alpha_i}^\vee$ that are vertices of the fundamental alcove $W_{G_{sc}}$ since the fundamental alcove is the portion of the Weyl chamber bounded on one side by the single hyperplane $H_{\theta,1}$ (since $\theta$ is the highest root).

\begin{lemma}
There is a one-to-one correspondence between the nontrivial elements of $Z(G_{sc})$ and the minuscule coweights.
\end{lemma}

\begin{proof}
Any coweight exponentiates to an element of the center. Indeed, if $v\in Q^\vee$, then $\alpha(v)\in\mathbb{Z}$ for all roots $\alpha\in Q$.  So $\Ad e^{2\pi i v}$ acts as the identity on each of the eigenspaces $\mathfrak{g}_\alpha$.  Conversely, in order for  $\Ad e^{2\pi i v}$ to act trivially, $\alpha(v)\in\mathbb{Z}$ for all $\alpha\in Q$.  Thus there is a mapping from the set of minuscule coweights into the center of $G_{sc}$.

We claim that this is a bijection onto $Z(G_{sc})\setminus\{e\}$.  A coweight is minuscule if and only if it is a coweight that is a vertex of the fundamental alcove.  But the fundamental alcove contains a representative from each conjugacy class of $G_{sc}$, and so the mapping is surjective.  The mapping is injective since... \commentx{Finish proof.}
\end{proof}

\subsection{Dynkin diagrams with highest roots}\index{Dynkin diagrams}
Here are the Dynkin diagrams for the simple Lie algebras with their highest root labelled by positive integers under the corresponding vertices.

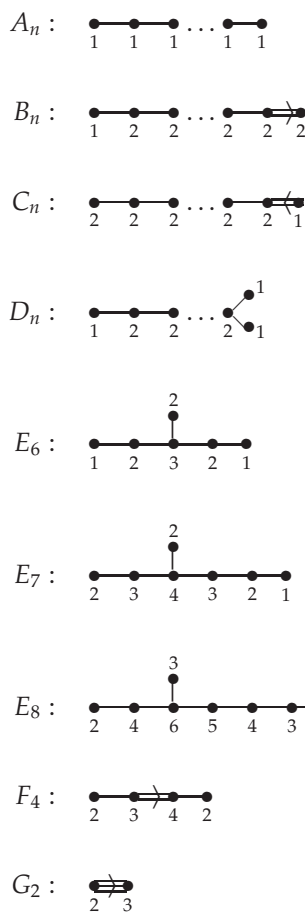
\begin{align*}
A_n:&\quad \underset{1}{\bullet}\!\!\onebar\!\!\underset{1}{\bullet}\!\!\onebar\!\!\underset{1}{\bullet}\dots\underset{1}{\bullet}\!\!\onebar\!\!\underset{1}{\bullet}\\
&\\
B_n:&\quad\underset{1}{\bullet}\!\!\onebar\!\!\underset{2}{\bullet}\!\!\onebar\!\!\underset{2}{\bullet}\dots\underset{2}{\bullet}\!\!\onebar\!\!\underset{2}{\bullet}\!\!\rightbars\!\!\underset{2}{\bullet}\\
&\\
C_n:&\quad\underset{2}{\bullet}\!\!\onebar\!\!\underset{2}{\bullet}\!\!\onebar\!\!\underset{2}{\bullet}\dots\underset{2}{\bullet}\!\!\onebar\!\!\underset{2}{\bullet}\!\!\leftbars\!\!\underset{1}{\bullet}\\
&\\
D_n:&\quad \underset{1}{\bullet}\!\!\onebar\!\!\underset{2}{\bullet}\!\!\onebar\!\!\underset{2}{\bullet}\dots\underset{2}{\bullet}\!\!^{\diagup^{\!\!\textstyle{\bullet}^1}}_{\diagdown_{\!\!\textstyle{\bullet}_1}}\\
&\\
E_6:&\quad \underset{1}{\bullet}\!\!\onebar\!\!\underset{2}{\bullet}\!\!\onebar\!\!\underset{3}{\bullet}^{{}^{\!\!\!\!\!\displaystyle{\arrowvert}^{\displaystyle{\!\!\overset{2}{\bullet}}}}}\!\!\!\onebar\!\!\underset{2}{\bullet}\!\!\onebar\!\!\underset{1}{\bullet}\\
&\\
E_7:&\quad \underset{2}{\bullet}\!\!\onebar\!\!\underset{3}{\bullet}\!\!\onebar\!\!\underset{4}{\bullet}^{{}^{\!\!\!\!\!\displaystyle{\arrowvert}^{\displaystyle{\!\!\overset{2}{\bullet}}}}}\!\!\!\onebar\!\!\underset{3}{\bullet}\!\!\onebar\!\!\underset{2}{\bullet}\!\!\onebar\!\!\underset{1}{\bullet}\\
&\\
E_8:&\quad \underset{2}{\bullet}\!\!\onebar\!\!\underset{4}{\bullet}\!\!\onebar\!\!\underset{6}{\bullet}^{{}^{\!\!\!\!\!\displaystyle{\arrowvert}^{\displaystyle{\!\!\overset{3}{\bullet}}}}}\!\!\!\onebar\!\!\underset{5}{\bullet}\!\!\onebar\!\!\underset{4}{\bullet}\!\!\onebar\!\!\underset{3}{\bullet}\!\!\onebar\!\!\underset{2}{\bullet}\\
&\\
F_4:&\quad \underset{2}{\bullet}\!\!\onebar\!\!\underset{3}{\bullet}\!\!\rightbars\!\!\underset{4}{\bullet}\!\!\onebar\!\!\underset{2}{\bullet}\\
&\\
G_2:&\quad  \underset{2}{\bullet}\!\!\!\rightthreebars\!\underset{3}{\bullet}
\end{align*}

\subsection{Elements of order $2$}  From the fundamental alcove, it is possible to identify up to conjugacy all of the elements of order two in $G_{ad}$.   An element of order two in $G_{ad}$ is conjugate to some one $e^{2\pi i\sigma}$ where $\sigma$ is an element of $\Delta_G$ such that $2\sigma\in L_{G_{ad}}$.  This can only occur at a vertex of the fundamental alcove.  Therefore:

\begin{proposition}\footnote{\commentx{The proof is unconvincing.}  See Borel, {\em Semisimple groups and Riemannian symmetric spaces}, Proposition II.4.4, p. 39.}
Let $\theta = \sum_{\alpha_j\in\Delta} d_j\alpha_j$ be the highest root.  Any element of order $2$ in $G_{ad}$ is conjugate to one of the $e^{\pi i\lambda_{\alpha_j}^\vee}$ for $d_j=1$ or to $e^{2\pi i\lambda_{\alpha_j}^\vee}$ for $d_j=2$.
\end{proposition}

\section{Symmetric spaces of type II}\index{Symmetric space!type II}
Let $(\mathfrak{g}_o\oplus\mathfrak{g}_o,\operatorname{flip})$ be an irreducible orthogonal symmetric Lie algebra of type II.  The $\widetilde{G}$ be a simply connected Lie group with Lie algebra $\mathfrak{g}_o$.  As a homogeneous space $\widetilde{G}=\widetilde{G}\times\widetilde{G}/\Delta(\widetilde{G})$.  Let $M$ be a connected Riemannian globally symmetric space with local Lie algebra $(\mathfrak{g}_o\oplus\mathfrak{g}_o,\operatorname{flip})$.  Then there is a covering map $\pi : \widetilde{G}\to M$, since $\widetilde{G}$ is the simply-connected symmetric space with the given local Lie algebra.  There is also a cover $\phi : \widetilde{G}\times\widetilde{G}\to I(M)_e$ with $\phi_{*,e} = \operatorname{id}_{\mathfrak{g}_o\oplus\mathfrak{g}_o}$, and these two covers are compatible in the sense that the diagram commutes:

$$ 
\xymatrix{
\widetilde{G}\ar[d]^{\pi}&\widetilde{G}\times\widetilde{G}\ar[l]\ar[d]^\phi&\ar[l]^{\ge} \phi^{-1}(I(M)_{e,o})\\
M&I(M)_e\ar[l]&\ar[l]^{\ge} I(M)_{e,o}.
}
$$

Now, $\Delta(\widetilde{G})$ is a connected Lie group with the same Lie algebra as $\phi^{-1}(I(M)_{e,o})$, whose identity component is simply connected, and therefore $\Delta(\widetilde{G})=\phi^{-1}(I(M)_{e,o})_e\trianglelefteq \phi^{-1}(I(M)_{e,o})$ is a normal subgroup, and
$$\Gamma := \pi^{-1}(o) \cong \frac{\phi^{-1}(I(M)_{e,o})}{\Delta(\widetilde{G})}$$
carries a group structure.  Thus $\Gamma\subset\widetilde{G}$ is a subgroup.

Let $\gamma\in\Gamma$.  Since $\pi(\gamma)=o$, there is a diagram
$$ 
\xymatrix{
\gamma \ar[d]&\ar[l] (\gamma,e)\ar[d]^\phi\\
o=\phi(\gamma,e).o & \ar[l]\phi(\gamma,e)
}
$$

Let $g\in\widetilde{G}$ be given.  Then note that $\phi(g,g).o=o$ (since $(g,g)\in\Delta(\widetilde{G})$ which covers the identity component of $I(M)_o$), so
$$ 
\xymatrix{
g\gamma g^{-1}\ar[d]&\ar[l] (g\gamma,g)\ar@{=}[r]\ar[d]^\phi&(g,g)(\gamma,e)\\
o=\phi(g,g)\phi(\gamma,e).o & \ar[l]\phi(g,g)\phi(\gamma,e)&
}
$$

Hence $g\gamma g^{-1}\in\Gamma$: so $\Gamma$ is a normal subgroup of $\widetilde{G}$.  On the other hand, $\Gamma$ is discrete, so it follows that $\Gamma\subset Z(\widetilde{G})$, since normal discrete subgroups of connected Lie groups are central.  This proves the following theorem:

\begin{theorem}
The connected Riemannian globally symmetric spaces of type II are exactly the compact Lie groups.
\end{theorem}

\section{Relative root systems}
\commentx{Why do we start again with noncompact type only to move on to compact type again?  Is there a clearer exposition?}

Let $(\mathfrak{g}_o,s_o)$ be a reductive orthogonal symmetric Lie algebra of noncompact type.  Let $\theta$ be the Cartan involution on $\mathfrak{g}$ defined by conjugation with respect to the compact real form of $\mathfrak{g}_o$. Let $\mathfrak{a}_o\subset \mathfrak{p}_o$ be a flat subspace.  Our first goal is to discuss a decomposition of $\mathfrak{g}_o$ into eigenspaces for $\ad\mathfrak{a}_o$:
$$\mathfrak{g}_o = C_{\mathfrak{g}_o}(\mathfrak{a}_o)\oplus\left(\bigoplus_{\alpha\in\Phi(\mathfrak{g}_o,\mathfrak{a}_o)} \mathfrak{g}_{o,\alpha}\right)$$
where $\Phi(\mathfrak{g}_o,\mathfrak{a}_o)\subset \mathfrak{a}_o^*$.  The roots here are in fact real, since for all $x\in\mathfrak{p}_o$, $\ad x$ is self-adjoint with respect to $B_\theta$.

The centralizer $C_{\mathfrak{g}_o}(\mathfrak{a}_o)$ can be decomposed as $\mathfrak{l}_o\oplus\mathfrak{a}_o$ where $\mathfrak{l}_o\subset \mathfrak{k}_o$.  Let $\mathfrak{t}_o\subset\mathfrak{l}_o$ be a maximal abelian subalgebra of $\mathfrak{l}_o$ that acts semisimply on $\mathfrak{k}_o$.

\begin{proposition}
$\mathfrak{h}_o=\mathfrak{t}_o\oplus\mathfrak{a}_o\subset\mathfrak{g}_o$ is a Cartan subalgebra that is stable under $s_o$.
\end{proposition}

\begin{proof}
By complexifying, it is enough to show that $\mathfrak{h}\subset\mathfrak{g}$ is a Cartan subalgebra.  Since $\mathfrak{g}$ is reductive, we must verify
\begin{enumerate}
\item For all $H\in\mathfrak{h}$, $\ad H$ is diagonalizable over $\mathbb{C}$.
\item $\mathfrak{h}\subset\mathfrak{g}$ is maximal abelian.
\end{enumerate}
The first of these follows since $\mathfrak{t}_o$ and $\mathfrak{a}_o$ both act semisimply on $\mathfrak{g}_o$ and commute with each other.  For the second, suppose $Z\in \mathfrak{g}=\mathfrak{g}_o\oplus i\mathfrak{g}_o$ commutes with $\mathfrak{h}_o$.  Writing $Z=X+iY$ where $X,Y\in\mathfrak{g}_o$, wew have $[X,\mathfrak{h}_o]=[Y,\mathfrak{h}_o]=0$.  So we can assume without loss of generality that $Z\in\mathfrak{g}_o$.  Since $[Z,\mathfrak{h}_o]=0$, $[s_o(Z),\mathfrak{h}_o]=0$ as well.  So $Z+s_oZ\in C_{\mathfrak{g}_o}(\mathfrak{t}_o)\cap\mathfrak{k}_o = \mathfrak{t}_o$.  But we also have $Z-s_oZ\in C_{\mathfrak{g}_o}(\mathfrak{a}_o)\cap\mathfrak{p}_o = \mathfrak{a}_o$.  Hence $Z\in\mathfrak{t}_o\oplus\mathfrak{a}_o=\mathfrak{h}_o$.
\end{proof}

{\em Remark.}  The subalgebra $\mathfrak{t}_o$ is a maximal toral subalgebra when regarded as a subalgebra of $\mathfrak{k}_o$.

The system $\Phi(\mathfrak{g}_o,\mathfrak{a}_o)$ is called a {\em relative root system}.  Decompose $\mathfrak{g}$ in terms of the Cartan subalgebra $\mathfrak{h}$:\index{Root!relative}
$$\mathfrak{g} = \mathfrak{h}\oplus\left(\bigoplus_{\alpha\in\Phi(\mathfrak{g},\mathfrak{h})}\mathfrak{g}_\alpha\right).$$

$$\mathfrak{l}=\mathfrak{t}\oplus\left(\underset{\alpha|_{\mathfrak{a}_o=0}}{\bigoplus_{\alpha\in\Phi(\mathfrak{g},\mathfrak{h})}}\mathfrak{g}_\alpha\right)$$.  

The eigenspaces in $\mathfrak{g}_o$ are
$$\mathfrak{g}_{o,\alpha}=\mathfrak{g}_o\cap\left(\underset{\lambda|_{\mathfrak{a}_o=\alpha}}{\bigoplus_{\lambda\in\Phi(\mathfrak{g},\mathfrak{h})}}\mathfrak{g}_\lambda\right)$$

$$\xymatrix{
i\mathfrak{t}_o^*\oplus\mathfrak{a}_o^* \ar[r]^{\operatorname{Res}|_{\mathfrak{a}_o}}& \mathfrak{a}_o\\
\Phi(\mathfrak{g},\mathfrak{h})\ar[u]^{\subset} \ar[r] & \Phi(\mathfrak{g}_o,\mathfrak{a}_o)\cup\{0\}
}$$

{\em Remarks.}
\begin{enumerate}
\item $\Phi(\mathfrak{g}_o,\mathfrak{a}_o)$ is a possibly {\em non-reduced} root system.
\item The only non-reduced root system $BC_n$: $\{\pm E_i\}_{1\le i\le n}\cup\{\pm E_i\pm E_j\}_{i<j}\cup \{\pm 2E_i\}_{1\le i\le n}$
\item Let $H_1\in\mathfrak{a}_o$ be regular with respect to $\Phi(\mathfrak{g}_o,\mathfrak{a}_o)$ and $H_2\in i\mathfrak{t}_o$ regular with respect to $\{\alpha|_{i\mathfrak{t}_o}\not=0\mid \alpha\in\Phi(\mathfrak{g},\mathfrak{h})\}$.  Scale $H_2$ such that $|\alpha(H_1)|>|\alpha(H_2)|$ for all $\alpha\in\Phi(\mathfrak{g},\mathfrak{h}), \alpha|_{\mathfrak{a}_o}\not=0$.  Then $H=H_1+H_2$ is regular with respect to $\Phi(\mathfrak{g},\mathfrak{h})$ and the sign of $\alpha(H)$ is the signa of $\alpha(H_1)$.
$$H\to \Phi^+(\mathfrak{g},\mathfrak{h})\xrightarrow{\operatorname{Res}|_{\mathfrak{a}_o}} \Phi^+(\mathfrak{g}_o,\mathfrak{a}_o)\cup\{0\}.$$
\item $s_o(\mathfrak{g}_{o,\alpha})=\mathfrak{g}_{o,-\alpha}$.  So we can decompose $\mathfrak{g}_{o,\alpha}\oplus\mathfrak{g}_{o,-\alpha}$ into its eigenspaces under $s_o$: $\mathfrak{k}_{o,\alpha}\oplus \mathfrak{p}_{o,\alpha}$.  If $H\in\mathfrak{a}_o$ is regular with respect to $\Phi(\mathfrak{g}_o,\mathfrak{a}_o)$, then $[H,\mathfrak{k}_{o,\alpha}]=\mathfrak{p}_{o,\alpha}$.  In particular,
$$\dim\mathfrak{g}_{o,\alpha} = \dim\mathfrak{g}_{o,-\alpha} = \dim\mathfrak{k}_{o,\alpha} = \dim\mathfrak{p}_{o,\alpha}.$$
\item Let $H\in\mathfrak{a}_o$.  Then
$$C_{\mathfrak{g}_o}(H) = \mathfrak{l}_o\oplus\mathfrak{a}_o\oplus\left(\underset{\alpha(H)=0}{\bigoplus_{\alpha\in\Phi(\mathfrak{g}_o,\mathfrak{a}_o)}}\mathfrak{g}_{o,\alpha}\right)$$
which we split at
$$C_{\mathfrak{k}_o}(H) = \mathfrak{l}_o\oplus\left(\underset{\alpha(H)=0}{\bigoplus_{\alpha\in\Phi(\mathfrak{g}_o,\mathfrak{a}_o)}}\mathfrak{k}_{o,\alpha}\right),\quad C_{\mathfrak{p}_o}(H) = \mathfrak{a}_o\oplus\left(\underset{\alpha(H)=0}{\bigoplus_{\alpha\in\Phi(\mathfrak{g}_o,\mathfrak{a}_o)}}\mathfrak{g}_{o,\alpha}\right).$$
So, by the previous remark,
$$\dim C_{\mathfrak{k}_o}(H) -\dim C_{\mathfrak{p}_o}(H)  = \dim\mathfrak{l}_o - \dim\mathfrak{a}_o = \dim\mathfrak{k}_o-\dim\mathfrak{p}_o.$$
\item As a result of the previous remark, the following are equivalent:
\begin{itemize}
\item $\dim C_{\mathfrak{k}_o}(H)=\dim\mathfrak{l}_o$
\item $\dim C_{\mathfrak{p}_o}(H) = \dim\mathfrak{a}_o$
\item $H$ is regular with respect to $\Phi(\mathfrak{g}_o,\mathfrak{a}_o)$
\item $\dim C_{\mathfrak{k}_o}(H)=\dim\mathfrak{h}_o$
\item $H\in\mathfrak{g}_o'\cap\mathfrak{a}_o$.
\end{itemize}
\item If an element $X=L+\sum_{\alpha\in\Phi(\mathfrak{g}_o,\mathfrak{a}_o)} (E_\alpha + s_o(E_\alpha))\in\mathfrak{k}_o$ were to normalize $\mathfrak{a}_o$, then note that all $E_\alpha=0$.  Hence $X\in \mathfrak{l}_o = C_{\mathfrak{k}_o}(\mathfrak{a}_o)$.  This shows that $N_{\mathfrak{k}_o}(\mathfrak{a}_o)=C_{\mathfrak{k}_o}(\mathfrak{a}_o)=\mathfrak{l}_o$.  In particular, $N_K(\mathfrak{a}_o)$ and $C_K(\mathfrak{a}_o)$ have the same Lie algebra.  This suggests the following definition.
\end{enumerate}

\begin{definition}
Let $G$ be a connected Lie group with $\Lie(G)=\mathfrak{g}_o$, and let $\sigma:G\to G$ be an involution such that $\sigma_{*,e}=s_o$, and let $K=G_\sigma$.  The geometric {\em relative Weyl group} is
$$W(G,\mathfrak{a}_o) = \frac{N_K(\mathfrak{a}_o)}{C_K(\mathfrak{a}_o)}.$$
The algebraic relative Weyl group is
$$W(\mathfrak{g}_o,\mathfrak{a}_o) = W(\Phi(\mathfrak{g}_o,\mathfrak{a}_o))$$
(the group generated by reflections in the hyperplanes of $\Phi(\mathfrak{g}_o,\mathfrak{a}_o)$.)
\end{definition}

{\em Fact.}  $W(G,\mathfrak{a}_o)=W(\mathfrak{g}_o,\mathfrak{a}_o)\subset\operatorname{GL}(\mathfrak{a}_o)$

\commentx{The following is a total mess.  The goal is to find a system of representatives for $K$-conjugacy classes in $P=\exp\mathfrak{p}_o$.  Compare with the non-compact case where $K\times\exp\mathfrak{p}_o\to G$ was a diffeomorphism.}
{\em Summary.}  Let $(\mathfrak{g}_o,s_o)$ be a type I orthgonal symmetric Lie algebra, $\mathfrak{g}_o=\mathfrak{k}_o\oplus\mathfrak{p}_o$, and let $\mathfrak{a}_o\subset\mathfrak{p}_o$ be a flat subspace.  Let $C_{\mathfrak{k}_o}(\mathfrak{a}_o)=\mathfrak{l}_o$ and $\mathfrak{t}_o\subset\mathfrak{l}_o$ be a maximal toral subalgebra.  Then
$$\mathfrak{h}_o = \mathfrak{t}_o\oplus\mathfrak{a}_o$$
is a Cartan subalgebra of $\mathfrak{g}_o$.\footnote{This is the so-called {\em maximally split} CSA, in contrast to the {\em maximally compact} CSA that starts with a maximal toral $\mathfrak{t}_o\subset\mathfrak{k}_o$ and sets $\mathfrak{h}_o=C_{\mathfrak{g}_o}(\mathfrak{h}_o)$. \commentx{Incorporate this footnote elsewhere.}}
There is a decomposition
$$\mathfrak{g}_o = (\mathfrak{l}_o\oplus\mathfrak{a}_o) \oplus\left(\bigoplus_{\alpha\in\Phi(\mathfrak{g}_o,\mathfrak{a}_o)}\mathfrak{g}_{o,\alpha}\right).$$
The root system $\Phi(\mathfrak{g},\mathfrak{h})$ that arises by complexifying $\mathfrak{g}_o$ and the Cartan subalgebra $\mathfrak{h}_o$ takes real values on restriction to $i\mathfrak{h}_o$, and the relative root system is such that restriction to $i\mathfrak{a}_o$ induces a mapping of root systems $\Phi(\mathfrak{g},\mathfrak{h})\xrightarrow{\operatorname{Res}|_{i\mathfrak{a}_o^*}}\Phi(\mathfrak{g}_o,\mathfrak{a}_o)\cup \{o\}$

Let
$$Q_{s_o} = \bigoplus_{\alpha\in\Phi(\mathfrak{g}_o,\mathfrak{a}_o)} \mathbb{Z}\alpha\subset i\mathfrak{a}_o^*,\qquad Q^\vee_{s_o} = \bigoplus_{\alpha\in\Phi(\mathfrak{g}_o,\mathfrak{a}_o)} \mathbb{Z}{\alpha^\vee}\subset i\mathfrak{a}_o^*.$$
$$P_{s_o}^\vee = \{\lambda \in i\mathfrak{a}_o\mid \langle\lambda, Q_{s_o}\rangle \subset \mathbb{Z} \}$$
Let $G$ be a connected Lie group with $\Lie(G)=\mathfrak{g}_o$.  Let $\sigma:G\to G$ be an involution such that $\sigma_{*,e}=s_o$.  Then
$$\xymatrix{
\mathfrak{h}_o\ar[d]_{\exp} &\supset \ker\exp|_{\mathfrak{h}_o} = 2\pi i L_G\\
H\le G
}$$
where $H$ is a maximal torus in $G$.  We have
$$L_{G,s_o} = L_G\cap i\mathfrak{a}_o, \quad L_{G_{sc},s_o} = Q^\vee_{s_o}, \quad L_{G_{\ad},s_o} = P^\vee_{s_o}$$
$$\widetilde{G}_{G,\mathfrak{a}_o} = W(G,\mathfrak{a}_o)\ltimes L_{G,\mathfrak{a}_o}\quad\text{acts on $i\mathfrak{a}_o$.}$$
Let $\Delta_{G,\mathfrak{a}_o}$ be the relative fundamental domain for $\widetilde{W}_{G,\mathfrak{a}_o}$.

\begin{theorem}
If there exists an involution $\sigma:G\to G$ such that $\sigma_{*,e}=s_o$, then $e^{2\pi i\Delta_{G,\mathfrak{a}_o}}$ is a system of representatives for $K$-conjugacy classes in $P=\exp(\mathfrak{p}_o)$.  Furthermore, $Z(G)\cap A = L_{G,\mathfrak{a}_o}/L_{G_{\ad},\mathfrak{a}_o}$ (here $A=\exp\mathfrak{a}_o$).  Elements of $Z(G)\cap A$ correspond to $\{o\}$ or minuscule coweight vertices of $\operatorname{cl}\Delta_{G,s_o}$.
\end{theorem}

\subsection{Maximal compact Cartan subalgebras}
\commentx{What's this about?  Why are we now discussing noncompact again?}
Let $(\mathfrak{g}_o,s_o)$ be a noncompact orthgonal symmetric Lie algebra, $\mathfrak{g}=\mathfrak{p}_o\oplus\mathfrak{k}_o$.  Let $\mathfrak{a}_o\subset\mathfrak{p}_o$ be a flat subspace, and as above let $C_{\mathfrak{k}_o}(\mathfrak{a}_o)=\mathfrak{l}_o$ and let $\mathfrak{t}_o\subset\mathfrak{l}_o$ be a maximal toral subalgebra.  Then
$$\mathfrak{h}_o = \mathfrak{t}_o\oplus\mathfrak{a}_o$$
is a Cartan subalgebra of $\mathfrak{g}_o$, called a {\em maximally split} Cartan subalgebra.  It maximizes the split part $\mathfrak{a}_o$.

By contrast, if instead $\mathfrak{t}_o\subset\mathfrak{k}_o$ is a maximal toral subalgebra, then
$$\mathfrak{h}_o = C_{\mathfrak{g}_o}(\mathfrak{t}_o) = \mathfrak{t}_o\oplus\mathfrak{b}_o,\qquad (\mathfrak{b}_o\subset\mathfrak{p}_o)$$
is called a {\em maximally compact} Cartan subalgebra.  This maximizes the compact part $\mathfrak{t}_o$.

\section{Fixed points of involutions}

\begin{lemma}
Let $G$ be a compact semisimple Lie group and $\sigma$ an automorphism of $G$.  Then the fixed point set $G_\sigma$ contains a regular element.
\end{lemma}

\begin{proof}
Identical to the proof of Theorem \ref{RegularElementsFlatSubspaces}.\commentx{Check.}
\end{proof}

\begin{theorem}
Let $G$ be a connected simply-connected compact Lie group and $\sigma : G\to G$ an involution.  Then $G_\sigma$ is connected.
\end{theorem}
\begin{proof}
Let $g\in G_\sigma$ be given.  We shall construct a path $\gamma:[0,1]\to G_\sigma$ such that $\gamma(0)=e$ and $\gamma(1)=g$.  There are two cases: (1) $g\in G_{reg}$, (2) $g\not\in G_{reg}$.
\begin{enumerate}
\item If $g\in G_{reg}$, then $C_G(g)=T$ is a maximal torus.  Let $i\mathfrak{t}_o=\Lie(T)$.  Then $g=\exp (2\pi i X)$ for some $X\in \Delta_G$.  Now $\sigma$ acts on $T$, and so $\sigma_{*,e}:i\mathfrak{t}_o\to i\mathfrak{t}_o$.  In particular, since $\sigma(g) =\exp (2\pi i\sigma_{*,e} X) = \exp 2\pi i X = g$, $\sigma_{*,e}X = X+Z$ for some $Z\in L_G$.  But on the other hand $\sigma_{*,e}$ acts as a Weyl reflection on $i\mathfrak{t}_o$, and therefore $\sigma_{*,e}\Delta_G$ is also a fundamental alcove for the affine Weyl group.  But since $X\in\Delta_G$ is regular, it is an interior point of $\Delta_G$, and therefore is also an interior point of $\Delta_G+Z$.  Since the Weyl alcoves are connected and disjoint, $\Delta_G=\Delta_G+Z$.  But this is only possible if $Z=0$.

Hence $\sigma_{*,e}X=X$.  Therefore the curve $\gamma(t) = \exp(2\pi i tX)$ is also fixed by $\sigma$, and $\gamma(0)=e,\gamma(1)=g$, as required.

\item Now suppose that $g\not\in G_{reg}$.  Act via conjugation by $g$ on $G_\sigma$:
$$G_\sigma\xrightarrow{c(g)} G_\sigma, \qquad x\mapsto gxg^{-1}.$$
Then $(G_{\sigma,c(g)})_{reg}\subset (G_\sigma)_{reg}\subset G_{reg}$.  Let $S\subset (G_{\sigma, c(g)})_e$ be a maximal torus.  Then $S$ contains regular points.  So $S\subset C_G(S) = T$, a maximal torus of $G$.  Now $g\in T$, say $g=\exp(2\pi i X)$ where $X\in i\mathfrak{t}_o$.  Let $i\mathfrak{s}_o\subset i\mathfrak{t}_o$ be the Lie algebra of $S$.  Then $\mathfrak{s}_o$ contains regular elements, so it is not contained in any reflection hyperplane of the affine Weyl group.  Thus $X+i\mathfrak{s}_o$ contains regular elements as well.  Suppose that $X+Z$ is regular, $Z\in i\mathfrak{s}_o$.  Then $n=\exp(2\pi i (X+Z))\in G_{\sigma,reg}$.  In particular, there is a path from $e$ to $n$ inside $G_\sigma$ (by case (1)).  Then the curve $\exp(2\pi i (X+tZ))$ is a path from $g$ to $n$ inside $G_\sigma$.  Composing these two paths gives a path from $e$ to $g$.
\end{enumerate}
\end{proof}

{\em Remark.}  The theorem is true for $\sigma$ an arbitrary automorphism.

\begin{corollary}
Let $G$ be a compact, connected, simply-connected Lie group and $g\in G$.  Then $C_G(g)$ is connected.
\end{corollary}

\commentx{What exactly is the goal of this long sequence of remarks?}
{\em Remarks.}  Let $G$ be a connected compact semisimple Lie group and $\sigma:G\to G$ an involution.  Let $K=G_{\sigma,e}$, $P=\exp\mathfrak{p}_o$ and $A=\exp \mathfrak{a}_o$.
\begin{enumerate}
\item $C_G(K)\cap P = C_G(K)\cap A \subset Z(G)\cap A =Z(G)\cap P =: F$.

\begin{proof}
\begin{description}
\item[$C_G(K)\cap P = C_G(K)\cap A$]  One inclusion is clear.  For the other inclusion, let $x\in C_G(K)\cap P$.  There exists a $k\in K$ such that $kxk^{-1}\in A$, since all flat subspaces are conjugate by an element of $K$.  But since $x$ centralizes $K$, $kxk^{-1}=x$.
\item[$C_G(K)\cap A\subset Z(G)\cap A$] Let $x\in C_G(K)\cap A$.  For any $g\in G$, we can write $g=k_1ak_2$ for some $k_1,k_2\in K$ and $a\in A$.  Then $xg=xk_1ak_2=k_1xak_2=k_1axk_2=gx$.  Thus $x\in Z(G)$.
\item[$Z(G)\cap A =Z(G)\cap P$] Clear.
\end{description}
\end{proof}
\item $N_G(K)\cap A = \{ x\in A\mid x^2=F\}$
\begin{proof}
\begin{description}
\item[$\subseteq$]  Let $x\in N_G(K)\cap A$.  If $k\in K$, then $xkx^{-1}=\overline{k}\in K$ as well.  Since $K$ is fixed by $\sigma$ and $A$ is inverted by $\sigma$,
$$xkx^{-1}=\overline{k}=\sigma(\overline{k}) = x^{-1}kx$$
whence $x^2kx^{-2}=k$.  That is, $x^2\in C_G(K)\cap A\subset F$.
\item[$\supseteq$] Let $x\in A$, $x^2\in F$.  Let $k\in K=G_\sigma$.  Then $\sigma(xkx^{-1}) = x^{-1}kx= x(x^{-2}kx^2)x^{-1} = xkx^{-1}$ since $x^2\in F$.  Thus $xkx^{-1}\in K$ as well.  Since this is true for all $k\in K$, $x\in N_G(K)\cap A$.
\end{description}
\end{proof}

\item $N_G(K)\subset A\cdot K$.

\begin{proof}
Let $g\in N_G(K)$.  Write $g=k_1ak_2$ for some $k_1,k_2\in K$ and $a\in A$.  Since $g$ normalizes $K$, $k_1^{-1} g = gk_3$ for some $k_3$; that is, $ak_2=gk_3$.  Hence $g=ak_2k_3^{-1}$.
\end{proof}

\item $N_G(K) = (N_G(K)\cdot A)\cdot K$

\item Ket $G\xrightarrow{\Psi} P$ be the Cartan map $g\mapsto g\sigma(g)^{-1}$, which when passed to the quotient defines the immersion $G/K\xrightarrow{\Psi} P\subset G$.  Note that
$$\xymatrix{
N_G(K)\cap A\ar[r] &F\\
x\ar@{}[u]|{\bigcup\!\!\!\hspace{-0.4pt}\mid}\ar@{|->}[r] & x^2\ar@{}[u]|{\bigcup\!\!\!\hspace{-0.4pt}\mid}
}$$
and $N_G(K)\to F$ induces an isomorphism $N_G(K)/K\xrightarrow{\cong}F$.  Also,
$$\Psi(gn) = \Psi(g)\Psi(n)\qquad\text{for all $g\in G$ and $n\in N_G(K)$.}$$
Indeed,
$$\Psi(gn) = g\underbrace{n\sigma(n)^{-1}}_{\in F}\sigma(g)^{-1} =  g\sigma(g)^{-1}n\sigma(n)^{-1} = \Psi(g)\Psi(n).$$

There are therefore two actions of $G$ on $F$.  For the first,
$$(P,F)\in (p,f)\mapsto pf \in F.$$
For the second, we use the identification of $F$ with $N_G(K)/K$.  Then
$$(G/K,N_G(K)/K) \ni (gK,nK)\mapsto gnK\in N_G(K)/K.$$
But the isomorphism $\Psi$ is equivariant with respect to these actions.
\end{enumerate}

\section{Symmetric spaces of type I}
Let $M$ be an irreducible Riemannian globally symmetric space of type I, with a given basepoint $o\in M$.  Let $(\mathfrak{g}_o,s_o)$ be the local Lie algebra of Killing fields at $o$.  Let $G$ be a simply-connected, connected, compact Lie group with $\Lie(G)=\mathfrak{g}_o$, and $\sigma:G\to G$ an involution for which $\sigma_{*,e}=s_o$.  We have shown that this implies that $G_\sigma=K$ is connected.  The homogeneous space $G/K$ is a connected and simply connected Riemannian globally symmetric space with the same local Lie algebra $(\mathfrak{g}_o,s_o)$.

$$\xymatrix{
e\in P\ar[d]^\pi & G\ar[d]^\phi\ar[l]^\Psi & \ge \phi^{-1}(I(M)_{e,o}) &\ge& K=\phi^{-1}(I(M)_{e,o})_e\\
o\in M & I(M)_e\ar[l] &\ge I(M)_{e,o} &
}$$

The abelian group $F$ contains as a subgroup $\Psi(\phi^{-1}(I(M)_{e,o}))$.  But by a diagram chase $\Psi(\phi^{-1}(I(M)_{e,o}))=\pi^{-1}(o)$, and $\pi^{-1}(o)\subset F$ acts on $P$ by covering transformations.  Thus
$$M = \frac{P}{\Psi(\phi^{-1}(I(M)_{e,o}))}.$$
Conversely, if $H\le F$ is a subgroup, then $P/H$ is a Riemannian globally symmetric space with local Lie algebra $\mathfrak{g}_o$.  ($F$ acts properly discontinuously on $P$.)

Summarizing:
\begin{theorem}
If $M$ is an irreducible Riemannian globally symmetric space of type I, then $M=P/H$for $H\le G=Z(G)\cap A$.
\end{theorem}

{\em Remark.}  Let $M$ be a Riemannian globally symmetric space, and let $\gamma$ be a closed geodesic, so that $\gamma(0)=\gamma(1)$, say.  Then
$$\dot{\gamma}(1) = \pi_\gamma \dot{\gamma}(0) = \dot{\gamma}(0).$$
So $\gamma$ loops.  If $M$ has rank $1$, the $\dim\mathfrak{a}_o=1$, and so all lines in $\mathfrak{p}_o$ are conjugate with respect to $\ad\mathfrak{k}_o$.  As a result, all geodesics in $K$ are conjugate with respect to $K$.  In particular, if one geodesic is closed, then all geodesics are closed and have the same length.

The {\em antipodal set} $A_o$ is the set of fixed points of $s_o$, other than $o$ itself.  For a rank one symmetric space, the antipodal set is the set of points that are a maximal distance from $o$.  The antipodal set is also a rank one symmetric space.

\chapter{Hermitian symmetric spaces}\label{HermitianSymmetricSpaces}

\section{Hermitian symmetric Lie algebras}
\begin{proposition}
Let $(\mathfrak{g}_o,s_o)$ be an irreducible simple orthogonal symmetric Lie algebra.  Then $s_o\in\operatorname{Int}(\mathfrak{g}_o)$ if and only if $\operatorname{rank}\mathfrak{g}_o=\dim\mathfrak{a}_o=\operatorname{rank}\mathfrak{k}_o$
\end{proposition}
Here the rank is defined as the dimension of a Cartan subalgebra.  Thus the involution is an interior automorphism if and only if the rank of $\mathfrak{g}_o$ is the rank of the symmetric space.

\begin{proof}
Note that any one of the equalities $\operatorname{rank}\mathfrak{g}_o=\dim\mathfrak{a}_o=\operatorname{rank}\mathfrak{k}_o$ implies the others, since $\mathfrak{a}_o$ lies inside a Cartan subalgebra for $\mathfrak{g}_o$.

For the forward implication, let $X\in\mathfrak{k}_o$ be a regular element.  Then $C_{\mathfrak{g}_o}(X)\subset\mathfrak{g}_o$ is a Cartan subalgebra that is $s_o$-invariant, since $s_oX=X$.  So $s_o|C_{\mathfrak{g}_o}(X) = \operatorname{id}$ and thus $C_{\mathfrak{g}_o}(X)\subset\mathfrak{k}_o$ is a Cartan subalgebra of $\mathfrak{g}_o$.  Thus $\operatorname{rank}\mathfrak{g}_o = \operatorname{rank}\mathfrak{k}_o$.

Conversely, assume that $\operatorname{rank}\mathfrak{g}_o=\operatorname{rank}\mathfrak{k}_o$.  Then any maximal toral subalgebra $\mathfrak{t}_o\subset\mathfrak{k}_o$ is a Cartan subalgebra of $\mathfrak{g}$.  Since $s_o|_{\mathfrak{t}_o}=\operatorname{id}$, it follows that there exists $H\in\mathfrak{t}_o$ such that $s_o=e^{\ad H} : \mathfrak{g}_o\to\mathfrak{g}_o$.  (Indeed, we can construct such an element of $\mathfrak{t}_o$ in terms of the root spaces in $\mathfrak{g}_o$.)\commentx{Check.}
\end{proof}

\begin{proposition}
Let $(\mathfrak{g}_o,s_o)$ be an irreducible semisimple orthogonal symmetric Lie algebra.  Then $\mathfrak{k}_o$ is a maximal proper subalgebra of $\mathfrak{g}_o$.
\end{proposition}

\begin{proof}
Suppose that $\mathfrak{k}_o\subset\mathfrak{l}_o\subset\mathfrak{g}_o$.  Then $\mathfrak{l}_o$ splits as $\mathfrak{l}_o=\mathfrak{k}_o\oplus (\mathfrak{l}_o\cap\mathfrak{p}_o)$.  Now $\mathfrak{l}_o\cap\mathfrak{p}_o$ is a $\mathfrak{k}_o$-invariant subspace, but by assumption $\mathfrak{k}_o$ acts irreducibly on $\mathfrak{p}_o$, so $\mathfrak{l}_o\cap\mathfrak{p}_o=0$ or $\mathfrak{p}_o$.
\end{proof}

\begin{proposition}
Let $(\mathfrak{g}_o,s_o)$ be an irreducible semisimple orthogonal symmetric Lie algebra with Killing form $B$ and positive definite $\mathfrak{k}_o$-invariant form $Q=cB|_{\mathfrak{p}_o\times\mathfrak{p}_o}$.  Let $A:\mathfrak{g}_o\to\mathfrak{g}_o$ be a linear isomorphism such that
\begin{enumerate}
\item $[A,s_o]=0$
\item $A|_{\mathfrak{k}_o}$ is a Lie algebra morphism
\item $A|_{\mathfrak{p}_o} : \mathfrak{p}_o\to\mathfrak{p}_o$ is an equivariant map of $\mathfrak{k}_o$-modules (meaning $A[k,p]=[Ak,Ap]$).
\end{enumerate}
Then there exists $a\not=0$ such that $B(Ax,Ay)=aB(x,y)$ and $[Ax,Ay]=aA[x,y]$ for all $x,y\in\mathfrak{p}_o$.  In particular, if $A$ leaves $Q$ invariant, then $A\in\operatorname{Aut}(\mathfrak{g}_o)$.
\end{proposition}

\begin{proof}
If $c=0$, then $B\equiv 0$, $[\mathfrak{p}_o,\mathfrak{p}_o]=0$, and so we can take $a=1$.  Otherwise, if $c\not=0$, then $B$ is a nondegenerate form on $\mathfrak{p}_o$.  Let $A^*:\mathfrak{p}_o\to\mathfrak{p}_o$ be the adjoint of $A$ with respect to $B$.  Since $B(A\cdot,A\cdot)$ is $\mathfrak{k}_o$-invariant, $A^*A$ commutes with the action of $\mathfrak{k}_o$ on $\mathfrak{p}_o$.  On the other hand $A^*A$ is symmetric with respect to $B$, and therefore has real eigenvalues.   Because the eigenspaces are $\mathfrak{k}_o$-invariant, irreducibility of $\mathfrak{p}_o$ gives $A^*A=a\operatorname{id}_{\mathfrak{p}_o}$ for some real $a\not=0$.  Thus $B(Ax,Ay)=aB(x,y)$ for all $x,y\in \mathfrak{p}_o$.  

Note that since $A|_{\mathfrak{k}_o}$ is a Lie algebra morphism, it preserves $B|_{\mathfrak{k}_o\times\mathfrak{k}_o}$.  Indeed, if $z,w\in\mathfrak{k}_o$, then
\begin{equation}\label{BisAinvariant}
B(Az,Aw) = \tr_{\mathfrak{g}_o}(\ad(Az)\ad(Aw)) =\tr_{\mathfrak{g}_o}(A\ad(z)A^{-1}\,A\ad(w)A^{-1}) = B(z,w).
\end{equation}
Now let $x,y\in\mathfrak{p}_o$ and $z\in\mathfrak{k}_o$.  Then
\begin{align*}
B([Ax,Ay],Az) &= B(Ax,[Ay,Az]) = B(Ax,A[y,z])\\
&=aB(x,[y,z]) = aB([x,y],z)\\
&=aB(A[x,y],Az)
\end{align*}
by \eqref{BisAinvariant}.  Since this is true for all $z\in\mathfrak{k}_o$, by nondegeneracy of $B$ it follows that $[Ax,Ay]=aA[x,y]$.
\end{proof}

\begin{definition}\index{Hermitian!symmetric Lie algebra}
A triple $(\mathfrak{g}_o,s_o,J)$ is said to be a Hermitian symmetric Lie algebra if
\begin{enumerate}
\item $(\mathfrak{g}_o,s_o)$ is orthogonal symmetric
\item $J:\mathfrak{p}_o\to\mathfrak{p}_o$ is a real-linear map with $J^2=-\operatorname{id}$.
\item $J$ is $\mathfrak{k}_o$-invariant: $J[k,p] = [k,Jp]$ for all $k\in\mathfrak{k}_o$, $p\in\mathfrak{p}_o$
\end{enumerate}
\end{definition}

This definition is justified by the following general considerations.  If $V_c$ is a complex vector space of dimension $n$, let $V$ be the underlying real vector space.  Multiplication by $i$ induces a real-linear map $J:V\to V$  satisfying $J^2=-\operatorname{id}$.  Conversely, if $(V,J)$ is a pair consisting of a real vector space $V$ of dimension $2n$ and a mapping $J:V\to V$  satisfying $J^2=-\operatorname{id}$, then $J$ is called a {\em complex structure} on $V$.  If $(V,J)$ is a complex structure, then $V\otimes\mathbb{C}$ splits into $+i$ and $-i$ eigenspaces of $J$.  A complex-valued bilinear form $H$ on $V$ is called {\em Hermitian} if\index{Hermitian!form}
\begin{enumerate}
\item $H(y,x)=\overline{H(x,y)}$
\item $H$ is $\mathbb{R}$-bilinear
\item $H(JX,Y)=iH(X,Y)$
\end{enumerate}

In the present situation, $J$ defines a complex structure on $\mathfrak{p}_o$ and by the previous proposition, $Q$ is $J$-invariant.  It follows that the form defined by
$$H(x,y) = Q(x,y) + iQ(x,Jy)$$
is a Hermitian form on $\mathfrak{p}_o$ that is $\mathfrak{k}_o$-invariant.

\subsection{Necessary and sufficient conditions for an orthogonal symmetric Lie algebra}
Let $(\mathfrak{g}_o,s_o)$  be an irreducible semisimple orthogonal symmetric Lie algebra.  We here investigate necessary and sufficient conditions for the existence of an invariant complex structure $J\in\operatorname{End}_{\mathfrak{k}_o}(\mathfrak{p}_o)$.  There are two cases:

{\bf Case I.}  $\mathfrak{p}$ is an irreducible $\mathfrak{k}_o$-module.

In this case, there can be no complex structure, since the eigenspaces of $J$ are also $\mathfrak{k}_o$-invariant.  We claim that in this case $\mathfrak{k}_o$ has trivial center.  Indeed, if $Z\in Z(\mathfrak{k}_o)$, then $\ad Z:\mathfrak{p}\to\mathfrak{p}$ is $\mathfrak{k}_o$-invariant.  By Schur's lemma, $\ad Z=\lambda\operatorname{id}_{\mathfrak{p}}$ for some $\lambda$.  On the other hand, since $\ad Z$ preserves the quadratic form $Q$, it must be tracefree, so $\lambda=0$.

{\bf Case II.}  $\mathfrak{p}$ is a reducible $\mathfrak{k}_o$-module.

Let $\theta:\mathfrak{p}\to\mathfrak{p}$ be the conjugation map with respect to the real subspace $\mathfrak{p}_o$.  There is no $\theta$-invariant proper submodule of $\mathfrak{p}$, since $(\mathfrak{g}_o,s_o)$ is irreducible and semisimple.  So there exists at least one decomposition $\mathfrak{p} = \mathfrak{p}_1\oplus\mathfrak{p}_2$ into $\mathfrak{k}_o$-submodules such that $\theta(\mathfrak{p}_1)=\mathfrak{p}_2$.

Define a mapping $A:U(1)\to \operatorname{End}_{\mathbb{C}}(\mathfrak{p})$ ($\lambda\mapsto A_\lambda$) by $A_\lambda|_{\mathfrak{p}_1}=\lambda \operatorname{id}$ and  $A_\lambda|_{\mathfrak{p}_2}=\overline{\lambda} \operatorname{id}$.  Then $A$ is a group morphism onto the subgroup of invertible elements.  Moreover, $A_\lambda$ is $\mathfrak{k}_o$-invariant and $\theta$-invariant, and so $A_\lambda :\mathfrak{p}_o\to\mathfrak{p}_o$ is also $\mathfrak{k}_o$ invariant.  Extend $A_\lambda$ to all of $\mathfrak{g}_o$ so that $A_\lambda|_{\mathfrak{k}_o}=\operatorname{id}$.

Then:
\begin{enumerate}
\item $A_\lambda:\mathfrak{g}_o\to\mathfrak{g}_o$ is a linear isomorphism
\item $A|_{\mathfrak{k}_o}$ is a Lie algebra morphism
\item $[A_\lambda,s_o]=0 $
\item $A_\lambda|_{\mathfrak{p}_o}$ is an equivariant map of $\mathfrak{k}_o$-modules
\item $A_\lambda$ preserves $Q$
\end{enumerate}
Hence, by the earlier proposition, $A_\lambda\in\operatorname{Aut}(\mathfrak{g}_o)$.

Now $A:S^1\to \operatorname{Aut}(\mathfrak{g}_o)$ is a one-parameter subgroup.  The derivative at the identity is a derivation of $\mathfrak{g}_o$, which is equal to $\ad X$ for some $X$, by semisimplicity:
$$\ad X = \left.\frac{d}{dt}\right|_{t=0} A_{e^{2\pi i t}} \in\operatorname{Der}(\mathfrak{g}_o).$$
Since $A_\lambda$ commutes with $s_o$, $X\in\mathfrak{k}_o$.  Furthermore, since $A_\lambda$ commutes with $\ad\mathfrak{k}_o$, $X\in Z(\mathfrak{k}_o)$.  In particular, $\dim Z(\mathfrak{k}_o)\ge 1$.

Consider now $J=A_i:\mathfrak{p}_o\to\mathfrak{p}_o$.  This satisfies $J^2=-\operatorname{id}$.  Note that by the construction of $A$, $\mathfrak{p}_1$ is the $+i$-eigenspace of $J$ and $\mathfrak{p}_2$ is the $-i$-eigenspace.  Suppose that $Z\in Z(\mathfrak{k}_o)$.  Then
\begin{itemize}
\item $e^{t\ad(Z)}\in\operatorname{Aut}(\mathfrak{g}_o)$ commutes with $s_o$
\item $e^{t\ad(Z)}|_{\mathfrak{p}_o}$ is a $\mathfrak{k}_o$-invariant map
\item  $e^{t\ad(Z)}:\mathfrak{p}\to\mathfrak{p}$ preserves the $\mathfrak{p}_1\oplus\mathfrak{p}_2$ decomposition of $\mathfrak{p}$.  
\item $e^{t\ad(Z)}$ commutes with $\theta$ (conjugation of $\mathfrak{g}$ with respect to $\mathfrak{g}_o$)
\end{itemize}

Since $\mathfrak{p}_1$ and $\mathfrak{p}_2$ are simple $\mathfrak{k}_o$ modules, it follows that
$$e^{t\ad (Z)}|_{\mathfrak{p}_1} = \lambda_t\operatorname{id},\quad e^{t\ad (Z)}|_{\mathfrak{p}_2} = \overline{\lambda}_t\operatorname{id}$$
for some eigenvalue $\lambda_t$.  Thus $e^{t\ad Z}=A_{\lambda_t}$.  Differentiating at the identity implies that $Z$ is a scalar multiple of $X$.  Thus $Z(\mathfrak{k}_o)$ is one-dimensional.

Write
$$\mathfrak{k}_o = Z(\mathfrak{k}_o)\oplus[\mathfrak{k}_o,\mathfrak{k}_o]$$
Let $i\mathfrak{t}_o\subset [\mathfrak{k}_o,\mathfrak{k}_o]$ be a maximal toral subalgebra, so
$$Z(\mathfrak{k}_o)\oplus i\mathfrak{t}_o\subset\mathfrak{k}_o $$
is a maximal toral subalgebra so $C_{\mathfrak{g}_o}(Z(\mathfrak{k}_o)\oplus i\mathfrak{t}_o)\subset\mathfrak{g}_o$ is a Cartan subalgebra.  On the other hand, 
$$C_{\mathfrak{g}_o}(Z(\mathfrak{k}_o)\oplus i\mathfrak{t}_o)\subset C_{\mathfrak{g}_o}(Z(\mathfrak{k}_o)) \subset C_{\mathfrak{k}_o}(Z(\mathfrak{k}_o)) = \mathfrak{k}_o.$$
So $\operatorname{rank}\mathfrak{k}_o = \operatorname{rank}\mathfrak{g}_o$.  In particular, by ???, $s_o\in\operatorname{Int}(\mathfrak{g}_o)$.

Summarizing,
\begin{theorem}
Let $(\mathfrak{g}_o,s_o)$ be an irreducible semisimple orthogonal symmetric Lie algebra.  The following are equivalent:
\begin{enumerate}
\item $(\mathfrak{g}_o,s_o)$ is Hermitian
\item $\mathfrak{p}$ is redicuble as a $\mathfrak{k}_o$-module
\item $Z(\mathfrak{k}_o)\not=0$.
\end{enumerate}
If these conditions hold, then $\operatorname{rank}\mathfrak{k}_o= \operatorname{rank}\mathfrak{g}_o$ and $s_o\in\operatorname{Int}(\mathfrak{g}_o)$
\end{theorem}

\begin{proposition}
Let $(\mathfrak{g}_o,s_o,J)$ be an irreducible Hermitian symmetric Lie algebra.  Let $G$ be a connected Lie group with $\Lie(G)=\mathfrak{g}_o$ and $K\le G$ a closed subgroup such that $\Lie(K)=\mathfrak{k}_o$.  If $J:\mathfrak{p}_o\to\mathfrak{p}_o$ is $K$-invariant, then $K$ is connected.
\end{proposition}

\begin{proof}
If $(\mathfrak{g}_o,s_o)$ is flat, then the connected group $G$ is a semidirect product $G=K\ltimes\mathfrak{p}_o$, so $K$ is connected.

If $(\mathfrak{g}_o,s_o)$ is noncompact, then $G$ is homeomorphic to $K\times\mathfrak{p}_o$, and so $K$ is connected.

Next suppose that $(\mathfrak{g}_o,s_o)$ is compact.  Then $J=\exp(\ad(Z))$ for some $Z\in Z(\mathfrak{k}_o)$.  Since $J$ is $K$-invariant, $K\le C_G(Z(\mathfrak{k}_o))$ which is connected.  On the other hand, $\Lie C_G(Z(\mathfrak{k}_o)) = C_{\mathfrak{g}_o}(Z(\mathfrak{k}_o)) = \mathfrak{k}_o$.  So $C_G(Z(\mathfrak{k}_o))=K_e$.  Combined with the inclusion $K\le K_e$, so $K=K_e$.
\end{proof}

\begin{theorem}
Let $(\mathfrak{g}_o,s_o,J)$ be a Hermitian symmetric Lie algebra.  Then there is a decomposition
$$ (\mathfrak{g}_o,s_o,J) = (\mathfrak{g}_f,s_f,J_f)\oplus(\mathfrak{g}_1,s_1,J_1)\oplus\cdots\oplus(\mathfrak{g}_t,s_t,J_t).$$
where $(\mathfrak{g}_f,s_f,J_f)$ is flat and irreducible, $(\mathfrak{g}_i,s_i,J_i)$ is semisimple and irreducible.  The decomposition is unique up to order.
\end{theorem}

\begin{proof}
Decompose the Lie algebra in the usual way.  Since $J$ is $\mathfrak{k}_o$-invariant, the components of $\mathfrak{p}_o$ are invariant under $J$ as well.
\end{proof}

\commentx{Some corollary was here, but didn't make sense.}

\section{Hermitian symmetric spaces}
Let $M$ be a $2n$-dimensional real manifold.  Let $J:T^1_0M\to T^1_0M$ be a bundle map such that $J^2=-\operatorname{id}$.  Then $J$ is called an almost complex structure and the pair $(M,J)$ an almost complex manifold.  If $M$ is a complex manifold and $J$ is the operation of multiplication by $i$, then the $(M^{\mathbb{R}},J)$ is an almost complex manifold.  The {\em Nijenhuis tensor} is defined by
$$N(X,Y) = [JX,JY] - J[JX,Y] - J[X,JY] - [X,Y].$$
It is easily checked that $N\in \mathscr{T}^1_2M$.  If $(M,J)$ is associated to a complex manifold, then the integrability of the eigenspaces of $J$ in $T^1_0M\otimes\mathbb{C}$ implies that $N\equiv 0$.  The Newlander--Nirenberg theorem asserts the converse:\footnote{A proof of this in the real-analytic category can be found in Kobayashi and Nomizu, Volume 2.  Since symmetric spaces are analytic, this version of the theorem is sufficient for our purposes.}
\begin{theorem}
$(M,J)$ is a complex manifold if and only if $N\equiv 0$.
\end{theorem}

\begin{definition}
A triple $(M,g,J)$ is said to be almost Hermitian if
\begin{enumerate}
\item $(M,g)$ is a Riemannian manifold
\item $(M,J)$ is an almost complex manifold
\item $g$ is $J$-invariant
\end{enumerate}
\end{definition}

In this situation,
$$H(X,Y) = g(X,Y) + i g(X,JY)$$
is a Hermitian form on $TM$ (defined previously).  Conversely, the metric $g$ is the real part of the form $H$.  Hence, all Hermitian forms on a given almost complex manifold $(M,J)$ arise in this way.

\begin{definition}
An almost Hermitian manifold $(M,g,J)$ is called K\"{a}hler if $\nabla J=0$.
\end{definition}

A more convenient formulation of the K\"{a}hler condition uses the {\em fundamental form}: the two form
$$\Omega(X,Y) = g(X,JY).$$
Since $J^2=-\operatorname{id}$ and $g$ is $J$-invariant, it follows that $\Omega$ is skew-symmetric (so $\Omega\in\Omega^2M$).

\begin{theorem}
Let $(M,g,J)$ be an almost Hermitian manifold.  The following are equivalent:
\begin{enumerate}
\item $(M,g,J)$ is K\"{a}hler
\item $\nabla H =0$
\item $d\Omega=0$
\end{enumerate}
Moreover, if any of these holds, then $(M,J)$ is a complex manifold.
\end{theorem}
\begin{proof}
Straightforward calcuation.
\end{proof}

\begin{definition}
An almost Hermitian manifold $(M,g,J)$ is said to be {\em Hermitian symmetric} at $x\in M$ if $(M,g)$ is globally symmetric at $x$ and $s_x\circ J=J\circ s_x$.  $(M,g,J)$ is Hermitian symmetric if it is Hermitian symmetric at every $x\in M$.
\end{definition}

{\em Notation.}  Let $H(M)=\{ f\in I(M)\mid f\circ J=J\circ f\}$ be the group of Hermitian isometries.  It is a closed subgroup of $I(M)$.

{\em Remark.}  If $(M,g,J)$ is Hermitian symmetric, then the set of transvections is a subset of $H(M)$, since transvections are generated by the global symmetries $s_x$.  But transvections act transitively on $M$, and so
$$M = H(M)_e/H(M)_{e,o}.$$
Also $H(M)_{e,o}\subset I(M)_o$ is a compact Lie group.

\begin{theorem}
Let $(M,g,J)$ be Hermitian symmetric.  Then $M$ is K\"{a}hler (in particular, it is complex).
\end{theorem}
\begin{proof}
$J$ is invariant under the set of all transvections, and therefore $\nabla J=0$.
\end{proof}

\begin{proposition}
Let $(M,g,J)$ and $(\overline{M},\overline{g},\overline{J})$ be Hermitian symmetric spaces.  An isometry $f : M\to \overline{M}$ is a Hermitian isometry if and only if $f_{*,o}\circ J_o = \overline{J}_{\overline{o}}\circ f_{*,o}$
\end{proposition}
(Here $o$ as usual is a fixed base point of $M$ and $\overline{o}$ is its image under $\overline{f}$.)
\begin{proof}
We need to show that $f_*$ respects the complex structure at every point.  Since $f_*$ respects the symmetry at each point, it is equivariant with respect to transvections, and the result follows since
$$J_x = \tau_{[o,x]}\circ J_o\circ\tau_{[x,o]}.$$
\end{proof}

\begin{proposition}
Let $(M,g,J)$ and $(\overline{M},\overline{g},\overline{J})$ be Hermitian symmetric spaces.  The differential at $o$ sets up a natural one-to-one correspondence between isometries $f : M\to \overline{M}$ and the set of Lie algebra isomorphisms $\phi : \mathfrak{g}_o\to \overline{\mathfrak{g}}_{\overline{o}}$ that commute with $s_o$ such that $\phi|_{\mathfrak{p}_o}$ is a Hermitian isometry.
\end{proposition}

{\em Remark.}  If $(M,g,J)$ is a Hermitian symmetric space, then $M=H(M)_e/H(M)_{e,o}$ and $(\Lie H(M), s_{o,*}, J_o)$ is a Hermitian symmetric Lie algebra.

{\em Remark.}  Let $(\mathfrak{g}_o,s_o,J_o)$ be a Hermitian symmetric Lie algebra and $G$ a connected Lie group such that $\Lie(G) =\mathfrak{g}_o$.  Let $\sigma:G\to G$ be an involution such that $\sigma_{*,e}=s_o$.  Let $K$ be an open subgroup of $G_\sigma$ such that $J$ is $\Ad K$-invariant; this is always possible since $J$ is invariant under $\Ad G_{\sigma,e}$.   Then $M=G/K$ is a Riemannian symmetric space, and we can extend $J$ from $o=eK\in G/K$ to a $G$-invariant almost complex structure on all of $M$; this extension is well-defined, since $J$ is $\Ad K$-invariant.  The pair $(g,J)$ is invariant under $G$, so the associated Hermitian metric is also invariant under $G$, and therefore $\nabla J=0$, so $(M,g,J)$ is  K\"{a}hler manifold.

Let $\bar{\sigma}:G/K\to G/K$ be the geodesic symmetry at $o=eK$:
$$\xymatrix{
G\ar[r]^\sigma&G\\
G/K\ar[r]^{\bar{\sigma}}&G/K
} $$
We claim that $\bar{\sigma}$ is Hermitian.  Let $gK=(\exp X)K$, $X\in\mathfrak{p}_o$.  Then the parallel transport
$$T_{\exp(X)} G/K \xrightarrow{\sigma_{*,\exp(X)K}}T_{\exp(-X)K}G/K$$
preserves the Hermitian structure, since $X\in\mathfrak{p}_o$.  Hence $\overline{\sigma}\circ J = J\circ\overline{\sigma}$.  Thus the transvections are Hermitian and $\bar{\sigma}$ is Hermitian, so $(G/K,g,J)$ is a Hermitian symmetric space.

{\em Exercise.}  Let $M$ be a rank one compact Hermitian symmetric space.
\begin{enumerate}[(a)]
\item Show that the second Betti number of $M$ is equal to $1$.  ({\em Hint:}  Consider first how to do this with $\mathbb{CP}^n$.)
\item Use the Kodaira embedding theorem to show that $M$ is projective algebraic.  (This is true for compact Hermitian symmetric spaces of any rank, but requires other methods.)
\end{enumerate}

Let $(\mathfrak{g}_o,s_o)$ be an irreducible Hermitian symmetric space.  Then $K$ is connected and $\dim_{\mathbb{R}} Z(K) = 1$.  Consider the universal cover of $G$
$$
\xymatrix{
\widetilde{G}\ar[d]^\pi& \ge \widetilde{G}_{\sigma,e}\\
G & \ge G_{\sigma,e}=K
}
$$
Also, $\widetilde{G}/\widetilde{G}_{\sigma,e}$ is the universal cover of $G/K$, and so is also a Hermitian symmetric space.  Note $\pi^{-1}(K)\le G_\sigma$ and $J$ is $\pi^{-1}(K)$-invariant.  Hence $\pi^{-1}(K)=\widetilde{G}_{\sigma,e}$.  That is, $G/K=\widetilde{G}/\widetilde{G}_{\sigma,e}$ is simply connected.

{\em Exercise.}  Give a different proof in of this fact in rank one along the following lines.  Prove that a compact Hermitian symmetric space $M$ of rank one has postive curvature, and then apply Synge's theorem to conclude that $M$ is simply-connnected.

\begin{theorem}
Let $M$ be a Hermitian symmetric space.  Then
$$M\cong M_f\times M_1\times\cdots\times M_t$$
(a Hermitian isometry of Hermitian symmetric spaces) where $M_f$ is flat and $M_i$ are irreducible and simply-connected.
\end{theorem}

\begin{proof}
The decomposition of Hermitian symmetric Lie algebras
$$(\Lie(H(M)), s_o,J)\cong (\mathfrak{h}_f,s_f,J_f)\oplus  (\mathfrak{h}_1,s_1,J_1)\oplus\cdots\oplus (\mathfrak{h}_t,s_t,J_t) $$
induces an inclusion
$$G_f\times G_1\times \cdots G_t = G \ge \pi^{-1}(H(M)_{e,o}) $$
where $G\xrightarrow{\pi} H(M)_{e,o}$.  Let $\rho_i : G\to G_i$ be the projection onto the factor $G_i$.  Then $\rho_i( \pi^{-1}(H(M)_{e,o}))$ leaves $J_i$ invariant, and so $K_i=(G_i)_{\sigma,e}$.  (Here $K=K_f\times K_1\times\cdots\times K_t$ is the isotropy group in $G$.)  Then
$$M \cong \frac{G_f}{G_f\cap\pi^{-1}(H(M)_{e,o})} \times \frac{G_1}{K_1}\times\cdots\times\frac{G_t}{K_t}.$$
The factors $G_1/K_1,\dots,G_t/K_t$ are simply connected, by the earlier remark.
\end{proof}

\section{Bounded symmetric domains}
A {\em domain} is an open connected subset of $\mathbb{C}^N$.  A domain is {\em bounded} if it is bounded with respect to the standard Euclidean metric on $\mathbb{C}^N$. Let $\mu$ be the Lebesgue measure on $\mathbb{C}^N$.  Denote by $H(D)$ the subspace of $L^2(D,\mu)$ of square-integrable functions in $D$ that are also holomorphic.

The following key lemma has a number of important corollaries:
\begin{lemma} 
For any compact set $A\subset D$, we have
$$\sup_{z\in A} |f(z)| \le C_A \|f\|_{L^2(D,\mu)}$$
for all $f\in H(D)$.
\end{lemma}
\begin{proof}
Let $r=\operatorname{dist}(A,\partial D)/2$.  By the mean value theorem,
\begin{align*}
|f(z)| &= \frac{1}{|B(z,r)|}\left|\int_{B(z,r)} f(\zeta)\,d\mu(\zeta)\right|\\
&\le \frac{\|f\|_{L^2(B(z,r),\mu)}}{\sqrt{|B(z,r)|}}\qquad\text{by Cauchy--Schwarz}\\
&\le \frac{1}{\sqrt{|B(z,r)|}}\,\|f\|_{L^2(D,\mu)}
\end{align*}
Letting $C_A=\frac{1}{\sqrt{|B(z,r)|}}$ (independent of $z$) and then taking the supremum over $z\in A$ implies the lemma.
\end{proof}

\begin{corollary}
$H(D)$ is a closed subspace of $L^2(D,\mu)$.  Hence $H(D)$ is a Hilbert space with respect to the $L^2$ inner product.
\end{corollary}
\begin{proof}
If $f_n$ is a sequence of functions in $H(D)$ that converges to a function $f\in L^2(D,\mu)$, then by the Lemma, $f_n\to f$ uniformly on compact subsets.  Since the compact limit of holomorphic functions is holomorphic, $f\in H(D)$ as well.
\end{proof}

\begin{corollary}
For all $z\in D$, the evaluation map $\operatorname{ev}_z : f\to f(z)$ is a continuous linear functional on $H(D)$.
\end{corollary}

\begin{proof}
Let $A=\{z\}$ in the lemma.  Then
$$|\operatorname{ev}_zf| \le C_A\|f\|_{H(D)}$$
for all $f\in H(D)$.
\end{proof}

Since the evaluation map is a continuous linear functional on a Hilbert space, the Riesz representation theorem implies that there exists, for each $z\in D$, an element $\phi_z\in H(D)$ such that
$$f(z) = \int_D f(\zeta) \overline{\phi_z(\zeta)}\,d\mu(\zeta).$$
Let $K(z,\overline{\zeta})  = \overline{\phi_z(\zeta)}$.  This is called the {\em Bergman kernel} of $D$.\index{Bergman!kernel}

We note that $K(z,\overline{\zeta}) = \overline{K(\zeta,\overline{z})}$.  Indeed,
\begin{align*}
\overline{K(\zeta,\overline{z})} &= \phi_\zeta(z) = \int_D \phi_\zeta(w)\overline{\phi_z(w)}\,d\mu(w)\\
&= \int_D \overline{K(\zeta,\overline{w})}K(z,\overline{w})\,d\mu(w)\\
&=\overline{\int_D K(\zeta,\overline{w})\overline{K(z,\overline{w})}\,d\mu(w)}\\
&=K(z,\overline{\zeta}).
\end{align*}
It follows that $K$ is holomorphic in its first argument and antiholomorphic in its second.  In particular, $K$ is smooth.

{\em Notation.}  We introduce the Wirtinger derivatives in $\mathbb{C}$:
$$\frac{\partial}{\partial z} := \frac{1}{2}\left(\frac{\partial}{\partial x} - i\frac{\partial}{\partial y}\right),\qquad \frac{\partial}{\partial \bar{z}} := \frac{1}{2}\left(\frac{\partial}{\partial x} + i\frac{\partial}{\partial y}\right).$$
In $\mathbb{C}^N$, a point $z$ has coordinates $(z_1,\dots,z_N)$.  Define
$$Z_j = \frac{\partial}{\partial z_j},\qquad Z_j^* = \frac{\partial}{\partial\bar{z}_j}.$$

\begin{definition}
Let $H = Z_i Z_j^* \log K(z,\overline{z}) dz_i\otimes d\overline{z}_j$ and put $g=\operatorname{Re} H$.  This is the {\em Bergman metric}.\index{Bergman!metric}
\end{definition}

\begin{proposition}
\mbox{}
\begin{enumerate}
\item $(D,g)$ is a K\"{a}hler manifold.
\item If $\phi:D\to D'$ is a biholomorphism, then $(D,g)\xrightarrow{\phi} (D',g')$ is an isometry.
\item In particular, the biholomorphisms of $(D,g)$ to itself are precisely the Hermitian isometries of $(D,g)$.  Denote these by $\mathscr{H}(D)$.
\end{enumerate}
\end{proposition}

\begin{proof}
\mbox{}
\begin{enumerate}
\item The fundamental form is $\Omega = \partial\overline{\partial}K(z,\overline{z}) = Z_i Z_j^* \log K(z,\overline{z}) dz_i\wedge d\overline{z}_j$.  This is closed since $d=\partial + \overline{\partial}$ and $\partial^2=0, \overline{\partial}^2=0, \partial\overline{\partial}=-\overline{\partial}\partial$.
\item By change of variables, for $f\in H(D')$, $w\in D'$, and $z=\phi^{-1}(w)\in D$, we have
\begin{align*}
f(w) &= \int_{D'} f(\zeta)K'(w,\overline{\xi})\,d\mu(\xi) \\
&= \int_D f(\phi(\zeta)) K'(w,\overline{\phi(\zeta)}) |J\phi(\zeta)|^2\,d\mu(\zeta)\\
&= \int_D f(\phi(\zeta)) K'(\phi(z),\overline{\phi(\zeta)}) |J\phi(\zeta)|^2\,d\mu(\zeta) \\
&= f(\phi(z))
\end{align*}
where $J\phi = \det(\partial\phi_i/\partial \zeta_j)_{1\le i,j\le N}$ is the holomorphic Jacobian.  Hence,
$$K(z,\overline{\zeta}) = K'(\phi(z), \overline{\phi(\zeta)})|J\phi(\zeta)|^2.$$
In particular,
\begin{align*}
\log K(z,\overline{z}) &= \log \left[K'(\phi(z), \overline{\phi(\zeta)})|J\phi(\zeta)|^2\right]\\
&= \log K'(\phi(z), \overline{\phi(\zeta)}) + \log J\phi(\zeta) + \log \overline{J\phi(\zeta)}.
\end{align*}
But $\log J\phi$ is holomorphic, and so annihilated by $Z_j^*$ and $\log\overline{J\phi}$ is antiholomorphic, and so annihilated by $Z_i$.  Thus
$$g = \phi^* g',$$
as required.
\item Since Hermitian isometries are biholomorphic, this follows from (2).
\end{enumerate}
\end{proof}

\begin{definition}
Let $D$ be a bounded domain.
\begin{enumerate}
\item $D$ is said to be {\em homogeneous} if $\mathscr{H}(D)$ acts transitively on $D$.
\item $D$ is said to be {\em symmetric} if each $z\in D$ is an isolated fixed point of some involution $\phi\in\mathscr{H}(D)$.
\item A bounded symmetric domain $D$ is called {\em reducible} if it is biholomorphic to a product $D_1\times D_2$ of bounded symmetric domains.  A bounded symmetric domain is {\em irreducible} if it is not reducible.
\end{enumerate}
\end{definition}

\begin{theorem}
If $D$ is a bounded symmetric domain, then $D$ is a Hermitian symmetric space.  (In particular, $D$ is homogeneous.)
\end{theorem}

Indeed, we showed in Chapter \ref{SymmetricSpaces} that $s_x$ is the only the involution with isolated fixed point $x$.  

{\em Examples.}

\begin{enumerate}
\item For $n=1$, a domain is homogeneous if and only if it is simply connected.  In that case, the Riemann mapping theorem assures that every domain is biholomorphic to the disc.  As a symmetric space, the disc is the hyperbolic plane.   The Bergman metric is the hyperbolic metric.
\item For $n=2,3$ every homogeneous bounded domain is symmetric.
\item For $n=4$, the set of bounded symmetric domains is a proper subset of the bounded homogeneous domains (Piatetsky \& Shapiro)
\item In general, if $D$ is a bounded homogeneous domain such that $\mathscr{H}(D)$ is unimodular, then $D$ is symmetric.
\end{enumerate}

\begin{theorem}
Let $D$ be a bounded symmetric domain.  The $D\cong M_1\times\cdots\times M_t$ factors as a product of irreducible non-compact Hermitial symmetric spaces.
\end{theorem}

\begin{proof}
Decompose $D\cong M_f\times M_1\times\cdots\times M_t$ with $M_f$ flat.  Then $\mathbb{C}^N=\widetilde{M}_f$ is the universal cover of $M_f$, for some $N$.  Let $z_i$ be the $i$th coordinate function on $D$.  Then
$$\xymatrix{
\mathbb{C}^N = \widetilde{M}_f \ar[d]\ar@{-->}[ddr] &\\
M_f \ar@{^(->}[r]& D\ar[d]^{z_i}\\
&\mathbb{C}
} $$
The indicated composite mapping is a bounded holomorphic function on $\mathbb{C}^N$, which is therefore constant.  This is true for all coordinate functions, and so $\mathbb{C}^N=\{pt\}$.

Likewise, if some $M_k$ were compact, then $z_i|_{M_k}$ would be a holomorphic function on the compact manifold $M_k$, which is again reduced to a constant.  So, if $M_k$ were compact, then $M_k=\{pt\}$ as well.
\end{proof}

The converse question is more interesting:

{\em Question.}  Let $M$ be a non-compact Hermitian symmetric space.  Is $M$ locally a bounded symmetric domain?

The answer, which turns out to be {\em Yes}, is addressed in the next sections.  The general result is that it is always possible to embed a Hermitian symmetric space into a generalized flag manifold.  In the compact case, this embedding is essentially an isomorphism.  In the noncompact case, the image is an open subset of a flag manifold, and there is a natural projection to a bounded symmetric domain in an affine complex space.

\section{Structure of Hermitian symmetric Lie algebras}
Let $(\mathfrak{g}_o,s_o)$ be an orthogonal symmetric simple Lie algebra which is not of compact type.  Write
$$\mathfrak{g}_o=\mathfrak{k}_o\oplus\mathfrak{p}_o$$
and let $i\mathfrak{t}_o\subset\mathfrak{k}_o$ be a maximal toral subalgebra.  Let $\mathfrak{h}_o = C_{\mathfrak{g}_o}(i\mathfrak{t}_o) = i\mathfrak{t}_o\oplus\mathfrak{b}_o$.  This is a Cartan subalgebra of $\mathfrak{g}_o$.  Now decompose the complexification into root spaces
$$\mathfrak{k}\oplus\mathfrak{p} = \mathfrak{g} = \mathfrak{h}\oplus\left(\bigoplus_{\alpha\in\Phi(\mathfrak{g},\mathfrak{h})} \mathfrak{g}_\alpha\right).$$
Since $\mathfrak{g}_o$ is of noncompact type, $\Phi(\mathfrak{g},\mathfrak{h})$ is real on $i\mathfrak{t}_o\oplus\mathfrak{b}_o$.  Relative to this decomposition of $\mathfrak{g}$, we can decompose $\mathfrak{g}_o$ via
$$\mathfrak{g}_o = \mathfrak{h}_o\oplus\left(\bigoplus_{\lambda\in\Phi^+(\mathfrak{g}_o,\mathfrak{t}_o)}\left(\underset{\alpha|_{i\mathfrak{t}_o} = \lambda}{\bigoplus_{\alpha\in\Phi(\mathfrak{g},\mathfrak{h})}}\mathfrak{g}_{\pm\alpha}\right)\cap\mathfrak{g}_o\right).$$

Indeed, let $\theta:\mathfrak{g}\to\mathfrak{g}$ be the antilinear involution of complex conjugation with respect to the real subspace $\mathfrak{g}_o$.  For all $H\in\mathfrak{h}$ and $x\in\mathfrak{g}_\alpha$, $[H,\theta x] = \overline{\alpha\circ\theta(H)}\, \theta x$.  Thus $\theta\mathfrak{g}_{\alpha}=\mathfrak{g}_{\overline{\alpha\circ\theta}}$.  Now if $H\in i\mathfrak{t}_o$, then $\theta(H)=H$ and so $\overline{\alpha\circ\theta(H)} = -\alpha(H)$, since $\alpha$ is imaginary on $i\mathfrak{t}_o$.  Hence $(\alpha\circ\theta)|_{i\mathfrak{t}_o} = -\alpha|_{i\mathfrak{t}_o}$, as required.

\begin{definition}\index{Root!compact}\index{Root!noncompact}\index{Root!complex}
\mbox{}
\begin{enumerate}
\item If $\mathfrak{g}_\alpha\subset\mathfrak{k}$, then $\alpha$ is called a {\em compact root}.
\item If $\mathfrak{g}_\alpha\subset\mathfrak{p}$, then $\alpha$ is called a {\em noncompact root}.
\item Otherwise $\alpha$ is called a {\em complex root}.
\end{enumerate}
\end{definition}
If $\alpha$ is either compact or noncompact, then $s(\alpha)=\alpha$ for in the former case $s|_{\mathfrak{g}_\alpha} = \operatorname{id}$ and in the latter $s|_{\mathfrak{g}_\alpha} = -\operatorname{id}$.  If $\alpha$ is complex, then $s(\alpha) \not=\alpha$.

\begin{lemma}
If $\operatorname{rank}(\mathfrak{k})=\operatorname{rank}(\mathfrak{g})$, then $\mathfrak{h}_o=\mathfrak{t}_o$, and so there are no complex roots.
\end{lemma}

\subsection{Example}  Consider the root lattice $C_2$.  The compact simply-connected group is $Sp(2)$.

The root space decomposition is
$$\mathfrak{g} = \mathfrak{t}\oplus\mathfrak{g}_{\pm\alpha}\oplus\mathfrak{g}_{\pm\beta}\oplus\mathfrak{g}_{\pm\theta_s}\oplus\mathfrak{g}_{\pm\theta}.$$
The involutions in $G_{\ad}$ are given by $e^{2\pi i\lambda_\beta^\vee/2} = e^{\pi i \lambda_\beta^\vee}$ and $e^{2\pi i\lambda_\alpha^\vee/2=e^2\pi i \lambda_\alpha^\vee}$.  Here are the Cartan decompositions for these involutions:
\begin{description}
\item[$\sigma = \Ad e^{\pi i \lambda_\beta^\vee}$]  Here $\mathfrak{k} = \mathfrak{g}_\sigma = \mathfrak{t}\oplus\mathfrak{g}_{\pm\alpha}$ and $\mathfrak{p}=\mathfrak{g}_{\pm\beta}\oplus\mathfrak{g}_{\pm\theta_s}\oplus\mathfrak{g}_\theta$.  Note $\dim Z(\mathfrak{k})=1$, so this symmetric space supports a complex structure.  (It is locally isometric to the Grassmannian $\operatorname{Gr}_2(\mathbb{C}^4)$.)
\item[$\sigma=\Ad e^{\pi i \lambda_\alpha^\vee}$]  Here $\mathfrak{k}= \mathfrak{g}_\sigma = \mathfrak{t}\oplus\mathfrak{g}_{\pm\beta}\oplus\mathfrak{g}_{\pm\theta}$ and $\mathfrak{p}=\mathfrak{g}_{\pm\alpha}\oplus\mathfrak{g}_{\pm\theta_s}$. Then $Z(\mathfrak{k})=0$, and there is no complex structure.  We already observed that this quotient was locally isometric to $S^4$, but in fact $S^4$ carries no almost complex structure (let alone a homogeneous complex structure).\footnote{Suppose $S^4$ had an almost complex structure, and split the complexified tangent bundle into eigenspaces $TS^4\otimes\mathbb{C} = E\oplus\overline{E}$.  Then the Pontrjagin numbers are related to the Chern classes by $p_2(TS^4) = c_1(E)^2 - 2c_2(E)$.  But $c_1(E)=0$ since $H^2(S^4,\mathbb{Z})=0$, $c_2(E)=2$ is the Euler characteristic, and $p_2(TS^4)=0$ since $S^4$ is cobordant to a point.  Contradiction.}
\end{description}

\subsection{General case}
Let us return to the general case.  Let $(\mathfrak{g}_o,s_o)$ be a semisimple orthogonal symmetric Lie algebra.  Let $\Delta = \{\alpha_1,\dots,\alpha_n\}$ be a base for $\Phi(\mathfrak{g}_o,\mathfrak{h}_o)$.  Let $\theta=d_1\alpha_1+\cdots+d_n\alpha_n$ be the highest root.  The fundamentla coroots $\lambda_j^\vee$ satisfy $\lambda_j^\vee(\theta)=d_j$.  In general, involutions come from $e^{\pi i \lambda_j^\vee}$ with $d_j=1,2$.  We deal with these two cases as above:

If $d_j=2$: For any root $\alpha$, the possible pairings with $\lambda_j^\vee$ are $\alpha(\lambda_j^\vee)=0,\pm 1,\pm 2$, since $\alpha_j$ appears in the {\em highest} root with coefficient $2$.  Now the involution $e^{\pi i \lambda_j^\vee}$ acts as the identity on the eigenspaces where $\lambda_j^\vee$ is even and acts as the negative identity on eigenspaces where $\lambda_j^\vee$ is odd.  Thus
\begin{align*}
\mathfrak{k}&=\mathfrak{t}\oplus\left(\bigoplus_{\alpha(\lambda_j^\vee)=0}\mathfrak{g}_{\pm\alpha}\right)\oplus\left(\bigoplus_{\beta(\lambda_j^\vee=2)}\mathfrak{g}_{\pm\beta}\right)\\
\mathfrak{p}&=\bigoplus_{\gamma(\lambda_j^\vee)=0}\mathfrak{g}_{\pm\gamma}.
\end{align*}
In this case, note that $Z(\mathfrak{k})=0$ and so there is no complex structure.

\vspace{24pt}

If $d_j=1$: $\mathfrak{k}=\mathfrak{t}\oplus \left(\bigoplus_{\alpha(\lambda_j^\vee)=0} \mathfrak{g}_{\pm\alpha}\right)$.  In this case $Z(\mathfrak{k})$ is 1-dimensional: it is the subspace of $\mathfrak{t}$ spanned by the single coroot $\lambda_j^\vee$.  We also have $\mathfrak{p}=\oplus_{\beta(\lambda_j^\vee)=1}\mathfrak{g}_{\pm\beta}$.  (Note that for any root $\beta$, there are only the three possibilities $\beta(\lambda_j^\vee)=0,1,-1$ since $\alpha_j$ appears in the {\em highest} root with coefficient $1$.)  Let
$$\mathfrak{u}^+ = \sum_{\beta(\lambda_j^\vee)=1}\mathfrak{g}_\beta,\quad \mathfrak{u}^- = \sum_{\beta(\lambda_j^\vee)=1}\mathfrak{g}_{-\beta}.$$
Then $[\mathfrak{u}^+,\mathfrak{u}^+]=[\mathfrak{u}^-,\mathfrak{u}^-]=0$.  

\subsection{Strongly orthogonal root systems}
Let $(\mathfrak{g}_o,s_o)$ be an irreducible symmetric orthogonal Lie algebra with $\mathfrak{g}$ simple (i.e., type I or III).  Let $\theta:\mathfrak{g}\to\mathfrak{g}$ be the antilinear conjugation with respect to the real subspace $\mathfrak{g}_o$.  As usual, put $\mathfrak{g}_o=\mathfrak{k}_o\oplus\mathfrak{p}_o$ and let $i\mathfrak{t}_o\subset\mathfrak{k}_o$ be a maximal toral subalgebra.  We have shown that
$$C_{\mathfrak{g}_o}(i\mathfrak{t}_o) = i\mathfrak{t}_o\oplus\mathfrak{b}_o$$
is a Cartan subalgebra.  Let $\Phi(\mathfrak{g}_o,\mathfrak{t}_o)$ be the relative root system for $i\mathfrak{t}_o$ on $\mathfrak{g}_o$, and decompose $\mathfrak{g}$ and $\mathfrak{g}_o$ via
\begin{align*}
\mathfrak{g} &= \mathfrak{h}\oplus\left(\bigoplus_{\alpha\in\Phi(\mathfrak{g},\mathfrak{h})}\mathfrak{g}_\alpha\right)\\
\mathfrak{g}_o &=\mathfrak{h}_o\oplus\bigoplus_{\lambda\in\Phi^+(\mathfrak{g}_o,\mathfrak{t}_o)}\Biggl(\biggl[\underset{\alpha|_{i\mathfrak{t}_o=\lambda}}{\bigoplus_{\alpha\in\Phi(\mathfrak{g},\mathfrak{h})}}\mathfrak{g}_{\pm\alpha}\biggr]\Biggr).
\end{align*}

\begin{lemma}
Let $\alpha\in \Phi(\mathfrak{g},\mathfrak{h})$.  Then $s_o\alpha = \alpha$ if and only if $\alpha|_{\mathfrak{b}_o}=0$.
\end{lemma}

\begin{proof}
$s_o$ acts as $-\operatorname{id}$ on $\mathfrak{b}$.
\end{proof}

In particular, if $\mathfrak{b}_o=0$, then every root is either compact ($\mathfrak{g}_\alpha\subset\mathfrak{k}$) or noncompact ($\mathfrak{g}_\alpha\subset\mathfrak{p}$).  Let $\Phi_c(\mathfrak{g},\mathfrak{h})\subset\Phi(\mathfrak{g},\mathfrak{h})$ be the set of compact roots and $\Phi_{nc}(\mathfrak{g},\mathfrak{h})\subset\Phi(\mathfrak{g},\mathfrak{h})$ be the set of noncompact roots.  Then we have

\begin{lemma}
If $\operatorname{rank}\mathfrak{k}=\operatorname{rank}\mathfrak{g}$, then $\Phi(\mathfrak{g},\mathfrak{h})=\Phi_c(\mathfrak{g},\mathfrak{h})\cup\Phi_{nc}(\mathfrak{g},\mathfrak{h})$.
\end{lemma}

\begin{proof}
The maximal toral subalgebra $\mathfrak{t}$ is a Cartan subalgebra of $\mathfrak{k}$, and $C_{\mathfrak{g}}(\mathfrak{t})=\mathfrak{t}\oplus\mathfrak{b}$ is a Cartan subalgebra of $\mathfrak{g}$.  These have the same dimension only if $\mathfrak{b}=0$, and the result follows.
\end{proof}

Let us therefore suppose that $\operatorname{rank}\mathfrak{k}=\operatorname{rank}\mathfrak{g}$ (equivalently $\mathfrak{b}=0$), so that every root is either compact or noncompact.  Then, by virtue of the identity $\overline{\alpha\circ\theta} = -\alpha$, we showed that $\theta :\mathfrak{g}_\alpha\to\mathfrak{g}_{-\alpha}$.  In particular,
$$\mathfrak{g}_{\alpha}\oplus\mathfrak{g}_{-\alpha} = \underbrace{[(\mathfrak{g}_{\alpha}\oplus\mathfrak{g}_{-\alpha})\cap\mathfrak{g}_o]}_{\subset\mathfrak{k}_\alpha} \oplus \underbrace{[(\mathfrak{g}_{\alpha}\oplus\mathfrak{g}_{-\alpha})\cap i\mathfrak{g}_o]}_{\subset\mathfrak{p}_\alpha}.$$
For $\alpha\in\Phi^+(\mathfrak{g},\mathfrak{h})$, let $E_\alpha$ be the generator of $\mathfrak{g}_\alpha$ and $F_\alpha$ the generator of $\mathfrak{g}_{-\alpha}$.  Then by the above
\begin{align*}
\mathfrak{k}_o &= \mathfrak{t}_o \oplus\left(\bigoplus_{\alpha\in\Phi_c^+(\mathfrak{g},\mathfrak{t})} \mathbb{R}[E_\alpha + F_\alpha]\right)\oplus\left(\bigoplus_{\alpha\in\Phi_c^+(\mathfrak{g},\mathfrak{t})} \mathbb{R}[iE_\alpha - iF_\alpha]\right)\\
\mathfrak{p}_o &= \left(\bigoplus_{\alpha\in\Phi_{nc}^+(\mathfrak{g},\mathfrak{t})} \mathbb{R}[E_\alpha + F_\alpha]\right)\oplus\left(\bigoplus_{\alpha\in\Phi_{nc}^+(\mathfrak{g},\mathfrak{t})} \mathbb{R}[iE_\alpha - iF_\alpha]\right).
\end{align*}

If (in addition) $s_o$ is conjugation with respect to $e^{\pi i \lambda_j^\vee}$ with $d_j=1$ (so that $(\mathfrak{g}_o,s_o)$ is Hermitian), then
\begin{align*}
\mathfrak{k} &= \mathfrak{t}\oplus\underset{\alpha(\lambda_j^\vee)=0}{\bigoplus_{\alpha\in\Phi(\mathfrak{g},\mathfrak{t})}}\mathfrak{g}_\alpha\\
&\\
\mathfrak{p} &= \underset{\beta(\lambda_j^\vee)=\pm 1}{\bigoplus_{\beta\in\Phi(\mathfrak{g},\mathfrak{t})}}\mathfrak{g}_\beta = \mathfrak{u}^-\oplus\mathfrak{u}^+\\
&\\
\mathfrak{u}^+ &=\underset{\beta(\lambda_j^\vee)=+ 1}{\bigoplus_{\beta\in\Phi(\mathfrak{g},\mathfrak{t})}}\mathfrak{g}_\beta\\
\mathfrak{u}^- &=\underset{\beta(\lambda_j^\vee)=- 1}{\bigoplus_{\beta\in\Phi(\mathfrak{g},\mathfrak{t})}}\mathfrak{g}_\beta.
\end{align*}
Equating the two decompositions, we see that
\begin{align*}
\Phi_c(\mathfrak{g},\mathfrak{t}) &= \Phi(\mathfrak{k},\mathfrak{t}) = \{\alpha\in\Phi(\mathfrak{g},\mathfrak{t})\mid \alpha(\lambda_j^\vee)=0\}\\
\Phi_{nc}(\mathfrak{g},\mathfrak{t}) &= \{\alpha\in\Phi(\mathfrak{g},\mathfrak{t})\mid \alpha(\lambda_j^\vee)=\pm 1\}.
\end{align*}

\begin{definition}\index{Strongly orthogonal}
A subset $F\subset\Phi$ of a root system $\Phi$ is {\em strongly orthogonal} if for any two linearly independent $\alpha,\beta\in F$, $\alpha\pm\beta\not\in \Phi$.
\end{definition}

{\em Remark.} Notice that any two linearly independent roots in $F$ are orthogonal, for if $\alpha,\beta\in F$ were not orthogonal, then the Weyl reflection $\sigma_\alpha\beta$ of $\beta$ in $\alpha$ in $\Phi$.  But $\sigma_\alpha\beta = \beta + k\alpha$ for some $k\in\mathbb{Z}$, and then the whole $\alpha$-string $\beta, \beta \pm\alpha,\dots,\beta +k\alpha$ would need to be in $\Phi$ as well.

\begin{definition}
The strongly orthogonal rank of a root system $\Phi$, denoted $\operatorname{SO-rank}(\Phi)$, is defined as
$$\operatorname{SO-rank}(\Phi) = \max\{ |F\cap\Phi^+|\mid \text{$F\subset\Phi$ is a strongly orthogonal root system}\}.$$
This is the maximal cardinality among the strongly orthogonal sets of positive roots. A strongly orthogonal set of positive $\Phi$ is called {\em maximal} if
$$|F| = \operatorname{SO-rank}(\Phi).$$
\end{definition}

\begin{theorem}
Let $(\mathfrak{g}_o,s_o)$ be an irreducible Hermitian symmetric Lie algebra.  Then $\operatorname{rank}_{\mathbb{R}}(\mathfrak{g}_o)=\operatorname{SO-rank}(\Phi_{nc}(\mathfrak{g},\mathfrak{t}))$.
\end{theorem}

\begin{proof}
Let $F\subset \Phi^+_{nc}(\mathfrak{g},\mathfrak{t})$ be strongly orthogonal.  Since
$$\bigoplus_{\alpha\in F} \mathbb{R}[E_\alpha+F_\alpha] \subset\mathfrak{p}_o$$
is an abelian subalgebra, $|F|\le\operatorname{rank}\mathfrak{g}_o$.  Suppose that the inequality is strict.  Then there exists a flat subspace $\mathfrak{a}_o$ such that
$$\bigoplus_{\alpha\in F} \mathbb{R}[E_\alpha+F_\alpha] \subsetneq \mathfrak{a}_o.$$
Let $X\in\mathfrak{a}_o$, $X\not\in \bigoplus_{\alpha\in F} \mathbb{R}[E_\alpha+F_\alpha]$.  Write
$$X = \sum_{\alpha\in\Phi_{nc}^+}a_\alpha(E_\alpha+F_\alpha) + b_\alpha(iE_\alpha-iF_\alpha) + \sum_{\beta\in F}c_\beta(iE_\beta-iF_\beta).$$
Since $\mathfrak{a}_o$ is abelian, $X$ commutes with $E_\gamma+F_\gamma$ for all $\gamma\in F$.  In particular, the $\mathfrak{t}_o$-component of $[X,E_\gamma+F_\gamma]$ is $c_\gamma[iE_\gamma-iF_\gamma, E_\gamma+F_\gamma]=c_\gamma 2i[E_\gamma,F_\gamma]$.  For this to be zero, $c_\gamma=0$.  Thus the last term in $X$ vanishes completely:
$$X = \sum_{\alpha\in\Phi_{nc}^+}a_\alpha(E_\alpha+F_\alpha) + b_\alpha(iE_\alpha-iF_\alpha).$$
So, commuting with $E_\gamma+F_\gamma$ for $\gamma\in F$ now gives
$$[X,E_\gamma+F_\gamma] = \sum_{\alpha\in\Phi^+_{nc}\setminus F} a_\alpha[E_\alpha+F_\alpha,E_\gamma+F_\gamma] + b_\alpha[iE_\alpha-iF_\alpha,E_\gamma+F_\gamma]=0.$$
Since $a_\alpha$ and $b_\alpha$ are real, it follows that
$$\sum_{\alpha\in\Phi^+_{nc}\setminus F} a_\alpha[E_\alpha+F_\alpha,E_\gamma+F_\gamma] =  \sum_{\alpha\in\Phi^+_{nc}\setminus F}b_\alpha[iE_\alpha-iF_\alpha,E_\gamma+F_\gamma]=0.$$
The first sum splits into pairs of the form
$$a_\alpha\underbrace{[E_\alpha,E_\gamma]}_{\mathfrak{g}_{\alpha+\gamma}} + a_\beta\underbrace{[E_\beta,F_\gamma]}_{\in\mathfrak{g}_{\beta-\gamma}} = 0$$
(and similarly for the second sum) which implies that one of three possibilities holds: either $\alpha+\gamma=\beta-\gamma$, $a_\alpha=a_\beta=0$, or $[E_\alpha,E_\gamma]=[E_\beta,F_\gamma]=0$.  But if the first possibility were true, then applying $\lambda_j^\vee$ to both sides would give
$$2=\lambda_j^\vee(\alpha+\gamma) = \lambda_j^\vee(\beta-\gamma) = 0,$$
a contradiction, so this cannot happen for {\em any} $\gamma$.  If the second possibility were not true, then $[E_\alpha,E_\gamma]=[E_\alpha,F_\gamma]=0$ for all $\gamma\in F$, which implies that $F\cup \{\alpha\}$ is strongly orthogonal.
\end{proof}

\section{Embedding theorems}
Let $(\mathfrak{g}_o,s_o)$ be an irreducible Hermitian symmetric Lie algebra whose complexification $\mathfrak{g}$ is semisimple.  Let $\mathfrak{g}_o=\mathfrak{k}_o\oplus\mathfrak{p}_o$ be the eigenspace decomposition for $s_o$, and $\mathfrak{g}=\mathfrak{k}\oplus\mathfrak{p}$ its complexification.  Let $i\mathfrak{t}_o\subset\mathfrak{k}_o$ be a maximal toral subalgebra,  Decompose these various spaces via, in the notation of the previous section,
\begin{align*}
\mathfrak{k} &= \mathfrak{t}\oplus \bigoplus_{\alpha\in\Phi(\mathfrak{g},\mathfrak{t})}\mathfrak{g}_\alpha\\
\mathfrak{k}_o &= i\mathfrak{t}_o\oplus \left(\bigoplus_{\alpha\in\Phi_c^+(\mathfrak{g},\mathfrak{t})} \mathbb{R}[E_\alpha+F_\alpha]\right)\oplus \left(\bigoplus_{\alpha\in\Phi_c^+(\mathfrak{g},\mathfrak{t})} \mathbb{R}[iE_\alpha-iF_\alpha]\right)\\
\mathfrak{p}_o &= \left(\bigoplus_{\alpha\in\Phi^+_{nc}(\mathfrak{g},\mathfrak{t})} \mathbb{R}[E_\alpha+F_\alpha]\right)\oplus \left(\bigoplus_{\alpha\in\Phi_{nc}^+(\mathfrak{g},\mathfrak{t})} \mathbb{R}[iE_\alpha-iF_\alpha]\right)\\
\mathfrak{a}_o &= \bigoplus_{\alpha\in F} \mathbb{R}[E_\alpha+F_\alpha]
\end{align*}
where $F\subset\Phi_{nc}^+$ is a maximal strongly orthogonal system of roots.  Finally, put
$$\mathfrak{u}^{\pm} = \bigoplus_{\alpha\in\Phi_{nc}^+(\mathfrak{g},\mathfrak{t})}\mathfrak{g}_{\pm\alpha},\qquad [\mathfrak{u}^+,\mathfrak{u}^+]=[\mathfrak{u}^-,\mathfrak{u}^-]=0.$$

We shall employ the notation in Table 1 to refer to the various Lie algebras involved, and associated connected Lie groups.

\begin{table}
\begin{tabular}{|c|c|}
\hline
Connected Lie group & Lie algebra\\
$G_{\mathbb{C}}$ & $\mathfrak{g}$\\
$K_{\mathbb{C}}$ & $\mathfrak{k}$\\
$K$ & $\mathfrak{k}_o$\\
$G$ & $\mathfrak{g}_o=\mathfrak{k}_o\oplus \mathfrak{p}_o$\\
$\widetilde{G}$ & $\widetilde{\mathfrak{g}}_o =\mathfrak{k}_o\oplus i\mathfrak{p}_o$\\
$U^{\pm}$ & $\mathfrak{u}^{\pm}$\\
$A$ & $\mathfrak{a}_o$\\
$P^{\pm}$ & $\mathfrak{k}\oplus\mathfrak{u}^\pm$\\
\hline
\end{tabular}
\caption{}
\end{table}

\subsection{Example: $\mathbb{sl}(2,\mathbb{C})$}  In the following discussion, it is helpful to keep in mind the following simple example. Let $\mathfrak{g}_o=\mathfrak{su}(1,1)$ and $\mathfrak{g}=\mathfrak{sl}(2,\mathbb{C})$.  Then
\begin{align*}
\mathfrak{sl}(2,\mathbb{C}) &= \underbrace{\operatorname{diag}(x,-x)}_{\mathfrak{k}}\oplus\underbrace{\begin{pmatrix} 0&w\\0&0\end{pmatrix}}_{\mathfrak{u}^+}\oplus\underbrace{\begin{pmatrix} 0&0\\w&0\end{pmatrix}}_{\mathfrak{u}^-}\\
\mathfrak{su}(1,1) &= \underbrace{\operatorname{diag}(ia,-ia)}_{\mathfrak{k}_o}\oplus\underbrace{\begin{pmatrix} 0&c+ib\\c-ib&0\end{pmatrix}}_{\mathfrak{p}_o},\qquad \mathfrak{a}_o=\begin{pmatrix}0&c\\c&0\end{pmatrix}
\end{align*}
The dual Lie algebra is $\widetilde{\mathfrak{su}(1,1)} = \mathfrak{su}(2)$ and then we have
$$\mathfrak{su}(2) = \underbrace{\operatorname{diag}(ia,-ia)}_{\mathfrak{k}_o}\oplus\underbrace{\begin{pmatrix} 0&-b+ic\\b-ic&0\end{pmatrix}}_{\mathfrak{p}_o}.$$

The groups are
\begin{align*}
G_{\mathbb{C}} &= SL(2,\mathbb{C})\\
K &= \left\{ \begin{pmatrix}\cos a&-\sin a\\ \sin a&\cos a\end{pmatrix}\mid a\in\mathbb{R}\right\}\\
G = SU(1,1) &= \left\{ \begin{pmatrix}\alpha&\beta\\ \bar{\beta}&\alpha\end{pmatrix}\mid \alpha,\beta\in\mathbb{C},\ |\alpha|^2-|\beta|^2 = 1\right\}\\
\widetilde{G} = SU(2) &= \left\{ \begin{pmatrix}\alpha&\beta\\ -\bar{\beta}&\alpha\end{pmatrix}\mid \alpha,\beta\in\mathbb{C},\ |\alpha|^2+|\beta|^2 = 1\right\}\\
U^+ &= \left\{ \begin{pmatrix}1&w\\ 0&1\end{pmatrix}\mid w\in\mathbb{C}\right\}\cong (\mathbb{C},+)\\
U^- &= \left\{ \begin{pmatrix}1&0\\ w&1\end{pmatrix}\mid w\in\mathbb{C}\right\}\cong (\mathbb{C},+)\\
P^+ &= \left\{ \begin{pmatrix}z&w\\ 0&z^{-1}\end{pmatrix}\mid w,z\in\mathbb{C},\ z\not=0\right\}\\
P^- &= \left\{ \begin{pmatrix}z&0\\ w&z^{-1}\end{pmatrix}\mid w,z\in\mathbb{C},\ z\not=0\right\}\\
A &= \left\{ \begin{pmatrix}\cosh c&\sinh c\\ \sinh c&\cosh c\end{pmatrix}\mid c\in\mathbb{R}\right\}.
\end{align*}

Now, let $SL(2,\mathbb{C})$ act on $\mathbb{C}^2$ in the usual manner.  Then $P^-$ is the stabilizer of the complex line spanned by $[1\ 0]^t$, and $SL(2,\mathbb{C})$ acts transitively on the space of complex lines through the origin, whence
$$\frac{SL(2,\mathbb{C})}{P^-} \cong \mathbb{CP}^1 = S^2. $$
Now, $SU(2)$ acts transitively and freely on $S^3\subset\mathbb{C}^4$, and we have
$$S^2 = \frac{SL(2,\mathbb{C})}{P^-} = \frac{SU(2)}{P^-\cap SU(2)} = \frac{SU(2)}{U(1)}$$
(which is precisely the Hopf fibration $S^2=S^3/S^1$).  

There is also an embedding
$$\frac{SU(1,1)}{U(1)} \to \frac{SL(2,\mathbb{C})}{P^-} = \mathbb{CP}^1.$$
However $SU(1,1)$ does not act transitively on $\mathbb{CP}^1$.  Indeed, the orbit of the point $[0\ 1]^t$ is
$$\begin{pmatrix}\alpha & \beta\\ \overline{\beta} & \overline{\alpha}\end{pmatrix} = \begin{pmatrix}\beta\\ \overline{\alpha}\end{pmatrix} \sim\begin{pmatrix}\beta/\overline{\alpha}\\ 1\end{pmatrix}.$$
From $|\alpha|^2-|\beta|^2=1$ it follows that $1=\frac{1}{|\beta|^2}+\left|\frac{\beta}{\alpha}\right|^2$, so $|\beta/\alpha|<1$.  So $SU(1,1)$ embeds into a disc in $\mathbb{CP}^1$.

Note the factorization
\begin{equation}\label{su11decomposition}
\begin{pmatrix}\cosh t& \sinh t\\\sinh t& \cosh t\end{pmatrix} = 
\begin{pmatrix}1& 0\\\tanh t& 1\end{pmatrix} 
\begin{pmatrix}\cosh t& 0\\0& \cosh t\end{pmatrix} 
\begin{pmatrix}1& \tanh t\\0& 1\end{pmatrix}.
\end{equation}

\subsection{General case}
We return now to the general case.  We have the following lemmata:

\begin{description}
\item[$N_{G_{\mathbb{C}}}(U^{\pm})=K_{\mathbb{C}}\ltimes U^{\pm} = P^{\pm}$]
\mbox{}
\begin{proof}\commentx{Check.}
Indeed, $\Lie N_{G_{\mathbb{C}}}(U^+) = N_{\mathfrak{g}}(\mathfrak{u}^+) = \mathfrak{k}\oplus\mathfrak{u}^+$, so $(N_{G_{\mathbb{C}}}(U^+))_e = P^+$.  Let $g\in N_{G_{\mathbb{C}}}(U^+)$.  Then $g$ acts on $\mathfrak{k}\oplus\mathfrak{u}^+$.  Let $\mathfrak{t}\subset\mathfrak{k}$ be a Cartan subalgebra.  There exists $x\in P^+$ such that $xg(\mathfrak{t})=\mathfrak{t}$.  There exists $y\in K_{\mathbb{C}}$ such that $yxg(\mathfrak{t})=\mathfrak{t}$ and, in addition, $yxg \left( \Phi^+(\mathfrak{k},\mathfrak{t})\right)=\Phi^+(\mathfrak{k},\mathfrak{t})$.  Since $yxg$ preserves the compact roots, $yxg\in K_{\mathbb{C}}$.  Thus $g\in P^+$.  So $N_{\mathbb{G}_{\mathbb{C}}}(U^\pm) = P^\pm$.

It suffices to show tht $K_{\mathbb{C}}\cap U^+=1$.  First, observe that $K_{\mathbb{C}}\cap U^+\trianglerighteq K_{\mathbb{C}}$ so $K_{\mathbb{C}}\subset C_{G_{\mathbb{C}}}(K_{\mathbb{C}}\cap U^+)$.  Also $\Lie(K_{\mathbb{C}}\cap U^+)=0$.  Let $s\in K_{\mathbb{C}}\cap U^+$.  Write $s=\exp X$, $X\in\mathfrak{u}^+$.  Since $X$ is centralized by $K_{\mathbb{C}}$, $X=0$.  So $s=e$.
\end{proof}

\item[$P^+\cap P^-=K_{\mathbb{C}}$]
\mbox{}
\begin{proof}
The inclusion $\supseteq$ is clear.  For the opposite inclusion, let $g\in P^+\cap P^-$.  There exists $k\in K_{\mathbb{C}}$ such that $kg\in P^+\cap U^-$.  In particular, $kg=\exp X$ for some $X\in\mathfrak{u}^-$.  On the other hand, $\exp X\in N_{G_{\mathbb{C}}}(U^+)$, so $X\in N_{\mathfrak{g}}(\mathfrak{u}^+)$.  But $\mathfrak{u}^-\cap N_{\mathfrak{g}}(\mathfrak{u}^+)=0$\commentx{(Why?)}, so $X=0$.  Hence $g\in K$.
\end{proof}

\item[$G\cap P^{\pm} = K = \widetilde{G}\cap P^\pm$] 
\mbox{}
\begin{proof}
We will show that $G\cap P^\pm=K$, the other equality being similar.  The inclusion $\supseteq$ is clear.  For the opposite inclusion, first note that $\Lie(G\cap P^+)=\mathfrak{g}_o\cap(\mathfrak{k}\oplus\mathfrak{u}^+) = \mathfrak{k}_o$.  So $K\trianglelefteq G\cap P^+$ and $(G\cap P^+)_e=K$.  Let $g\in G\cap P^+$.  Let $\mathfrak{t}_o\subset\mathfrak{k}_o$ be a maximal toral subalgebra, and $T\subset K$ a maximal torus, with Lie algebra $\Lie(T)=\mathfrak{t}_o$.  There exists $k\in K$ such that $\Ad(kg)\mathfrak{t}_o=\mathfrak{t}_o$.  Thus $kg\in C_G(\mathfrak{t}_o) = C_G(T) = T$.  Since $T\subset K$, $g\in K$. 
\end{proof}
\end{description}

Combining these remarks with the Cartan immersion, we conclude
\begin{proposition}
The mapping $U^+\times K_{\mathbb{C}}\times U^-\xrightarrow{\phi} G_{\mathbb{C}}$ defined by $\phi(x,k,y)=xky$ is injective, everywhere regular, and complex analytic.
\end{proposition}

\begin{theorem}\index{Embedding theorems}
\mbox{}
\begin{description}
\item[First embedding theorem] There exist Hermitian embeddings $\widetilde{G}/K\hookrightarrow G_{\mathbb{C}}/P^-$ and $G/K\hookrightarrow G_{\mathbb{C}}/P^-$.  The embedding $G/K\hookrightarrow G_{\mathbb{C}}/P^-$ is a Hermitian isometry, and the embedding $\widetilde{G}/K\hookrightarrow G_{\mathbb{C}}/P^-$ has open image.
\item[Second embedding theorem] $G\subset U^+K_{\mathbb{C}}U^-$.  Hence
$$\xymatrix{
G/K \ar@{^(->}[r]\ar[dr]^\tau & (U^+K_{\mathbb{C}}U^-)/P^-\ar@{^(->}[r]\ar[d]^{\pi} & G_{\mathbb{C}}/P^-\\
& U^+\cong \mathbb{C}^N &
} $$
is an embedding onto the bounded symmetric domain $\tau(G/K)\subset\mathbb{C}^N$.
\end{description}
\end{theorem}
\begin{proof}
It remains to show that $\tau(G/K)$ is bounded.  For this, let $\mathfrak{a}_o= \bigoplus_{\alpha\in F} \mathbb{R}[E_\alpha+F_\alpha]$ be a flat subspace and let $A=\exp\mathfrak{a}_o\subset G\subset G_{\mathbb{C}}$.  We claim that $A\subset U^+K_{\mathbb{C}}U^-$.  Indeed, for any $\alpha$, $E_\alpha,F_\alpha,[E_\alpha,F_\alpha]$ spans a Lie subalgebra of $\mathfrak{g}$ isomorphic to $\mathfrak{sl}(2,\mathbb{C})$.  Under the isomorphism $E_\alpha+F_\alpha\mapsto \begin{pmatrix}0&1\\1&0\end{pmatrix}$.  The one parameter group $\exp t(E_\alpha+E_\beta)$ is a subgroup of $U^+K_{\mathbb{C}}U^-$, which decomposes by \eqref{su11decomposition} as
$$\exp t(E_\alpha+F_\alpha) = \exp (tE_\alpha)k_1\exp(tF_\alpha).$$
The product of two elements of this form is
\begin{align*}
\exp t(E_\alpha+F_\alpha) \exp t(E_\beta + F_\beta) &=\exp (tE_\alpha)k_1\exp(tF_\alpha) \exp (tE_\beta)k_2\exp(tF_\beta) \\
&=\exp (tE_\alpha)k_1\exp(tE_\beta) \exp (tF_\alpha)k_2\exp(tF_\beta) \\
&\in U^+k_1k_2U^-.
\end{align*}
Hence
$$G=KAK\subset KU^+K_{\mathbb{C}}U^- = U^+K_{\mathbb{C}}U^-.$$
So if $g\in G$, $g=k_1ak_2$ for some $k_i\in K$, $a\in A$, and $a=u^+k_3u^-$ for some $u^\pm\in U^\pm$.  Hence
$$\tau(g) = \tau(k_1ak_2) = \tau(k_1u^+k_3u^-k_2) = \tau(k_1u^+k_1^{-1}) = k_1\tau(a)k_1^{-1} \subset K\tau(A)K.$$
But $K$ is compact and $\tau(A)$ is bounded (by the $\mathfrak{su}(1,1)$ case).
\end{proof}

Combining the embedding theorem with the decomposition theorem for symmetric spaces, we have
\begin{theorem}
Let $D$ be a bounded symmetric domain.  Then $D\cong D_1\times\cdots\times D_t$ where each $D_i$ is a bounded symmetric domain that is irreducible as a symmetric space.  The decomposition is unique apart from the ordering.
\end{theorem}

\chapter{Classification of real simple Lie algebras}\label{Classification}
To classify symmetric spaces, it is necessary (and sufficient) to classify all real simple Lie algebras.  The essential idea is to start with a complex simple Lie algebra $\mathfrak{g}$ and to fix a Cartan subalgebra $\mathfrak{h}\subset\mathfrak{g}$ and a Cartan involution $\theta:\mathfrak{g}\to\mathfrak{g}$.  Then the strategy is to mark the Dynkin diagram in a way that shows the involution, and gives enough information to determine the status of each of the roots that are fixed by $\theta$: whether they are compact, noncompact, real, or imaginary.  There are two possible approaches to this classification.  In one, the Cartan subalgebra $\mathfrak{h}_o$ is the maximal compact CSA, and in the other it is the maximally split CSA.  These lead us, respectively, to introduce the Vogan diagram and Satake diagram for a real form.  Each of these has advantages and disadvantages: the approach via Vogan diagrams gives the classification much more easily and one can read off the subalgebra $\mathfrak{k}_o$ more readily, but with Satake diagrams the rank and restricted root system can be determined more easily.

\section{Classical structures}
We here fix notation for the classical groups and Lie algebras.  Various groups can usefully be regarded as matrix groups over $\mathbb{R}$, $\mathbb{C}$, or $\mathbb{H}$.  We can pass between these via the following identifications.  If $Z=A+iB\in GL(n,\mathbb{C})$ is a complex matrix, with $A,B$ real, then
$$\phi(Z) = \begin{pmatrix}A & -B\\ B& A\end{pmatrix} \in GL(2n,\mathbb{R}).$$
is a Lie group morphism into a closed subgroup of $GL(2n,\mathbb{R})$.  Similarly, if $H=A+Bj\in GL(n,\mathbb{H})$ is a quaternion matrix, with $A,B\in GL(n,\mathbb{C})$ complex (identifying $\mathbb{C}=\mathbb{R}[i]\subset\mathbb{H}$), then
$$\psi(H) = \begin{pmatrix} A & -\overline{B}\\ B&\overline{A}\end{pmatrix} \in GL(2n,\mathbb{C})$$
is a Lie group morphism onto a closed subgroup of $GL(2n,\mathbb{C})$.

\begin{description}
\item[$\mathfrak{sl}(n,\mathbb{R})$] Lie algebra of trace-free real $n\times n$ matrices.
\item[$\mathfrak{sl}(n,\mathbb{C})$] Lie algebra of trace-free complex $n\times n$ matrices.
\item[$\mathfrak{sl}(n,\mathbb{H})$] Lie algebra of $x\in\mathfrak{gl}(n,\mathbb{H})$ such that $\operatorname{Re}\tr x = 0$.  Equivalently, $\mathfrak{sl}(n,\mathbb{H})=\mathfrak{sl}(2n,\mathbb{C})\cap\psi(\mathfrak{gl}(n,\mathbb{H}))$.  This Lie algebra is also denoted by $\mathfrak{su}^*(2n)$.
\item[$\mathfrak{so}(n,\mathbb{R})$] Lie algebra of $x\in\mathfrak{gl}(n,\mathbb{R})$ such that $x+x^t=0$.  We denote it by just $\mathfrak{so}(n)$.
\item[$\mathfrak{so}(n,\mathbb{C})$] Lie algebra of $x\in\mathfrak{gl}(n,\mathbb{C})$ such that $x+x^t=0$.
\item[$\mathfrak{so}(n,\mathbb{H})$] Lie algebra of $x\in\mathfrak{gl}(n,\mathbb{H})$ such that $x+x^t=0$.  Equivalently, $\mathfrak{so}(2n,\mathbb{C})\cap\psi(\mathfrak{gl}(n,\mathbb{H}))$.  We denote it also by $\mathfrak{so}^*(2n)$.
\item[$\mathfrak{su}(n)$] Lie algebra of $x\in\mathfrak{gl}(n,\mathbb{C})$ such that $x+\overline{x}^t=0$.
\item[$\mathfrak{su}(n,\mathbb{H})$]  Lie algebra of $x\in\mathfrak{gl}(n,\mathbb{H})$ such that $x+\overline{x}^t=0$.  This is also denoted by $\mathfrak{sp}(n)$.  Note that $\mathfrak{sp}(n)\not=\mathfrak{sp}(n,\mathbb{R})$, defined below, but we do have $\mathfrak{sp}(n)\cong\mathfrak{sp}(n,\mathbb{C})\cap\psi(\mathfrak{gl}(n,\mathbb{H}))$.
\item[$\mathfrak{sp}(n,\mathbb{R})$]  Lie algebra of $x\in\mathfrak{gl}(2n,\mathbb{R})$ that leaves invariant the symplectic form $J$: $Jx + x^tJ=0$.
\item[$\mathfrak{sp}(n,\mathbb{C})$]  Lie algebra of $x\in\mathfrak{gl}(2n,\mathbb{C})$ that leaves invariant the symplectic form $J$: $Jx + x^tJ=0$.
\end{description}

\subsection{Mixed signature}  Let $I_{p,q}$ denote the pseudo-Hermitian form of signature $(p,q)$:
$$I_{p,q} = \begin{pmatrix}I_p&\\&-I_q\end{pmatrix}.$$
We can define in addition
\begin{description}
\item[$\mathfrak{so}(p,q)$] Lie algebra of $x\in\mathfrak{gl}(n,\mathbb{R})$ leaving $I_{p,q}$ invariant: $I_{p,q}x + x^tI_{p,q} = 0$.  Note that $\mathfrak{so}(n) = \mathfrak{so}(0,n)=\mathfrak{so}(n,0)$.
\item[$\mathfrak{su}(p,q)$] Lie algebra of $x\in\mathfrak{gl}(n,\mathbb{C})$ leaving $I_{p,q}$ invariant: $I_{p,q}x + \overline{x}^tI_{p,q} = 0$.  Note that $\mathfrak{su}(n) = \mathfrak{su}(0,n)=\mathfrak{su}(n,0)$.
\item[$\mathfrak{sp}(p,q)$]  Subalgebra of $x\in\mathfrak{sp}(n,\mathbb{C})$ that leave $I_{p,q}$ invariant: $I_{p,q}x + x^tI_{p,q}=0$.  Note $\mathfrak{sp}(n)=\mathfrak{sp}(0,n)=\mathfrak{sp}(n,0)=\mathfrak{sp}(n,\mathbb{C})\cap \mathfrak{u}(2n)$.
\end{description}

\subsection{Examples}
\begin{itemize}
\item The space of complex structures on $\mathbb{R}^{2n}$ can be identified with $GL(2n,\mathbb{R})/GL(n,\mathbb{C})\cong SO(2n)/U(n)$.
\item The space of real structures on $\mathbb{C}^n$ is $GL(n,\mathbb{C})/GL(n,\mathbb{R})\cong U(n)/O(n)$.
\item The space of quaternionic structures on $\mathbb{C}^{2n}$ is $GL(2n,\mathbb{C})/GL(n,\mathbb{H})\cong SU(2n)/SU^*(2n)$.
\end{itemize}

\begin{table}
\caption{Symmetric spaces of Type II and IV}
\begin{tabular}{cccc}
$\mathbb{C}$-simple& Compact form & Type II & Type IV\\
$\mathfrak{sl}(n,\mathbb{C})$ & $\mathfrak{su}(n)$ & $SU(n)$ & $SL(n,\mathbb{C})/SU(n)$\\
$\mathfrak{so}(n,\mathbb{C})$ & $\mathfrak{so}(n)$ & $SO(n)$ & $SO(n,\mathbb{C})/SO(n)$\\
$\mathfrak{sp}(n,\mathbb{C})$ & $\mathfrak{sp}(n)$ & $Sp(n)$ & $Sp(n,\mathbb{C})/Sp(n)$\\
$\mathfrak{e}_{6,7,8}^{\mathbb{C}}$ & $\mathfrak{e}_{6,7,8}$ & $E_{6,7,8}$ & $E_{6,7,8}^{\mathbb{C}}/E_{6,7,8}$\\
$\mathfrak{f}_4^{\mathbb{C}}$ & $\mathfrak{f}_4$ & $F_4$ & $F_4^{\mathbb{C}}/F_4$\\
$\mathfrak{g}_2^{\mathbb{C}}$ & $\mathfrak{g}_2$ & $G_2$ & $G_2^{\mathbb{C}}/G_2$
\end{tabular}
\end{table}

\section{Vogan diagrams}
Let $(\mathfrak{g}_o,s_o)$ be an irreducible orthgonal symmetric Lie algebra, with $\mathfrak{g}$ semisimple over $\mathbb{C}$.  Let $i\mathfrak{t}_o\subset\mathfrak{k}_o$ be a maximal toral subalgebra and $\mathfrak{h}_o = C_{\mathfrak{g}_o}(i\mathfrak{t}_o)$ the associated maximally compact Cartan subalgebra of $\mathfrak{g}_o$.  This decomposes as $\mathfrak{h}_o=i\mathfrak{t}_o\oplus\mathfrak{b}_o$ for some $\mathfrak{b}_o\subset\mathfrak{p}_o$. Recall that a root $\alpha\in\Phi(\mathfrak{g}_o,\mathfrak{h}_o)$ is:
\begin{itemize}
\item Real if $\alpha\equiv 0$ on $i\mathfrak{t}_o$.
\item Imaginary if $\alpha\equiv 0$ on $\mathfrak{b}_o$.  In this case, a root is compact if $\mathfrak{g}_\alpha\subset \mathfrak{k}$ and noncompact if $\mathfrak{g}_\alpha\subset\mathfrak{p}$.
\item Complex if neither holds.
\end{itemize}

\begin{definition}\footnote{More information on Vogan diagrams can be found in Knapp, {\em Lie groups: beyond and introduction}, Chapter VI.}\index{Vogan diagram}
The Vogan diagram associated to the triple $(\mathfrak{g}_o,\mathfrak{h}_o,\Delta(\mathfrak{g}_o,\mathfrak{h}_o)$ is a triple $(D,\theta,\text{marking in $D^\theta$})$ consisting of
\begin{itemize}
\item The Dynkin diagram $D$ corresponding to the set $\Delta(\mathfrak{g}_o,\mathfrak{h}_o)$ of simple roots.
\item An involution $\theta:D\to D$ of the Dynkin diagram (indicated by folding the diagram so that the nodes correspond).
\item A subset of nodes in the fixed point set $D^\theta$ that are marked if the corresponding simple root is noncompact.
\end{itemize}
\end{definition}

There could be many Vogan diagrams corresponding to a given simple orthogonal symmetric Lie algebra $(\mathfrak{g}_o,s_o)$.  However, it turns out (Proposition \ref{BoreldeSiebenthal}) that associated to any such Lie algebra there is a Vogan diagram having at most one marked node $\alpha_j$, and that this node has $d_j=\langle\lambda_j^\vee,\theta\rangle=1$ or $2$.  Hence we shall confine attention to this case.  We shall not prove here that every possible Vogan diagram gives a simple real Lie algebra, although this is true.  Rather we content ourselves with the fact that once we have listed all of those that are admissible, we shall have obtained the complete classification.

\subsection{Vogan diagrams, no involution}  We first look at the diagrams whose involution is trivial.  In each case, the compact form of the Lie algebra has no marked node.  If a node is marked, the subalgebra $\mathfrak{k}_o$ is the Lie algebra obtained by the following procedure.  Suppose that $\alpha_j$ is marked.  If $d_j=1$, then $\mathfrak{k}_o$ is the Lie algebra with one-dimensional center whose semisimple part is the compact form of the Lie algebra whose (possibly disconnected) Dynkin diagram is obtained by deleting $\alpha_j$.  If $d_j=2$, then $\mathfrak{k}_o$ is the compact form of the Lie algebra obtained by deleting $\alpha_j$ and including the affine root $-\theta$.

\begin{description}
\item[$A_n$]
The Dynkin diagram, with the highest root labeled, and showing the affine root as well, is
$$\xymatrix{
&&&\circ\ar@{--}[drr]\ar@{--}[dll]&&\\
&\osarray{1}{\circ}\ar@{-}[r]\ar@{}[d] & \osarray{1}{\circ}\ar@{}[d]\ar@{-}[r] & \osarray{1}{\circ}\ar@{}[d]\ar@{..}[r] &\osarray{1}{\circ}\ar@{-}[r]\ar@{}[d] & \osarray{1}{\circ}\ar@{}[d]\\
\mathfrak{g}_o:&\mathfrak{su}(1,n) & \mathfrak{su}(2,n-1) & \mathfrak{su}(2,n-2) & \mathfrak{su}(n-1,2) & \mathfrak{su}(n,1) &
}
$$
Also shown is the Lie algebra $\mathfrak{g}_o$ with the associated noncompact root.  The maximal compact subalgebra of $\mathfrak{g}_o$ with the corresponding to the marked root is, up to a center, the Lie algebra whose Dynkin diagram is complementary to the marked root.  Thus for $\mathfrak{g}_o=\mathfrak{su}(p,q)$, the isotropy algebra is $\mathfrak{k}_o=\mathfrak{s}(\mathfrak{u}(p)\oplus\mathfrak{u}(q))$.

We can also read off the compact forms.  These all have $\mathfrak{g}_o=\mathfrak{su}(n+1)$, and their isotropy algebra $\mathfrak{k}_o$ must come from the corresponding dual noncompact form.  Thus these are $\mathfrak{g}_o=\mathfrak{su}(n+1)$, $\mathfrak{k}_o=\mathfrak{s}(\mathfrak{u}(p)\oplus\mathfrak{u}(q))$ with $p+q=n+1$.  In each case the rank of the symmetric space is $\min(p,q)$.

The restricted root system and center are
\begin{align*}
\Phi(\mathfrak{g}_o,\mathfrak{a}_o) &= BC_p \qquad 0<p<q\qquad &|Z(G^{sc}_c)\cap A|=1\\
&=C_p \qquad p=q\qquad &|Z(G^{sc}_c)\cap A|=2.
\end{align*}

The corresponding simply connected forms of the symmetric spaces are all Hermitian.  Summarizing,

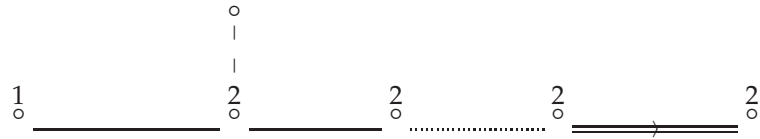
\begin{tabular}{cccc}
Type & Space & Note & Cartan type\\
&&&\\
Type III & $\displaystyle\frac{SU(p,q)}{S(U(p)\times U(q))}$ & $1<p\le q$ & AIII\\
&&&\\
 &  $\displaystyle\frac{SU(1,n)}{S(U(1)\times U(n))}$ & Hermitian hyperbolic space & AIV\\
&&&\\
Type I & $\displaystyle\frac{SU(n+1)}{S(U(p)\times U(q))}$ & $1<p\le q$ & AIII\\
&&&\\
 & $\displaystyle\frac{SU(n+1)}{S(U(1)\times U(n))}$ & Complex projective space & AIV
\end{tabular}

\vspace{24pt}

\item[$B_n$]
$$
\xymatrix{
&&\circ\ar@{--}[d]&&&\\
&\osarray{1}{\circ}\ar@{-}[r] & \osarray{2}{\circ}\ar@{-}[r] & \osarray{2}{\circ}\ar@{..}[r] &\osarray{2}{\circ}\ar@{=}[r]\ar@{}[r]|{\rangle} & \osarray{2}{\circ}\\
\mathfrak{g}_o: &\mathfrak{so}(2,2n-1) &\mathfrak{so}(4,2n-3)&\cdots&\mathfrak{so}(3,2n-2)&\mathfrak{so}(1,2n)
}
$$

The noncompact form has $\mathfrak{g}_o=\mathfrak{so}(p,q)$, $p+q=2n+1$, $p,q>0$, and $\mathfrak{k}_o=\mathfrak{so}(p)+\mathfrak{so}(q)$.  The compact form has $\mathfrak{g}_o=\mathfrak{so}(n+1)$ and $\mathfrak{k}_o=\mathfrak{so}(p)+\mathfrak{so}(q)$.  These both have rank $\min(p,q)$.  The restricted root system is
$$\Phi(\mathfrak{g}_o,\mathfrak{a}_o) = B_{\operatorname{rank}},\qquad |Z(B_c^{sc}\cap A| = 2.$$
The space is Hermitian if and only if $p=2$.

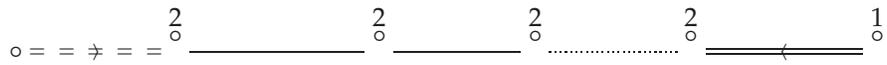
\begin{tabular}{cccc}
Type & Space & Note & Cartan type\\
&&&\\
Type III & $\displaystyle\frac{SO(p,q)}{SO(p)\times SO(q))}$ & $p+q=2n+1, p,q\ge 2$ & BI\\
&&&\\
 &  $\displaystyle\frac{SO(1,2n)}{SO(2n)}$ & Hyperbolic space & BII\\
&&&\\
Type I & $\displaystyle\frac{SO(2n+1)}{SO(p)\times SO(q))}$ & $+1$ multiply-connected form& BI\\
&&&\\
& $\displaystyle\frac{SO(2n+1)}{SO(2n)}$ & Sphere ($+1$ multiply-connected form $\mathbb{RP}^{2n}$) & BII
\end{tabular}

\item[$C_n$]
$$
\xymatrix{
\circ\ar@{==}[r]\ar@{}[r]|{\rangle}&\osarray{2}{\circ}\ar@{-}[r] & \osarray{2}{\circ}\ar@{-}[r] & \osarray{2}{\circ}\ar@{..}[r] &\osarray{2}{\circ}\ar@{=}[r]\ar@{}[r]|{\langle} & \osarray{1}{\circ}\\
\mathfrak{g}_o: &\mathfrak{sp}(1,n-1)&\mathfrak{sp}(2,n-2)&\cdots&\mathfrak{sp}(n-1,1)&\mathfrak{sp}(n,\mathbb{R})
}
$$
The noncompact forms are $\mathfrak{g}_o=\mathfrak{sp}(p,q)$ ($p+q=n, p,q\ge 1$) with $\mathfrak{k}_o=\mathfrak{sp}(p)+\mathfrak{sp}(q)$.  The corresponding compact form is $\mathfrak{g}_o=\mathfrak{sp}(n)$ and $\mathfrak{k}_o=\mathfrak{sp}(p)+\mathfrak{sp}(q)$.  The rank is $\min(p,q)$.

The exceptional case (when the last node is colored) is $\mathfrak{g}_o=\mathfrak{sp}(n,\mathbb{R})$ and $\mathfrak{k}_o=\mathfrak{u}(n)$, and the corresponding compact form $\mathfrak{g}_o=\mathfrak{sp}(n)$ and $\mathfrak{k}_o=\mathfrak{u}(n)$.  These two symmetric spaces have rank $n$.  The simply connected forms for these two symmetric spaces are Hermitian, but none of the others are.

The restricted roots and center are

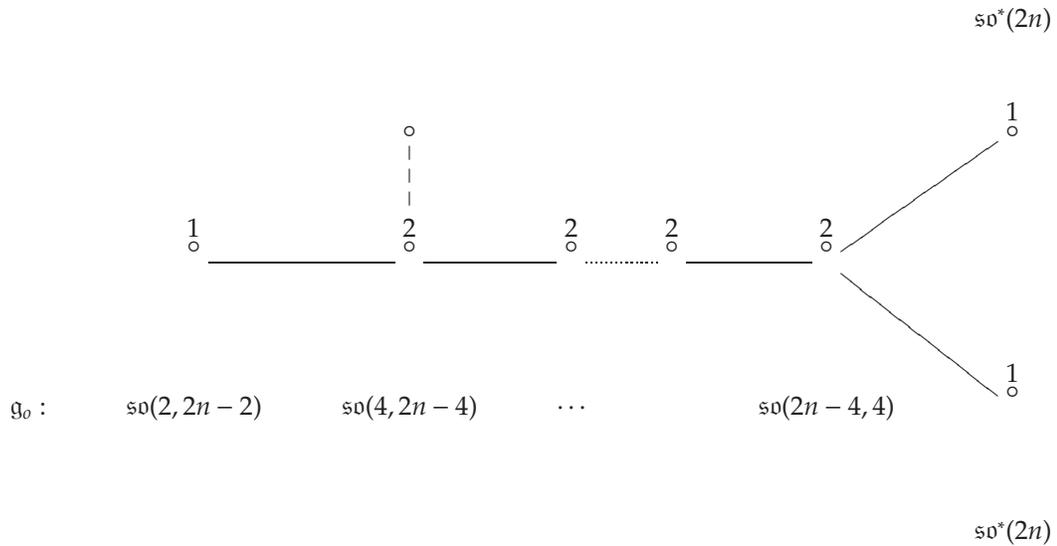
\begin{table}[htp]
\begin{tabular}{cccc}
$\displaystyle\Phi(\mathfrak{g}_o,\mathfrak{a}_o)$ & Note & Group & $|Z(G_c^{sc})\cap A|$\\
$BC_{\operatorname{rank}}$ & $p\not=q$ & $Sp(p,q)$ & 1\\
$C_p$ & $p=q$ & $Sp(p,q)$ & 2\\
$C_n$ & & $Sp(n,\mathbb{R})$ & 2
\end{tabular}
\end{table}

\vspace{24pt}

\begin{table}[htp]
\begin{tabular}{cccc}
Type & Space & Note & Cartan type\\
&&&\\
Type III & $\displaystyle\frac{Sp(p,q)}{Sp(p)\times Sp(q))}$ & $p+q=n, p,q\ge 2$& CII\\
&&&\\
& $\displaystyle\frac{Sp(1,n-1)}{Sp(1)\times Sp(n-1)}$ & Quaternionic hyperbolic space & CII\\
&&&\\
& $\displaystyle\frac{Sp(n,\mathbb{R})}{U(n)}$ & Siegel space & CI\\
&&&\\
Type I & $\displaystyle\frac{Sp(n)}{Sp(p)\times Sp(q))}$ &  $\begin{matrix}\text{Grassmannian of quaternionic $p$-planes in $\mathbb{H}^{p+1+1}$}\\p+q=n, p,q\ge 2\\\text{$+1$ multiply-connected form if $p=q$} \end{matrix}$& CII\\
&&&\\
& $\displaystyle\frac{Sp(1,n-1)}{Sp(1)\times Sp(n-1)}$ & $\mathbb{HP}^1$ & CI\\
&&&\\
& $\displaystyle\frac{Sp(n)}{U(n)}$ & $\begin{matrix}\text{Complex Lagrangian Grassmannian}\\\text{$+1$ multiply-connected form}\end{matrix}$ & CI
\end{tabular}
\end{table}

\item[$D_n$]

$$
\xymatrix{
&&&&&&\mathfrak{so}^*(2n)\\
&&\circ\ar@{--}[d]&&&&\overset{\displaystyle 1}{\circ}\\
&\osarray{1}{\circ}\ar@{-}[r]&\osarray{2}{\circ}\ar@{-}[r]&\osarray{2}{\circ}\ar@{..}[r]&\osarray{2}{\circ}\ar@{-}[r]&\osarray{2}{\circ}\ar@{-}[ur]\ar@{-}[dr]&\\
\mathfrak{g}_o:&\mathfrak{so}(2,2n-2)&\mathfrak{so}(4,2n-4)&\cdots&&\mathfrak{so}(2n-4,4)&\osarray{1}{\circ}\\
&&&&&&\mathfrak{so}^*(2n)
}
$$

The noncompact forms are $\mathfrak{g}_o=\mathfrak{so}(p,q), p+q=2n, p\ge 2,q\ge 4$ (even) with $\mathfrak{k}_o=\mathfrak{so}(p)+\mathfrak{so}(q)$ with rank $\min(p,q)$, and $\mathfrak{g}_o=\mathfrak{so}^*(2n)$ and $\mathfrak{k}_o=\mathfrak{u}(n)$ with rank $\lfloor n/2\rfloor$. (Recall: $\mathfrak{so}^*(2n)=\mathfrak{so}(n,\mathbb{H})$.)  The corresponding compact forms are $\mathfrak{g}_o=\mathfrak{so}(2n)$ and $\mathfrak{k}_o=\mathfrak{so}(p)+\mathfrak{so}(q)$ for $p+q=2n, p\ge 2,q\ge 4$, and $\mathfrak{k}_o=\mathfrak{u}(n)$.  Of these, the forms $(\mathfrak{so}(2,2n-2),\mathfrak{so}(2)+\mathfrak{so}(2n-2))$, $(\mathfrak{so}(2n),\mathfrak{so}(2)+\mathfrak{so}(2n-2))$, $(\mathfrak{so}^*(2n),\mathfrak{u}(n))$ and $(\mathfrak{so}(2n),\mathfrak{u}(n))$ carry a Hermitian structure.

The restricted roots and center are

\begin{table}[htp]
\begin{tabular}{cccc}
$\displaystyle\Phi(\mathfrak{g}_o,\mathfrak{a}_o)$ & Note & Group & $|Z(G_c^{sc})\cap A|$\\
$B_{\operatorname{rank}}$ & $p\not=q$ & $SO(p,q)$ & 2\\
$D_p$ & $p=q$ & $SO^*(n)$ & 4\\
$C_{n/2}$ & $n$ even & $SO^*(n)$ & 2\\
$BC_{(n-1)/2}$ & $n$ odd & $Sp(n,\mathbb{R})$ & 1
\end{tabular}
\end{table}

\vspace{24pt}

\begin{table}[htp]
\begin{tabular}{cccc}
Type & Space & Note & Cartan type\\
&&&\\
Type III & $\displaystyle\frac{SO(p,q)}{SO(p)\times SO(q))}$ & $p+q=2n, p,q\ge 2$ even& DI\\
&&&\\
& $\displaystyle\frac{SO^*(2n)}{U(n)}$ & & DIII\\
&&&\\
Type I & $\displaystyle\frac{SO(2n)}{SO(p)\times SO(q))}$ &  $\begin{matrix}p+q=2n, p,q\ge 2 (\text{even})\\\text{$+1$ multiply-connected form if $p\not=q$}\\\text{$+2$ forms if $p=q$} \end{matrix}$& DI\\
&&&\\
& $\displaystyle\frac{SO(2n)}{U(n)}$ &  & DIII\\
\end{tabular}
\end{table}
\end{description}

\subsection{Exceptional groups}
\begin{definition}\index{Index}
The {\em index} of an orthogonal symmetric Lie algebra $(\mathfrak{g}_o,s_o)$ is
$$-\tr s_o = \dim\mathfrak{p}_o-\dim\mathfrak{k}_o.$$
\end{definition}
The noncompact real forms of the exceptional Lie algebra are labeled by their index.

\begin{description}
\item[$E_6$]
$$
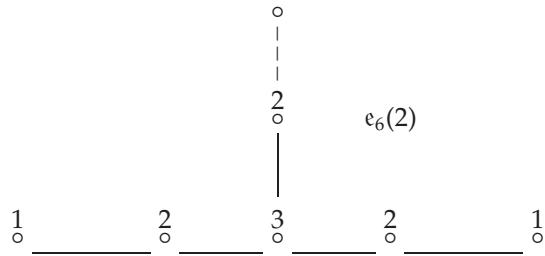
\xymatrix{
&&&\circ\ar@{--}[d]&&\\
&&&\overset{\displaystyle 2}{\circ}\ar@{-}[d]&\mathfrak{e}_6(2)&\\
&\osarray{1}{\circ}\ar@{-}[r] & \osarray{2}{\circ}\ar@{-}[r] & \osarray{3}{\circ}\ar@{-}[r] &\osarray{2}{\circ}\ar@{-}[r] & \osarray{1}{\circ}\\
\mathfrak{g}_o:&\mathfrak{e}_6(-14)&\mathfrak{e}_6(2)&&\mathfrak{e}_6(2)&\mathfrak{e}_6(-14)
}
$$

The compact forms have $\mathfrak{g}_o=\mathfrak{e}_6$ and $\mathfrak{k}_o=\mathfrak{so}(10)\oplus\mathfrak{u}(1)$ (rank $2$) or $\mathfrak{su}(6)\oplus\mathfrak{su}(2)$ (rank $4$); of these only the one with $\mathfrak{k}_o=\mathfrak{so}(10)\oplus\mathfrak{u}(1)$ is Hermitian.  The corresponding noncompact forms are $\mathfrak{g}_o=\mathfrak{e}_6(-14)$ with $\mathfrak{k}_o=\mathfrak{so}(10)\oplus\mathfrak{u}(1)$ (Hermitian, rank $2$) and $\mathfrak{g}_o=\mathfrak{e}_6(2)$ with $\mathfrak{k}_o=\mathfrak{su}(6)\oplus\mathfrak{su}(2)$ (not Hermitian, rank $4$).

The restricted roots and center are

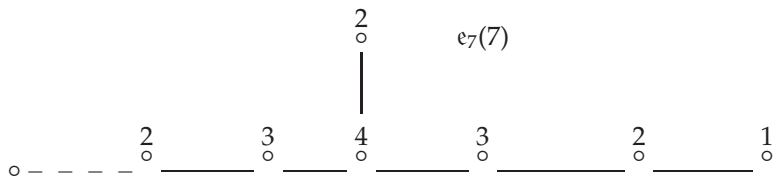
\begin{table}[htp]
\begin{tabular}{cccc}
$\displaystyle\Phi(\mathfrak{g}_o,\mathfrak{a}_o)$ & Group & $|Z(G_c^{sc})\cap A|$\\
$BC_2$ & $E_6(-14)$ & 1\\
$F_4$ & $E_6(2)$ & 1
\end{tabular}
\end{table}

\vspace{24pt}

\begin{table}[htp]
\begin{tabular}{ccc}
Type & Space & Cartan type\\
&&\\
Type III & $\displaystyle\frac{E_6(-14)}{\operatorname{Spin}(10)\times U(1)}$ & EIII\\
&&\\
& $\displaystyle\frac{E_6(2)}{SU(6)\times SU(2)}$ & EII\\
&&\\
Type I & $\displaystyle\frac{E_6}{\operatorname{Spin}(10)\times U(1)}$ &  EIII\\
&&\\
& $\displaystyle\frac{E_6}{SU(6)\times SU(2)}$ & EII
\end{tabular}
\end{table}

\item[$E_7$]
$$
\xymatrix{
&&&\overset{\displaystyle 2}{\circ}\ar@{-}[d]&\mathfrak{e}_7(7)&\\
\circ\ar@{--}[r]&\osarray{2}{\circ}\ar@{-}[r] & \osarray{3}{\circ}\ar@{-}[r] & \osarray{4}{\circ}\ar@{-}[r] &\osarray{3}{\circ}\ar@{-}[r] & \osarray{2}{\circ}\ar@{-}[r] &\osarray{1}{\circ}\\
\mathfrak{g}_o:&\mathfrak{e}_7(-5)&&&\mathfrak{e}_7(-5)&\mathfrak{e}_7(-25)
}
$$

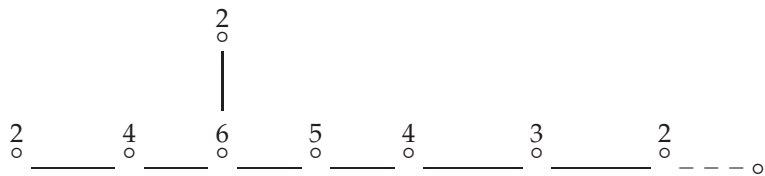
\commentx{Table missing $\Phi(\mathfrak{g}_o,\mathfrak{a}_o)$ and $|Z(G_c^{sc})\cap A|$}
\begin{table}[htp]
\begin{tabular}{cccc}
$\mathfrak{g}_o$ (noncompact) & $\mathfrak{k}_o$ & Note & Rank \\
$\mathfrak{e}_7(-5)$ & $\mathfrak{so}(12)\oplus\mathfrak{su}(2)$ && $4$ \\
$\mathfrak{e}_7(-25)$ & $\mathfrak{e}_6\oplus\mathfrak{u}(1)$ & Hermitian & $3$ \\
$\mathfrak{e}_7(7)$ & $\mathfrak{su}(8)$ && $7$ 
\end{tabular}
\end{table}

\begin{table}[htp]
\begin{tabular}{ccc}
Type & Space & Cartan type\\
&&\\
Type III & $\displaystyle\frac{E_7(-5)}{\operatorname{Spin}(12)\times SU(2)}$ & EVI\\
&&\\
& $\displaystyle\frac{E_7(-25)}{E_6\times U(1)}$ & EVII\\
&&\\
& $\displaystyle\frac{E_7(7)}{SU(8)}$ & EV\\
&&\\
Type I & $\displaystyle\frac{E_7}{\operatorname{Spin}(12)\times SU(2)}$ & EVI\\
&&\\
& $\displaystyle\frac{E_7}{E_6\times U(1)}$ & EVII\\
&&\\
& $\displaystyle\frac{E_7}{SU(8)}$ & EV\\
\end{tabular}
\end{table}

\item[$E_8$]
$$
\xymatrix{
&&&\overset{\displaystyle 2}{\circ}\ar@{-}[d]&&&&\\
&\osarray{2}{\circ}\ar@{-}[r] & \osarray{4}{\circ}\ar@{-}[r] & \osarray{6}{\circ}\ar@{-}[r] &\osarray{5}{\circ}\ar@{-}[r] & \osarray{4}{\circ}\ar@{-}[r] &\osarray{3}{\circ}\ar@{-}[r]&\osarray{2}{\circ}\ar@{--}[r]&\circ\\
\mathfrak{g}_o:&\mathfrak{e}_8(8)&&&&&\mathfrak{e}_8(-24)&
}
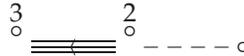
$$

\begin{table}[htp]
\begin{tabular}{ccccc}
$\mathfrak{g}_o$ (noncompact) & $\mathfrak{k}_o$ & Rank & $\Phi(\mathfrak{g}_o,\mathfrak{a}_o)$ & $|Z(G_c^{sc})\cap A|$\\
$\mathfrak{e}_8(8)$ & $\mathfrak{so}(16)$ & $8$ & $E_8$ & $1$ \\
$\mathfrak{e}_8(-24)$ & $\mathfrak{e}_7\oplus\mathfrak{su}(2)$ &$4$ & $F_4$ & $1$ 
\end{tabular}
\end{table}

\begin{table}[htp]
\begin{tabular}{ccc}
Type & Space & Cartan type\\
&&\\
Type III & $\displaystyle\frac{E_8(8)}{\operatorname{Spin}(16)}$ & EVIII\\
&&\\
& $\displaystyle\frac{E_8(-24)}{E_7\times SU(2)}$ & EIX\\
&&\\
Type I & $\displaystyle\frac{E_8}{\operatorname{Spin}(16)}$ & EVIII\\
&&\\
& $\displaystyle\frac{E_8}{E_7\times SU(2)}$ & EIX\\
\end{tabular}
\end{table}

\item[$F_4$]  
$$
\xymatrix{
&\osarray{2}{\circ}\ar@{-}[r] & \osarray{4}{\circ}\ar@{=}[r]\ar@{}[r]|{\langle} & \osarray{3}{\circ}\ar@{-}[r] & \osarray{2}{\circ}\ar@{--}[r]&\circ\\
\mathfrak{g}_o: &\mathfrak{f}_4(-20)&&&\mathfrak{f}_4(4)&
}
$$

\begin{table}[htp]
\begin{tabular}{ccccc}
$\mathfrak{g}_o$ (noncompact) & $\mathfrak{k}_o$ & Rank & $\Phi(\mathfrak{g}_o,\mathfrak{a}_o)$ & $|Z(G_c^{sc})\cap A|$\\
$\mathfrak{f}_4(4)$ & $\mathfrak{sp}(3)\oplus\mathfrak{su}(2)$ & $4$ & $F_4$ & $1$ \\
$\mathfrak{f}_4(-20)$ & $\mathfrak{so}(9)$ &$1$ & $BC_1$ & $1$ 
\end{tabular}
\end{table}

\begin{table}[htp]
\begin{tabular}{ccc}
Type & Space & Cartan type\\
&&\\
Type III & $\displaystyle\frac{F_4(4)}{\operatorname{Sp}(3)\times SU(2)}$ & FI\\
&&\\
& $\displaystyle\frac{F_4(-20)}{Spin(9)}$ & FII\\
&&\\
Type I & $\displaystyle\frac{F_4}{\operatorname{Sp}(3)\times SU(2)}$ & FI\\
&&\\
& $\displaystyle\frac{F_4}{Spin(9)}$ & FII\\
&&\\
\end{tabular}
\end{table}

\item[$G_2$]
$$
\xymatrix{
&\osarray{3}{\circ}\ar@3{-}[r]\ar@{}[r]|{\langle} & \osarray{2}{\circ}\ar@{--}[r]&\circ\\
\mathfrak{g}_o: &&\mathfrak{g}_2(2)&
}
$$

Here $\mathfrak{k}_o=\mathfrak{su}(2)+\mathfrak{su}(2)$, the rank is $2$, $\Phi(\mathfrak{g}_o,\mathfrak{a}_o)=G_2$, and $|Z(G^{sc}_c\cap A)|=1$.  The Type III symmetric space is $G_2(2)/SU(2)\times SU(2)$ and the Type I symmetric space is $G_2(2)/SU(2)\times SU(2)$.  Both are Cartan type $G$.
\end{description}

\subsection{Vogan diagrams with an involution}
The Cartan involution $\theta$ induces an automorphism of the Dynkin diagram, but not every such automorphism is admissible.  The set of fixed points $\mathfrak{k}$ is a reductive Lie algebra, and its Dynkin diagram must arise from that of $\mathfrak{g}$ by folding it along the involution.

\begin{description}
\item[$A_n$]
For $n$ even,
$$
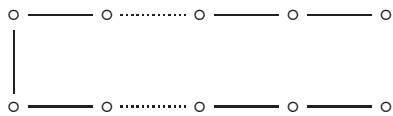
\xymatrix{
\circ\ar@{-}[d]\ar@{-}[r]& \circ\ar@{..}[r] &\circ\ar@{-}[r]&\circ\ar@{-}[r]&\circ\\
\circ\ar@{-}[r]& \circ\ar@{..}[r] &\circ\ar@{-}[r]&\circ\ar@{-}[r]&\circ
}$$
For $n$ odd,
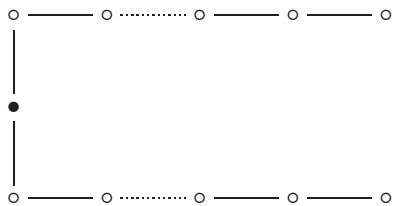
$$\xymatrix{
\circ\ar@{-}[d]\ar@{-}[r]& \circ\ar@{..}[r] &\circ\ar@{-}[r]&\circ\ar@{-}[r]&\circ\\
\bullet\ar@{-}[d] &&&&\\
\circ\ar@{-}[r]& \circ\ar@{..}[r] &\circ\ar@{-}[r]&\circ\ar@{-}[r]&\circ
}$$

Then $\mathfrak{g}_o=\mathfrak{sl}(n+1,\mathbb{R})$ in the noncompact case and $\mathfrak{g}_o=\mathfrak{su}(n+1)$ in the compact case.  In both cases $\mathfrak{k}_o=\mathfrak{so}(n+1)$.  The rank is $n$.  We have $\Phi(\mathfrak{g}_o,\mathfrak{a}_o)=A_n$ and $|Z(G_c^{sc})\cap A|=n+1$.

\begin{table}[htp]
\begin{tabular}{ccc}
Type & Space & Cartan type\\
&&\\
Type III & $\displaystyle\frac{SL(n+1,\mathbb{R})}{SO(n+1)}$ & AI\\
&&\\
Type I & $\displaystyle\frac{SL(n+1,\mathbb{R})}{SO(n+1)}$ & AI\\
\end{tabular}
\end{table}

There is a remaining form when $n$ is odd and there are no noncompact (painted) roots.  
$$
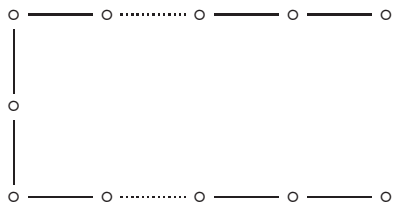
\xymatrix{
\circ\ar@{-}[d]\ar@{-}[r]& \circ\ar@{..}[r] &\circ\ar@{-}[r]&\circ\ar@{-}[r]&\circ\\
\circ\ar@{-}[d] &&&&\\
\circ\ar@{-}[r]& \circ\ar@{..}[r] &\circ\ar@{-}[r]&\circ\ar@{-}[r]&\circ
}$$

Then $\mathfrak{g}_o=\mathfrak{sl}\left(\frac{n+1}{2},\mathbb{H}\right)\cong \mathfrak{su}^*(n+1)$ in the noncompact case and $\mathfrak{g}_o=\mathfrak{su}(n+1)$ in the compact case.  In both cases $\mathfrak{k}_o=\mathfrak{sp}\left(\frac{n+1}{2}\right)$.  The rank is $\frac{n+1}{2}$.  We have $\Phi(\mathfrak{g}_o,\mathfrak{a}_o)=A_{(n+1)/2}$ and $|Z(G_c^{sc})\cap A|=\frac{n+1}{2}$.

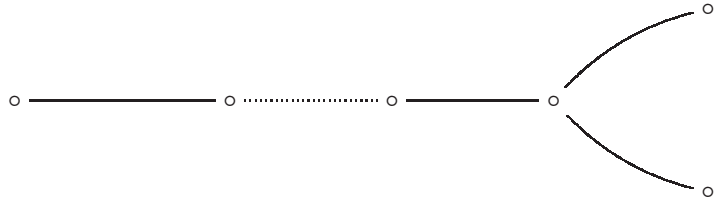
\begin{table}[htp]
\begin{tabular}{ccc}
Type & Space & Cartan type\\
&&\\
Type III & $\displaystyle\frac{SU^*(n+1)}{Sp\left(\frac{n+1}{2}\right)}$ & AII\\
&&\\
Type I & $\displaystyle\frac{SU(n+1)}{Sp\left(\frac{n+1}{2}\right)}$ & AII\\
\end{tabular}
\end{table}

\item[$D_n$]

$$\xymatrix{
&&&&&\circ \\
&\circ\ar@{-}[r]& \circ\ar@{..}[r] &\circ\ar@{-}[r]&\circ\ar@/^/@{-}[ur]\ar@/_/@{-}[dr]&\\
&&&&&\circ\\
\mathfrak{g}_o:&\mathfrak{so}(2n-1,1) & \mathfrak{so}(2n-3,3)  &\cdots&\mathfrak{so}(3,2n-3)&
}$$

We have $\mathfrak{g}_o=\mathfrak{so}(p,q)$ in the noncompact case and $\mathfrak{g}_o=\mathfrak{so}(2n)$ in the compact case, and $\mathfrak{k}_o=\mathfrak{so}(p)\oplus\mathfrak{so}(q)$.  Here $p,q\ge 1$ are both odd, $p+q=2n$.  The rank is $\min(p,q)$, $\Phi(\mathfrak{g}_o,\mathfrak{a}_o)=B_{rank}$, and $|Z(G_c^{sc})\cap A|=2$.

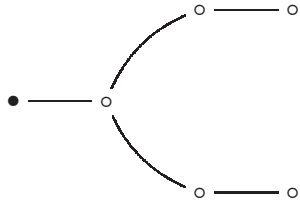
\begin{table}[htp]
\begin{tabular}{ccc}
Type & Space & Cartan type\\
&&\\
Type III & $\displaystyle\frac{SO(p,q)}{SO(p)\times SO(q)}$ & DI\\
&&\\
Type I & $\displaystyle\frac{SO(2n)}{SO(p)\times SO(q)}$ & DI
\end{tabular}
\end{table}

\item[$E_6$]

$$\xymatrix{
&&\circ\ar@{-}[r]&\circ \\
\bullet\ar@{-}[r]&\circ\ar@/^/@{-}[ur]\ar@/_/@{-}[dr]&\\
&&\circ\ar@{-}[r]&\circ
}$$

We have $\mathfrak{g}_o=\mathfrak{e}_6(6)$ in the noncompact case and $\mathfrak{g}_o=\mathfrak{e}_6$ in the compact case, and $\mathfrak{k}_o=\mathfrak{sp}(4)$.  The rank is $6$, $\Phi(\mathfrak{g}_o,\mathfrak{a}_o)=E_6$, and $|Z(G_c^{sc})\cap A|=3$.

\begin{table}[htp]
\begin{tabular}{cccc}
Type & Space & Note &Cartan type\\
&&&\\
Type III & $\displaystyle\frac{E_6(6)}{Sp(4)}$ & & EI\\
&&&\\
Type I & $\displaystyle\frac{E_6}{Sp(4)}$ & $+1$ multiply-connected form &EI
\end{tabular}
\end{table}

The remaining type has no noncompact roots:
$$
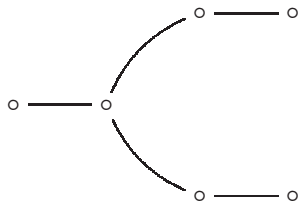
\xymatrix{
&&\circ\ar@{-}[r]&\circ \\
\circ\ar@{-}[r]&\circ\ar@/^/@{-}[ur]\ar@/_/@{-}[dr]&\\
&&\circ\ar@{-}[r]&\circ\\
}$$
Here $\mathfrak{g}_o=\mathfrak{e}_6(-26)$ in the noncompact case and $\mathfrak{g}_o=\mathfrak{e}_6$ in the compact case, and $\mathfrak{k}_o=\mathfrak{f}_4$.  The rank is $2$, $\Phi(\mathfrak{g}_o,\mathfrak{a}_o)=A_2$, and $|Z(G_c^{sc})\cap A|=3$.

\begin{table}[htp]
\begin{tabular}{cccc}
Type & Space & Note &Cartan type\\
&&&\\
Type III & $\displaystyle\frac{E_6(-26)}{F_4}$ & & EIV\\
&&&\\
Type I & $\displaystyle\frac{E_6}{F_4}$ & $+1$ multiply-connected form &EIV
\end{tabular}
\end{table}

\end{description}

\begin{proposition}[Borel and de Siebenthal theorem]\label{BoreldeSiebenthal}\index{Borel--de Sibenthal theorem}
Let $\mathfrak{g}_o$ be a real simple Lie algebra and $\mathfrak{h}_o$ a maximally compact Cartan subalgebra.  There exists a regular element such that the Vogan diagram relative to the simple root system $\Delta\subset\Phi(\mathfrak{g}_o,\mathfrak{h}_o)$ has at most one marked node.
\end{proposition}

\begin{proof}[Proof sketch]
Define a subspace $\mathcal{H}\subset\mathfrak{t}_o$ by
$$\mathcal{H}=\left(\bigcap_{\alpha\in\Phi_{im}(\mathfrak{g}_o,\mathfrak{h}_o)}\ker\alpha\right)^\perp.$$
Consider the subset $\Lambda$ of $H$ defined by
$$\Lambda=\left\{ H\in\mathcal{H}\mid 
\begin{matrix}
\alpha(H)\in\mathbb{Z}&\text{for all $\alpha\Delta_{im}(\mathfrak{g}_o,\mathfrak{h}_o)$}\\
\alpha(H)\in 2\mathbb{Z}+1&\text{for all $\alpha\Delta_{nc}(\mathfrak{g}_o,\mathfrak{h}_o)$}
\end{matrix} \right\}$$
Then, if there are any noncompact roots at all, then $\Lambda\not=0$.  Indeed, it contains the element
$$\sum_{\alpha_i\in\Delta_{nc}} \lambda_{\alpha_i}^\vee.$$
Let $H_o$ be an element of this set having smallest norm.  Choose a set of simple roots $\Delta^+$ for which $H_o$ is dominant.  The Vogan diagram with respect to $\Delta^+$ has at most one marked node.  Indeed, suppose that it contained an additional marked node $\omega\in\Delta^+$, then since $H_o$ is dominant with respect to $\Delta^+$,
$$\langle H_o-\omega,\omega\rangle > 0$$
and so $H_o-2\omega_i$ would be dominant with respect to $\Delta^+$, in $\Lambda$, and have smaller norm than $H_o$.  But this contradicts the minimality of $H_o$.
\end{proof}

\commentx{Left out section on Satake diagrams.}

\bibliography{symmetricspaces}{}
\bibliographystyle{plain}

\printindex

\end{document}